\documentclass[12pt,reqno]{amsart}
\usepackage{amsmath,amssymb,amsthm,amsfonts}
\usepackage{mathrsfs}
\usepackage{amsfonts}
\usepackage{a4wide}
\allowdisplaybreaks \numberwithin{equation}{section}
\usepackage{color}
\usepackage[colorlinks=true]{hyperref}

\newtheorem{theorem}{Theorem}[section]

\newtheorem{lemma}[theorem]{Lemma}

\newtheorem{thm}{Theorem}[section]

\newtheorem{lem}[thm]{Lemma}

\newtheorem{prop}[thm]{Proposition}

\theoremstyle{definition}

\newtheorem{remark}[theorem]{Remark}

\newcommand{\R}{\mathbb{R}}
\newcommand{\ds}{\displaystyle}

\begin{document}

\title[Solutions for a perturbed fractional Laplacian equation]
{ Large number of bubble solutions for a perturbed fractional Laplacian equation}

 \author{Chunhua Wang\,\,\,and\,\,\,Suting Wei}
\address{School of Mathematics and Statistics \& Hubei Key Laboratory of Mathematical Sciences, Central China Normal University, Wuhan, 430079, P. R. China }
\email{  chunhuawang@mail.ccnu.edu.cn}
\address{Department of Mathematics, South China
Agricultural University, Guangzhou,
510642, China }
\email{  stwei@scau.edu.cn}

 \begin{abstract}
This paper deals with the following nonlinear perturbed fractional Laplacian equation
$$
(-\Delta)^{s} u=K(|y'|,y'')u^{\frac{N+2s}{N-2s}\pm\epsilon},\,\,u>0,\,\,u\in
D^{1,s}(\R^{N}),
$$
where $0<s<1, N\geq 4,$ $(y',y'')\in \R^2\times\R^{N-2},$ $\epsilon>0$ is a small parameter and $K(y)$ is nonnegative and bounded.
By combining a finite reduction argument and local Pohozaev type of identities, we prove that if $N\geq 4,\max\{\frac{N+1-\sqrt{N^{2}-2N+9}}{4},\frac{3-\sqrt{N^{2}-6N+13}}{2}\}<s<1$ and
$K(r,y'')$  has a stable critical point $(r_0, y_0'')$ with
$r_0>0$ and $K(r_0, y_0'')>0,$ then the above problem has large number of bubble solutions if $\epsilon>0$ is small enough.
Also there exist solutions whose functional energy is in the order $\epsilon^{-\frac{N-2s-2}{(N-2s)^{2}}}$. Here, instead of estimating directly the derivatives of the reduced functional,
 we apply some local Pohozaev identities  to
locate the concentration points of the bubble solutions.
Moreover, the concentration points
 of the bubble solutions include a saddle point of $K(y).$

{ Key words }:  bubble solutions;  fractional Laplacian; Pohozaev identities; finite dimensional reduction method.

{ AMS Subject Classification }:35B05; 35B45.
 \end{abstract}

\maketitle

\section{ Introduction and main result }\label{s1}

In this paper, we are interested in the existence of large number of bubble solutions to the following perturbed fractional nonlinear
elliptic equation
\begin{equation}\label{equation}
(-\Delta)^{s} u=K(|y'|,y'')u^{\frac{N+2s}{N-2s}\pm\epsilon},\,\,u>0,\,\,u\in
D^{1,s}(\R^{N}),
\end{equation}
where $0<s<1$ for $N\geq 4,$ $(y',y'')\in \R^2\times\R^{N-2}$, $\epsilon>0$ is a small parameter and $K(y)$ is nonnegative and bounded.

For any $s\in (0,1)$, $(-\Delta)^s$ is the nonlocal operator defined as
\begin{equation*} \label{defofoperator}
\begin{split}
(-\Delta)^sf(y)=&\,C_{N,s}P.V.\int_{\R^N}\frac{f(y)-f(x)}{|x-y|^{N+2s}}dx
\\
=&\,C_{N,s}\lim_{\varepsilon\rightarrow
0^+}\int_{\R^N\setminus B_\varepsilon(y)}\frac{f(y)-f(x)}{|x-y|^{N+2s}}dx,
\end{split}
\end{equation*}
where P.V. stands for the principal value and
$$
C_{N,s}=\pi^{-(2s+\frac{N}{2})}\frac{\Gamma(\frac{N}{2}+s)}{\Gamma(-s)}.
$$
This operator is well defined in $C_{loc}^{1,1}(\R^N)\cap L_{s}(\R^N),$ where
$$
L_{s}(\R^N)\,=\,\Big\{u\in L^{1}_{loc}(\R^N),\int_{\R^{N}}\frac{|u(y)|}{1+|y|^{N+2s}}dy<+\infty\Big\}.
$$
For more details about the fractional Laplacian operator, we refer the readers to \cite{CLM-19, NPV-12} and the reference therein.

The fractional Laplacian operator appears in many areas including biological modeling, physics and
mathematical finances, and can be regarded as the infinitesimal generator of a stable Levy process (see for example\cite{A-09}).
From the view point of mathematics, an important feature of the fractional Laplacian operator is its nonlocal property, which makes it more challenge than the classical Laplacian operator.
This nonlocal operator in $\R^{N}$ can be expressed as a generalized Dirichiel-to-Neumann map for a certain elliptic boundary value problem with  local differential operators defined on the upper half-space $\R^{N}_{+}=\{(y,t):y\in \R^{N},t>0\}$, we also learn from \cite{cs}: given a solution $u=u(y)$ of $(-\Delta)^su(y)=0$ in $\R^{N}$, one can equivalently consider the dimensionally extended problem for $\tilde{u}=\tilde{u}(y,t)$ which solves
\begin{equation}
\left\{
\begin{array}{ll}
div(t^{1-2s}\nabla\tilde{u})=0,\qquad \hbox{in}\,\, \R^{N+1}_{+}
\\[1mm]
-\lim_{t\rightarrow 0}d_st^{1-2s}\partial_t
\tilde{u}(y,t)=(-\Delta)^su(y), \qquad \hbox{on}\,\, \partial\R^{N+1}_{+},
\end{array}
\right.
\end{equation}
where $d_s=2^{2s-1}\Gamma(s)/\Gamma(1-s)$ is a positive constant.
Here, $\tilde{u}(y,t)$ satisfies
\begin{equation}\label{theextensionofu}
\tilde{u}(y,t)=\mathcal{P}_{s}[u]:=\int_{\R^N} \mathcal{P}_{s}(y-\xi,t)u(\xi)d \xi,\quad (y,t)\in \R_{+}^{N+1}:=\R^N\times (0,\infty),
\end{equation}
where
\begin{equation}
\mathcal{P}_{s}(y,t)=\beta(N,s)\frac{t^{2s}}{(|y|^2+t^2)^{\frac{N+2s}{2}}}
\end{equation}
with constant $\beta(N,s)$ such that $\int_{\R^N}\mathcal{P}_{s}(y,t)d x=1$. Moreover, $\tilde{u}\in L^2(t^{1-2s},Q)$ for any compact set $Q$ in $\overline{\R_{+}^{N+1}} $ , $\nabla \tilde{u}  \in L^2(t^{1-2s},\R_{+}^{N+1})$ and $\tilde{u}\in C^{\infty}(\R_{+}^{N+1}).$ Moreover, $\tilde{u}$ satisfies (see [7])
\begin{equation}
\|\nabla\tilde{u}\|_{L^2(t^{1-2s},\R_{+}^{N+1})}=N_s\|u\|_{\dot{H}^{s}(\R^N)}.
\end{equation}
In recent years, fractional problems have been extensively investigated,
see for example \cite{BCDS-12-JDE,BCDS-13-PRSE,BCSS-15-AIPA,CT-10-AM,cs,FQT-12-PRSE,GN-16-DCDS,JLX-14-JEMS,NTW-19-JDE,T-11-CVPDE,TX-11-DCDS,YYY-15-JFA}
and the reference therein.

It is well-known that when $K(y)\equiv1$, the following functions
$$
U_{x,\lambda}(y)=(4^s\gamma)^{\frac{N-2s}{4s}}\Bigl(\frac{\lambda}{1+\lambda^{2}|y-x|^{2}}\Bigr)^{\frac{N-2s}{2}}
:=C_{N,s}\Bigl(\frac{\lambda}{1+\lambda^{2}|y-x|^{2}}\Bigr)^{\frac{N-2s}{2}},\,\,\,\lambda>0,\,\,
x\in\R^{N},
$$
where $\gamma=\frac{\Gamma(\frac{N+2s}{2})}{\Gamma(\frac{N-2s}{2})}$ are the unique (up to the translation and scaling) solutions for the problem
\begin{equation}\label{1.7}
(-\Delta)^{s} u=u^{\frac{N+2s}{N-2s}},\,\,\,u>0\,\,\,\text{in}\,\,\R^{N}.
\end{equation}

 When $s=1$, in \cite{WW-95-RCMP}, Wang and Wei constructed a single bubble solution to \eqref{equation} provided that $K(y)$
 has a non-degenerate critical point and $|K(y)|\leq C(1+|y|^{m}),m<2.$
 In \cite{L-16-S}, assuming that there exist positive constant $r_{0}>0,$ $ \theta$ and $\iota$ such that
 $$
 K(r)=K(r_{0})-C_{0}|r-r_{0}|^{m}+O(|r-r_{0}|^{m+\theta}),\,\,r\in (r_{0}-\iota,r_{0}+\iota), m\in [2,N-2),\
 $$
  Liu proved that \eqref{equation} has large number of bubble solutions if $\epsilon>0$ small enough.
  Moreover, he proved there exists solutions whose functional energy is in the order of $\epsilon^{-\frac{N-2-m}{(N-2)^{2}}}.$
 Motivated by \cite{ PWW-19-JDE} and \cite{LNW-07-CPAM}, we intend in  this paper to construct
 large number of bubbles to \eqref{equation} whose energy is very large if $\epsilon>0$ is small enough
 under more general assumptions on $K(y).$ Here we consider the case  $K(y)=K(|y'|,y'')=K(r,y''),$ where $y=(y',y'')\in\R^2\times\R^{N-2}$. We assume that $K(y)\geq 0$ is bounded and satisfies the following two conditions which are the same as that of \cite{PWW-19-JDE}:

{\bf $(K_{1})$}\,\,$K(r,y'')$ has a stable critical point
$(r_0, y_0'')$ in the following sense: $K(r,y'')$ has a critical point
$(r_0, y_0'')$ satisfying $r_0>0$  and  $K(r_0, y''_0)=1,$ and
$$
\deg (\nabla (K(r,y'')), ( r_0, y_0''))\ne 0.
$$

{\bf $(K_{2})$}\,\,$K(r,y'')\in C^{3}(B_{\vartheta}((r_0, y''_0))),$ where $\vartheta>0$ small and
\begin{equation}\label{q}
\Delta K(r_0,y_0''):=
\frac{\partial^{2}K(r_0,y_0'')}{\partial r^{2}}
+\sum_{i=3}^{N}\frac{\partial^{2}K(r_0,y_0'')}{\partial y^{2}_{i}}<0.
\end{equation}

Obviously, the assumption $(K_{2})$ contains a saddle point $(r_{0},y''_{0})$ of $K(y).$ In the sequel, we denote $y_{0}=(y'_{0},y''_{0})\in \R^{2}\times \R^{N-2},$
where $|y'_{0}|=r_{0}.$
Since  $(r_{0},y''_{0})$ is a critical point of $K(r,y'')$, $(K_{2})$ implies  $K(y)\in C^{3}(B_{\vartheta}( y_0))$ and
\begin{equation}\label{q1}
\Delta K(y_0)
:=\sum_{i=1}^{N}\frac{\partial^{2} K(y_0)}{\partial y^{2}_{i}}<0.
\end{equation}

\medskip

Our main result  is the following

 \begin{theorem}\label{thm1.1}
  Suppose that  $K\ge 0$ is bounded and satisfies $(K_{1})$ and $(K_{2}).$ If $N\geq 4,\max\{\frac{N+1-\sqrt{N^{2}-2N+9}}{4},\frac{3-\sqrt{N^{2}-6N+13}}{2}\}<s<1$,
 then there exists $\epsilon_{0}>0$ such that for $\epsilon\in (0,\epsilon_{0})$
problem~\eqref{equation} has  a solution $u_{\epsilon}$ whose number to the bubbles is of the order $\epsilon^{-\frac{N-2s-2}{(N-2s)^{2}}}$ as $\epsilon\rightarrow 0$.
Particularly, \eqref{equation} has large number of bubble solutions for small $\epsilon>0$
\end{theorem}

\begin{remark}\label{re-1}
$2s>\tau$ (defined in section \ref{s2}) and $0<s<1$ are equivalent to $\frac{N+1-\sqrt{N^{2}-2N+9}}{4}<s<1,$ which is needed in
\eqref{c-1} and \eqref{c-2} in Lemma \ref{lem2.5}.
$N>4+2\tau-2s$ and $0<s<1$ are equivalent to $\frac{3-\sqrt{N^{2}-6N+13}}{2}<s<1,$ which is needed in
\eqref{16-7-2} in Lemma \ref{lem2.5}.
Both of them are technical assumptions and they can be satisfied automatically when $s=1,$
but here we do not know how to get rid of them.
\end{remark}


\begin{remark}\label{r2}
 We conjecture that when $N=3=2+2s$, a similar result can also be obtained.
 Since in this case the number of bubbles in construction behaves as a logarithm of
 the bubbles' height, one have to make some corresponding modifications. For the details, the readers can refer to \cite{VS-19-ANA}.
\end{remark}

Now we outline the main idea in the
proof of Theorem \ref{thm1.1} and discuss the main difficulties in
the proof of such a result.

Throughout the remainder of this paper, we shall prove Theorem \ref{thm1.1}  in detail for the case $2^{*}_{s}-1+\epsilon$
since the case $2^{*}_{s}-1-\epsilon$ can be obtained by slightly modifying the arguments.
We use a Lyapunov-Schmidt
 reduction argument to prove Theorem \ref{thm1.2}. More precisely, we follow the method
 in \cite{PWW-19-JDE} to construct bubble solutions of problem
 \eqref{equation},  where the existence of infinitely many solutions
  for the
  prescribed scalar curvature problem is proved. In \cite{PWW-19-JDE}, there is no parameter
 appearing in their problem. Peng, Wang and Wei used $m$,
 the number of the bubbles of the solution, as the parameter to construct infinitely
 many positive bubble solutions.
 This idea was first introduced by Wei and Yan in \cite{wy}, which was applied
 to study other problems, such as \cite{cwy,DLY,GPY,LWX,WWY,WWY1,wy,WY1,WY2,WY4}. Unlike \cite{PWW-19-JDE}, in our proof, we use $\epsilon$
 as the parameter in the construction of bubble solutions, but the number of bubbles
 depends on the parameter $\epsilon.$ This is motivated by \cite{LNW-07-CPAM}, where they constructed multiple spikes to a singular perturbed
 problem and the number of spike depends on the small parameter.
 Such problems are considered in \cite{L-16-S,WW-95-RCMP}. Since $(r_0, y_0'')$
  may be a saddle point of $K(r, y''),$ we can not determine the location of the bubbles by using minimization or maximization procedure.
   Here we will use the Pohozaev identities to find algebraic equations which determine the location of the bubbles. We will discuss this in more details later. This idea was first introduced by Peng, Wang and Yan \cite{PWY-18-JFA},
which was used to deal with other problems such as \cite{GLN-19-CV,gln-19}. Moreover, the application of some local Pohozaev identities can
simplify many complicated and tedious computations which were involved in estimating directly the partial derivatives of the reduced functional such as \cite{Re, Re1}.

Define
\[
\begin{split}
 H_{s}=\Bigl\{&u:u\in D^{1,s}(\R^{N}),\;\;  u(y_1, -y_2,y'')=
u(y_1, y_2, y''),
 \\
 &
u(r\cos\theta,r\sin\theta,y'')=u\Bigl(r\cos\Bigl(\theta+\frac{2\pi j
}{m}\Bigl),r\sin\Bigl(\theta+\frac{2\pi j
}{m}\Bigl),y''\Bigl)\Bigl\}.
\end{split}
\]
Let
$$
x_{j}=\Bigl(\bar r \cos\frac{2(j-1)\pi}{m},\bar r
\sin\frac{2(j-1)\pi}{m},\bar{y}''\Bigr),\,\,\,j=1,\cdots,m,
$$
where $\bar{y}''$ is a vector in $\R^{N-2}.$

We will use $U_{x_{j},\lambda}$
as an approximate solution. Let
$\delta>0$ be a small constant, such that  $K(r, y'')>0$ if  $|(r,y'')- (r_0, y_0'')|\le
10\delta$.   

Denote
\begin{eqnarray*}
Z_{\bar r,\bar{y}'',\lambda}(y)=\sum_{j=1}^{m}U_{x_{j},\lambda}(y).
\end{eqnarray*}
By the weak symmetry of  $K(y)$, we observe that
$K(x_j)=K(\bar r,\bar{y}'')$,
$j=1,\cdots,m$. In this paper, we always assume that  ${m=[\epsilon^{-\frac{N-2s-2}{(N-2s)^{2}}}]>0}$ is a large integer,
$\lambda\in[L_{0}\epsilon^{-\frac{1}{N-2s}},L_{1}\epsilon^{-\frac{1}{N-2s}}]$ for
some constants $L_{1}>L_{0}>0$ and
\begin{equation}\label{r}
 |(\bar r,\bar{y}'')-(r_{0},y''_{0})|\leq \epsilon^{\frac{1+\iota}{N-2s}},
\end{equation}
where $\iota>0$ is a small constant.

\begin{remark}\label{re1}
Note that $\lim_{\epsilon\rightarrow 0}\lambda^{\frac{N-2s}{2}\epsilon}=c,$
where $c$ is some positive constat.
Then there exist $\epsilon_{0}>0$ and $C>0$ independent of $\epsilon$ such that for any
$\epsilon\in(0,\epsilon_{0}),$
\begin{equation}\label{la}
\lambda^{\frac{N-2s}{2}\epsilon}\leq C.
\end{equation}
\end{remark}

In order to prove Theorem \ref{thm1.1}, we will show the following
result.
\begin{theorem}\label{thm1.2}
Under the assumptions of Theorem~\ref{thm1.1},  there exists $\epsilon_{0}>0$, such that for any $\epsilon \in (0, \epsilon_{0})$,
problem \eqref{equation} has a solution $u_{\epsilon}$ of the form
\begin{equation}\label{sol}
u_{\epsilon}=Z_{\bar r_{\epsilon},\bar{y}_{\epsilon}'',\lambda_{\epsilon}}+\varphi_{\epsilon}=\sum_{j=1}^{m}
U_{x_{j},\lambda_{\epsilon}}+\varphi_{\epsilon},
\end{equation}
where $\varphi_{\epsilon}\in H_{s}$. Moreover,  as
$\epsilon\rightarrow 0$,
$\lambda_{\epsilon}\in[L_{0}\epsilon^{-\frac{1}{N-2s}},L_{1}\epsilon^{-\frac{1}{N-2s}}]$,
$ (\bar r_\epsilon,\bar{y}_{\epsilon}'')\to (r_{0},y''_{0}),$ and
$\lambda_{\epsilon}^{-\frac{N-2s}{2}}\|\varphi_{\epsilon}\|_{L^{\infty}}\rightarrow
0$.
\end{theorem}

Now we outline some of the main ideas in the proof of such a result.
 The functional corresponding to problem~\eqref{equation} is
\[
I(u)=\frac12 \int_{\R^N} |(-\Delta)^{\frac{s}{2}} u|^2 dy -\frac1{2^{*}_{s}+\epsilon}\int_{\R^N}K(y)(u)_{+}^{2^*_{s}+\epsilon} dy.
\]
Using the reduction argument, the problem of finding a critical point for $I(u)$ with the  form \eqref{sol} can be reduced to that of finding a critical point of the following function
$$
F(\bar r, \bar y'', \lambda):=I(Z_{\bar r,\bar{y}'',\lambda}+ \varphi_{\bar r,\bar{y}'',\lambda}),
$$
where $\lambda\in[L_{0}\epsilon^{-\frac{1}{N-2s}},L_{1}\epsilon^{\frac{1}{N-2s}}]$ for
some constants $L_{1}>L_{0}>0$ and $(\bar r,\bar{y}'')$ satisfies \eqref{r}.
Instead of estimating the derivatives of the reduced function $F(\bar r, \bar y'', \lambda)$ with respect
to $\bar{r}$ and $\bar{y}_{k},k=3,\cdots,N$ directly which
involve very complicated calculations, we turn to prove that if $(\bar{r},\bar{y}'',\lambda)$ satisfies the following local Pohozaev identities:
\begin{equation}
\begin{split}
&\int_{\partial''\mathcal{B}_{\rho}^{+}(y_0)}t^{1-2s}\frac{\partial\tilde{u}_\epsilon}{\partial \nu}\frac{\partial\tilde{u}_\epsilon}{\partial y_i} dS
-\frac{1}{2}\int_{\partial''\mathcal{B}_{\rho}^{+}(y_0)}t^{1-2s}\nabla|\tilde{u}_\epsilon|^2 \nu_idS
\\&+\frac{1}{2^{*}_{s}+\epsilon}\int_{\partial B_{\rho}(y_0)}K(y)(u_\epsilon)_{+}^{2^{*}_{s}+\epsilon}\nu_ids
-\frac{1}{2^{*}_{s}+\epsilon}\int_{B_{\rho}(y_0)}\frac{\partial K(y)}{\partial y_i}(u_\epsilon)_{+}^{2^{*}_{s}+\epsilon}dy
=0
\end{split}
\end{equation}
and
\begin{equation}
\begin{split}
&\int_{\partial''\mathcal{B}_{\rho}^{+}(y_0)} t^{1-2s}\langle Y,\nabla\tilde{u}_\epsilon \rangle\frac{\partial\tilde{u}_\epsilon}{\partial \nu}dS
-\frac{1}{2}\int_{\partial''\mathcal{B}_{\rho}^{+}(y_0)}t^{1-2s}|\nabla\tilde{u}_\epsilon|^2\langle Y,\nu\rangle dS
\\
&
+\frac{N-2s}{2}\int_{\partial''\mathcal{B}_{\rho}^{+}(y_0)}t^{1-2s}\tilde{u}_\epsilon\frac{\partial \tilde{u}_\epsilon}{\partial\nu} dS
+\frac{1}{2^{*}_{s}+\epsilon} \int_{\partial B_{\rho}(y_0)} K(r,y'')(u_\epsilon)_{+}^{2^{*}_{s}+\epsilon}\langle y,\nu\rangle ds
\\
&-\frac{N}{2^{*}_{s}+\epsilon}\int_{ B_{\rho}(y_0)}K(r,y'') (u_\epsilon)_{+}^{2^*_{s}+\epsilon}dy
-\frac{1}{2^{*}_{s}+\epsilon}
\int_{B_{\rho}(y_0)}\langle \nabla K,y\rangle(u_\epsilon)_{+}^{2^*_{s}+\epsilon} dy
=0,
\end{split}
\end{equation}
where $u_{\epsilon}=Z_{\bar{r},\bar{y}'',\lambda}+\varphi, \tilde{u}_{\epsilon}$
is the harmonic extension of $u_{\epsilon}$ (see \eqref{theextensionofu}) and satisfies  the following equation
\begin{equation}\label{e}
\left\{
\begin{array}{ll}
div(t^{1-2s}\nabla\tilde{u}(y,t))=0,\\[1mm]
-\lim_{t\rightarrow 0}t^{1-2s}\partial_t
\tilde{u}(y,t)=K(y)(u)_{+}^{2^{*}_{s}-1+\epsilon}.
\end{array}
\right.
\end{equation}
Moreover,
\begin{equation}
\begin{split}
\mathcal{B}_{\rho}^{+}(y_0)=\,&\{Y=(y,t):|Y-(y_0,0)|\leq  \rho \quad and \quad t>0\}\subset \R^{N+1}_{+},
\\
\partial'\mathcal{B}_{\rho}^{+}(y_0)=\,&\{Y=(y,t):|Y-(y_0,0)|\leq  \rho \quad and \quad t=0\}\subset \R^{N},
\\
\partial''\mathcal{B}_{\rho}^{+}(y_0)=\,&\{Y=(y,t):|Y-(y_0,0)|= \rho\quad and \quad t>0\}\subset\R^{N+1}_{+},
\\
&\partial B_{\rho}(y_0)=\,\{y:|y-y_0|= \rho \}\subset \R^{N},
\end{split}
\end{equation}
where $\rho$ is a small positive constant.

Due to the non-localness of the fractional Laplacian operator,
we can not built some local Pohozaev identities for problem \eqref{equation}. So we need to study the corresponding
harmonic extension problem \eqref{e}. Hence, we have to estimate this kind of integrals
which do not appear in the local problem. Here we use some similar arguments as \cite{gln-19}.

The rest of this paper is organized as follows. In section~\ref{s2}, we will
carry out the reduction procedure. Then, we will study the reduced finite
dimensional problem and prove Theorem~\ref{thm1.2} in
section~\ref{s3}. We put all the technical estimates in Appendices
\ref{sa}, \ref{s4}, \ref{sd} and \ref{sc}.

\section{Finite-dimensional reduction}\label{s2}

In this section, we perform a finite dimensional
reduction by using $Z_{\bar r,\bar{y}'',\lambda}(y)$ as the approximation solution and considering
the linearization of the problem \eqref{equation} around the approximation solution.
First, we introduce the following norms:
\begin{equation}\label{2.1}
\|u\|_{*}=\sup_{y\in\R^{N}}\Big(\sum_{j=1}^{m}\frac{1}{(1+\lambda|y-x_{j}|)^{\frac{N-2s}{2}+\tau}}\Big)^{-1}\lambda^{-\frac{N-2s}{2}}|u(y)|
\end{equation}
and
\begin{equation}\label{2.2}
\|f\|_{**}=\sup_{y\in\R^{N}}\Big(\sum_{j=1}^{m}\frac{1}{(1+\lambda|y-x_{j}|)^{\frac{N+2s}{2}+\tau}}\Big)^{-1}\lambda^{-\frac{N+2s}{2}}|f(y)|,
\end{equation}
where $\tau=\frac{N-2s-2}{N-2s}$.

Denote
\begin{eqnarray*}
Z_{j,1}=\frac{\partial U_{x_{j},\lambda}}{\partial \lambda},\,\,\,\,\,
Z_{j,2}=\frac{\partial U_{x_{j},\lambda}}{\partial \bar
r},\,\,\,\,\,Z_{j,k}=\frac{\partial U_{x_{j},\lambda}}{\partial \bar{y}''_{k}},\;\;
k=3,\cdots,N.
\end{eqnarray*}
We consider the following problem:
\begin{equation}\label{2.3}
 \left\{
 \begin{array}{ll}
   (-\Delta)^{s} \varphi
-(2^{*}_{s}-1+\epsilon) K(r, y'')Z_{\bar r,\bar{y}'',\lambda}^{2^{*}_{s}-2+\epsilon}\varphi=h
+\ds\sum_{l=1}^{N}c_{l}\ds\sum_{j=1}^{m}U_{x_{j},\lambda}^{2^{*}_{s}-2}Z_{j,l},\,\,\,\text{in}\,\,\,\,\R^{N}, \\[2.5mm]
    \varphi\in H_{s},\qquad\, \sum\limits_{j=1}^{m} \ds\int_{\R^N}Z_{x_{j},\lambda}^{2^{*}_{s}-2}Z_{j,l}\varphi=0,\;\qquad l=1,\cdots,N, \\[2.5mm]
 \end{array}
 \right.
\end{equation}
for some real numbers $c_{l}$.

 \begin{lem}\label{lem2.1}
 Suppose that $\varphi_{\epsilon}$ solves \eqref{2.3} for $h=h_{\epsilon}$. If $\|h_{\epsilon}\|_{**}$ goes to zero as $\epsilon$ goes to zero, so does $\|\varphi_{\epsilon}\|_{*}$.
 \end{lem}

 \begin{proof}
We follow the idea in \cite{WY1} and proceed the proof by contradiction. Suppose  that there exist $\epsilon\rightarrow 0,\,\bar r_\epsilon\to r_{0}, \bar{y}''_{\epsilon}\to  y''_{0}, \lambda_{\epsilon}\in
[L_{0}\epsilon^{-\frac{1}{N-2s}},L_{1}\epsilon^{-\frac{1}{N-2s}}]$ and
$\varphi_{\epsilon}$ solving problem \eqref{2.3} for
$h=h_{\epsilon},\lambda=\lambda_{\epsilon},\bar r=\bar r_{\epsilon},\bar{y}''=\bar{y}''_{\epsilon}$ with
$\|h_{\epsilon}\|_{**}\rightarrow 0$ and $\|\varphi_{\epsilon}\|_{*}\geq c>0.$ Without loss of generality, we
may assume that $\|\varphi_{\epsilon}\|_{*}=1$.

We have
\begin{equation}\label{2.4}
\begin{array}{ll}
|\varphi_\epsilon(y)|&\leq C\ds\int_{\R^N} \frac{1}{|y-z|^{N-2s}}Z_{\bar r,\bar{y}'',\lambda}^{2^{*}_{s}-2+\epsilon}(z)|\varphi_\epsilon(z)|dz\\
&\qquad+
C\ds\int_{\R^N} \frac{1}{|y-z|^{N-2s}}\Big( |h_{\epsilon}(z)|+\Big|\ds\sum_{l=1}^{N}c_{l}\ds\sum_{j=1}^{m}U_{x_{j},\lambda}^{2^{*}_{s}-2}(z)Z_{j,l}(z)\Big|\Big)dz,\\
&:=I_1+I_2+I_{3},
 \end{array}
\end{equation}
where $C$ defines some positive constant.

For the first term $I_1$, using Lemma \ref{lemb3}, we can prove
\begin{equation}\label{2.5}
\begin{array}{ll}
 I_1
&\leq C\|\varphi_\epsilon\|_{*}\,\ds\int_{\R^N} \frac{1}{|y-z|^{N-2s}}Z_{\bar r,\bar{y}'',\lambda}^{2^{*}_{s}-2+\epsilon}(z)\sum_{j=1}^{m}\frac{\lambda^{\frac{N-2s}{2}}}
{(1+\lambda|z-x_{j}|)^{\frac{N-2s}{2}+\tau}}dz
\\
&\leq C\|\varphi_\epsilon\|_{*}\lambda^{\frac{N-2s}{2}}
\ds\sum_{j=1}^{m}\frac{1}{(1+\lambda|y-x_{j}|)^{\frac{N-2s}{2}+\tau+\sigma}}.
 \end{array}
\end{equation}


For the second term $I_2$, applying Lemma \ref{lemb2}, we have
\begin{equation}
\begin{split}
I_2&\leq C\|h_{\epsilon}\|_{**}\ds\int_{\R^N} \sum\limits_{j=1}^{m}\frac{\lambda^{\frac{N+2s}{2}}}{|y-z|^{N-2s}(1+\lambda|z-x_{j}|)^{\frac{N+2s}{2}+\tau}}dz
\\
&=C\lambda^{\frac{N+2s}{2}}\|h_{\epsilon}\|_{**}\ds\sum\limits_{j=1}^{m}\int_{\R^N}\frac{\lambda^{N-2s}}{|\lambda(z-y)|^{N-2s}}\frac{1}{(1+\lambda|z-x_{j}|)^{2s+\frac{N-2s}{2}+\tau}}dz
\\
&=C\lambda^{\frac{N+2s}{2}}\|h_{\epsilon}\|_{**}\ds\sum\limits_{j=1}^{m}\int_{\R^N}\frac{\lambda^{N-2s}}{|\lambda(z-x_j)-\lambda(y-x_j))|^{N-2s}}\frac{1}{(1+\lambda|z-x_{j}|)^{2s+\frac{N-2s}{2}+\tau}}dz
\\
&=C\lambda^{\frac{N+2s}{2}}\lambda^{N-2s}\|h_{\epsilon}\|_{**}\ds\sum\limits_{j=1}^{m}\int_{\R^N}\frac{1}{|t-\lambda(y-x_j))|^{N-2s}}\frac{1}{(1+t)^{2s+\frac{N-2s}{2}+\tau}}\frac{1}{\lambda^{N}}dt
\\
& \leq C\|h_{\epsilon}\|_{**}\ds\sum\limits_{j=1}^{m}\frac{\lambda^{\frac{N-2s}{2}}}{(1+\lambda|y-x_{j}|)^{\frac{N-2s}{2}+\tau}}.
\end{split}
\end{equation}
Also, by Lemma \ref{lemb2} we have
\begin{align}
|I_{3}|
&\leq C\sum_{l=1}^{N}c_{l}\ds\sum_{j=1}^{m}\ds\int_{\R^N} \frac{1}{|y-z|^{N-2s}}U_{x_{j},\lambda}^{2_{s}^{*}-2}(z)\frac{\lambda^{\frac{N-2s}{2}+n_l}}{(1+\lambda|z-x_j|)^{N-2s}}dz \nonumber
\\
&\leq \lambda^{\frac{N-2s}{2}+n_l}C\sum_{l=1}^{N}c_{l}\ds\sum_{j=1}^{m}\ds\int_{\R^N} \frac{1}{|y-z|^{N-2s}}\frac{\lambda^{2s}}{(1+\lambda|z-x_j|)^{4s}}\frac{1}{(1+\lambda|z-x_j|)^{N-2s}}dz \nonumber
\\
&\leq \lambda^{\frac{N-2s}{2}+n_l}C\sum_{l=1}^{N}c_{l}\ds\sum_{j=1}^{m}\ds\int_{\R^N} \frac{1}{|y-z|^{N-2s}}\frac{\lambda^{2s}}{(1+\lambda|z-x_j|)^{N+2s}}dz
\\
&\leq C\sum\limits_{l=1}^{N}|c_{l}|
\ds\sum_{j=1}^{m}\frac{\lambda^{\frac{N-2s}{2}+n_{l}}}{(1+\lambda|y-x_{j}|)^{\frac{N-2s}{2}+\tau}}, \nonumber
\end{align}
where $n_{1}=-1$, $n_{l}=1$, $l=2,\cdots,N$.

Therefore, we have
\begin{equation}\label{normalofvarphi}
\begin{array}{ll}
 &\Big(\sum\limits_{j=1}^{m}\frac{\lambda^{\frac{N-2s}{2}}}{(1+\lambda|y-x_{j}|)^{\frac{N-2s}{2}+\tau}}\Big)^{-1}|\varphi_\epsilon(z)|\\[5mm]
 \leq&C\|\varphi_\epsilon\|_{*}\ds\frac{\sum\limits_{j=1}^{m}\frac{1}{(1+\lambda|y-x_{j}|)^{\frac{N-2}{2}+\tau+\sigma}}}
{\sum\limits_{j=1}^{m}\frac{1}{(1+\lambda|y-x_{j}|)^{\frac{N-2}{2}+\tau}}}+C\|h_{\epsilon}\|_{**}+C\sum\limits_{l=1}^{N}|c_{l}|\,\lambda^{n_l}.
 \end{array}
\end{equation}

Next, we will estimate $c_{l},l=1,\cdots,N.$ Multiplying both sides of \eqref{2.3} by $Z_{1,k}\,(k=1,\cdots,N)$ and integrating, we get
\begin{equation}\label{2.8}
\begin{split}
&\ds\sum_{l=1}^{N}c_{l}\sum_{j=1}^{m} \langle U_{x_{j},\lambda}^{2_{s}^{*}-2}Z_{j,l},Z_{1,k} \rangle \\
=&
\bigl\langle \ds(-\Delta)^{s} \varphi_{\epsilon}
-(2_{s}^{*}-1+\epsilon)K(r,y'')Z_{\bar r,\bar{y}'',\lambda}^{2_{s}^{*}-2+\epsilon}
\varphi_\epsilon,Z_{1,k}\big\rangle-\langle h_{\epsilon}, Z_{1,k}\rangle.
 \end{split}
 \end{equation}
First of all, there exists a  constant $\bar{c}>0$ such that
 \begin{equation}\label{2.14'}
 \sum_{j=1}^{m}\langle
U_{x_{j},\lambda}^{2_{s}^{*}-2}Z_{j,l},Z_{1,k}\rangle
= ( \bar{c}  +o(1)) \delta_{lk}\lambda^{n_l+n_k}.
 \end{equation}
From Lemma \ref{lemb1}, we obtain that
\begin{equation}\label{2.8'}
 \begin{split}
\big|\big\langle h_{\epsilon}(y), Z_{1,k}(y)\big\rangle\big|
\leq \,&C\|h_{\epsilon}\|_{**}\ds\int_{\R^N}\frac{\lambda^{\frac{N-2s}{2}+n_{k}}}{(1+\lambda|y-x_{1}|)^{N-2s}}
\sum_{j=1}^{m}\frac{\lambda^{\frac{N+2s}{2}}}{(1+\lambda|y-x_{j}|)^{\frac{N+2s}{2}+\tau}} \,dy\\
=  \,&C\, \lambda^{n_{k}}\|h_{\epsilon}\|_{**}\Big\{\int_{\R^N}\frac{\lambda^{N}}{(1+\lambda|y-x_{1}|)^{\frac{3N-2s}{2}+\tau}} \,dy
\\
&\qquad+\int_{\R^N}\frac{\lambda^{\frac{N-2s}{2}+n_{k}}}{(1+\lambda|y-x_{1}|)^{N-2s}}
\sum_{j=2}^{m}\frac{\lambda^{\frac{N+2s}{2}}}{(1+\lambda|y-x_{j}|)^{\frac{N+2s}{2}+\tau}} \,dy\Big\}
\\
\leq \, & C \lambda^{n_{k}}\|h_{\epsilon}\|_{**}\Big( C + C\sum_{j=2}^{m}\frac1{(\lambda|x_j-x_1|
)^\tau}\Big)
\\
\leq \,& C \lambda^{n_{k}}\|h_{\epsilon}\|_{**}.
 \end{split}
\end{equation}

On the other hand, direct calculation gives
\begin{equation}\label{2.9.5}
 \Big|\bigl\langle \ds(-\Delta)^{s} \varphi_\epsilon
-(2_{s}^{*}-1+\epsilon)K(r,y'')Z_{\bar r,\bar{y}'',\lambda}^{2_{s}^{*}-2+\epsilon}
\varphi_\epsilon,Z_{1,k}\bigr\rangle \Big|=O\Bigl(
\frac{\lambda^{n_{k}}\|\varphi_\epsilon\|_{*}}{{\lambda^{1+\iota}}}\Bigr),
 \end{equation}
whose proof we put in Appendix \ref{sd}.

 Combining \eqref{2.14'}, \eqref{2.8'} and \eqref{2.9.5}, we have
 \begin{equation}\label{2.16}
 |c_{l}|\leq C \frac{1}{\lambda^{n_{l}}}\Big(\frac{\|\varphi_\epsilon\|_{*}}{\lambda^{1+\iota}}+\|h_{\epsilon}\|_{**}\Big).
 \end{equation}
Then by \eqref{normalofvarphi} and $\|\varphi_\epsilon\|_{*}=1$,
that there is $R>0$ such that
 \begin{equation}\label{2.12}
 \|\lambda^{-\frac{N-2s}{2}}\varphi_\epsilon\|_{L^\infty(B_{R/\lambda}(x_{j}))}\geq a>0,
 \end{equation}
 for some $j$. But $\tilde{\varphi}_\epsilon(y)=\lambda^{-\frac{N-2s}{2}}\varphi_\epsilon(\frac{y}{\lambda}+x_{j})$ converges uniformly in any compact set to a solution $u$ of
\begin{equation}\label{2.13}
(-\Delta)^{s}  u-(2_{s}^{*}-1)U_{0,\Lambda}^{2_{s}^{*}-2}u =0,\,\,\,\text{in}\,\,\,
\R^{N},
 \end{equation}
 for some $\Lambda\in[\Lambda_{1},\Lambda_{2}]$ and $u$ is perpendicular to the kernel of \eqref{2.13}. So $u=0$. This is a contradiction to \eqref{2.12}.
 \end{proof}

From Lemma \ref{lem2.1}, applying the same argument as in the proof of Proposition 4.1 in \cite{dfm}, we can
prove the following result:
 \begin{lem}\label{lem2.2}
 There exist $\epsilon_{0}>0$ and a constant $C>0$ independent of $\epsilon$, such that for $\epsilon\in (0,\epsilon_{0})$ and
 all $h_{\epsilon}\in L^{\infty}(\R^{N})$, problem \eqref{2.3} has a unique solution $\varphi_\epsilon\equiv L_{\epsilon}(h_{\epsilon})$. Moreover,
 \begin{equation}\label{2.10}
 \|L_{\epsilon}(h_{\epsilon})\|_{*}\leq C\|h_{\epsilon}\|_{**},\,\,\,\,|c_{l}|\leq C \epsilon^{\frac{n_l}{N-2s}}\|h_{\epsilon}\|_{**}.
 \end{equation}
  \end{lem}

 Now we consider
 \begin{equation}\label{nonlinearequation}
 \left\{
 \begin{array}{ll}
&(-\Delta)^{s} (Z_{\bar r,\bar{y}'',\lambda}+\varphi)=
K(r,y'')(Z_{\bar r,\bar{y}'',\lambda}+\varphi)_+^{2^{*}_{s}-1+\epsilon}
+\sum_{l=1}^{N}c_{l}\sum_{j=1}^{m}U_{x_{j},\lambda}^{2^{*}_{s}-2}Z_{j,l},\text{in}\,\,\,\,\R^{N},
\\
&\varphi\in H_{s},\,\,\,\int_{\R^N}
\sum\limits_{j=1}^{m} U_{x_{j},\lambda}^{2^{*}_{s}-2}Z_{j,l}\varphi=0,\;l=1,2,\cdots,N.
 \end{array}
 \right.
\end{equation}
In the rest of this section, we devote ourselves to prove the following proposition by
using the contraction mapping theorem.

\begin{prop}\label{prop2.3}
There exist $\epsilon_{0}>0$ and a constant $C>0$, independent of $\epsilon$, such that for each $\epsilon\in (0,\epsilon_{0})$, $\lambda\in[L_{0}\epsilon^{-\frac{1}{N-2s}},L_{1}\epsilon^{-\frac{1}{N-2s}}]$, $\bar r\in[r_{0}-\theta,r_{0}+\theta],\bar{y}''\in
B_{\theta}(y''_{0})$, where $\theta>0$ small, \eqref{nonlinearequation} has a
unique solution $\varphi=\varphi_{\bar r,\bar{y}'',\lambda}\in
H_{s}$ satisfying
\begin{equation}\label{2.22}
\|\varphi\|_{*}\leq
C\epsilon^{\frac{1+\iota}{N-2s}},\,\,\,\,\, |c_{l}|\leq
C\epsilon^{\frac{1+n_l+\iota}{N-2s}},
\end{equation}
where $\iota>0$ is a small constant.
\end{prop}

We first rewrite \eqref{nonlinearequation} as
\begin{equation}\label{nonlinearequation1}
 \left\{
 \begin{array}{ll}
&(-\Delta)^{s} \varphi
-(2^{*}_{s}-1+\epsilon)K(r,y'')
Z_{\bar r,\bar{y}'',\lambda}^{2^{*}_{s}-2+\epsilon}
 \varphi
 \\[1mm]
 &\qquad \qquad  =N( \varphi)+l_{\epsilon}+\sum_{l=1}^{N}c_{l}\sum_{j=1}^{m}U_{x_{j},\lambda}^{2^{*}_{s}-2}Z_{j,l},\,\,\text{in}\,\,\,\,\R^{N},\\[1mm]
& \varphi\in H_{s},\,\,\,\ds\int_{\R^N}
\sum\limits_{j=1}^{m} U_{x_{j},\lambda}^{2^{*}_{s}-2}Z_{j,l} \varphi=0,\,\,\;l=1,\cdots,N,
 \end{array}
 \right.
\end{equation}
where
\[
N(\varphi) =K(r,y'')
\Big[\big(Z_{\bar r,\bar{y}'',\lambda}+\varphi\big)_+^{2^{*}_{s}-1+\epsilon}- Z_{\bar r,\bar{y}'',\lambda}^{2^{*}_{s}-1+\epsilon}-\big(2^{*}_{s}-1+\epsilon\big)
Z_{\bar r,\bar{y}'',\lambda}^{2^{*}_{s}-2+\epsilon}\varphi\Big],
\]
and
\begin{align}
l_{\epsilon}= \Big[K(r,y'') Z_{\bar r,\bar{y}'',\lambda }^{2^{*}_{s}-1+\epsilon}
-\sum_{j=1}^{m} U_{x_{j},\lambda}^{2^{*}_{s}-1}\Big].
\end{align}

In order to apply the contraction mapping theorem to prove Proposition \ref{prop2.3}, we need to estimate $N(\varphi)$ and
$l_{\epsilon}$ respectively.

\begin{lem}\label{lem2.4}
If ${N\geq 4}$, then
$$
||N(\varphi)||_{**}\leq C\|\varphi\|_{*}^{\min(2^{*}_{s}-1+\epsilon,2)}.
$$
\end{lem}

\begin{proof}
We have
\begin{align}
N(\varphi)\,\leq\,\left\{ \begin{array}{ll}
C\,|\varphi|^{2^{*}_{s}-1+\epsilon},\,& 2^{*}_{s}< 3,
\\[2mm]
|Z_{\bar r,\bar{y}'',\lambda}|^{2^{*}_{s}-3+\epsilon}\varphi^{2}+
C|\varphi|^{2^{*}_{s}-1+\epsilon},\,&2^{*}_{s}\geq 3.
\end{array} \right.
\end{align}

 First, we consider $2^{*}_{s}< 3.$
  By the discrete H\"{o}lder inequality,  we get
 \begin{align*}
 |N(\varphi)|\leq\,&C\|\varphi\|_{*}^{2^{*}_{s}-1+\epsilon}
 \Big(\sum_{j=1}^{m}\frac{\lambda^{\frac{N-2s}{2}}}{(1+\lambda|y-x_{j}|)^{\frac{N-2s}{2}+\tau}}\Big)^{2^{*}_{s}-1+\epsilon} \\
\leq\,&
 C\|\varphi\|_{*}^{2^{*}_{s}-1+\epsilon}\sum_{j=1}^{m}\frac{\lambda^{\frac{N+2s}{2}+\frac{N-2s}{2}\epsilon}}{(1+\lambda|y-x_{j}|)^{\frac{N+2s}{2}+\tau}}
 \\
 &\times\Big(\sum_{j=1}^{m}\frac{1}{(1+\lambda|y-x_{j}|)^{\frac{N+2s+(N-2s)\epsilon}{4s+(N-2s)\epsilon}(\frac{N-2s}{2}+\tau)-\frac{N-2s}{4s+(N-2s)\epsilon}(\frac{N+2s}{2}+\tau)}}\Big)^{\frac{4s}{N-2s}+\epsilon}\\
 \leq\,&  C\|\varphi\|_{*}^{2^{*}_{s}-1+\epsilon}\lambda^{\frac{N+2s}{2}}\sum_{j=1}^{m}\frac{1}{(1+\lambda|y-x_{j}|)^{\frac{N+2s}{2}+\tau}},
 \end{align*}
 and we have used Remark \ref{re1}.
 Therefore,
 $$
 \|N(\varphi)\|_{**}\leq C\|\varphi\|_{*}^{2^{*}_{s}-1+\epsilon}.
 $$

Similarly, if $2^{*}_{s}\geq 3$, we have
\begin{align*}
|N(\varphi)|\leq\,& C |Z_{\bar r,\bar{y}'',\lambda}|^{2^{*}_{s}-3+\epsilon}\varphi^{2}+
C|\varphi|^{2^{*}_{s}-1+\epsilon}\\
\leq\,&
 C\|\varphi\|_{*}^{2}\Big(\sum_{j=1}^{m}\frac{\lambda^{\frac{N-2s}{2}}}{(1+\lambda|y-x_{j}|)^{\frac{N-2s}{2}+\tau}}\Big)^2
 \Big(\sum_{j=1}^{m}\frac{\lambda^{\frac{N-2s}{2}}}{(1+\lambda|y-x_{j}|)^{N-2s}}\Big)^{2^{*}_{s}-3+\epsilon}\\
 &
 + C\|\varphi\|_{*}^{2^{*}_{s}-1+\epsilon}\Big(\sum_{j=1}^{m}\frac{\lambda^{\frac{N-2s}{2}}}{(1+\lambda|y-x_{j}|)^{\frac{N-2s}{2}+\tau}}\Big)^{2^{*}_{s}-1+\epsilon}
 \\
 \leq\,&
 C(\|\varphi\|_{*}^{2}+\|\varphi\|_{*}^{2^{*}_{s}-1+\epsilon})\Big(\sum_{j=1}^{m}\frac{\lambda^{\frac{N-2s}{2}}}{(1+\lambda|y-x_{j}|)^{\frac{N-2s}{2}+\tau}}\Big)^{2^{*}_{s}-1+\epsilon}
 \\
 \leq\,&
  C\|\varphi\|_{*}^{2}\sum_{j=1}^{m}\frac{\lambda^{\frac{N+2s}{2}}}{(1+\lambda|y-x_{j}|)^{\frac{N+2s}{2}+\tau}}.
\end{align*}
Hence, we obtain $||N(\varphi)||_{**}\leq C\|\varphi\|_{*}^{\min(2^{*}_{s}-1+\epsilon,2)}.$
\end{proof}

Next, we will give the estimate of $l_{\epsilon}.$
\begin{lem}\label{lem2.5}
If $N\geq 4,\max\{\frac{N+1-\sqrt{N^{2}-2N+9}}{4},\frac{3-\sqrt{N^{2}-6N+13}}{2}\}<s<1$,  then there is a small constant $\iota>0$, such that
$$
\|l_{\epsilon}\|_{**}\leq
C\epsilon^{\frac{1+\iota}{N-2s}}.
$$

\end{lem}

\begin{proof}
Recall that
\begin{align}\label{j1}
l_{\epsilon}=\,& K(y) Z_{\bar r,\bar{y}'',\lambda }^{2^{*}_{s}-1+\epsilon}
-\sum_{j=1}^{m} U_{x_{j},\lambda}^{2^{*}_{s}-1} \nonumber
\\
=\,&K(y)\Big( Z_{\bar r,\bar{y}'',\lambda }^{2^{*}_{s}-1+\epsilon}-Z_{\bar r,\bar{y}'',\lambda }^{2^{*}_{s}-1}\Big)
+K(y)\Big(Z_{\bar r,\bar{y}'',\lambda }^{2^{*}_{s}-1}-\sum_{j=1}^{m} U_{x_{j},\lambda}^{2^{*}_{s}-1}\Big)
 +\big(K(y)-1\big)\sum_{j=1}^{m} U_{x_{j},\lambda}^{2^{*}_{s}-1}
 \nonumber\\
:=\,&J_{1}\,+\, J_{2}\,+\,J_{3}.
\end{align}

By the mean value theorem and the discrete H\"{o}lder inequality, it follows from Remark \ref{re1} that
\begin{align}\label{j11}
 J_{1}&\leq C\epsilon Z_{\bar r,\bar{y}'',\lambda }^{2^{*}_{s}-1+\kappa\epsilon}|\ln Z_{\bar r,\bar{y}'',\lambda } | \nonumber
\\
&\leq C\epsilon \Big(\sum_{j=1}^{m}\frac{\lambda^{\frac{N-2s}{2}}}{(1+\lambda|y-x_{j}|)^{N-2s}}\Big)^{2^{*}_{s}-1+\kappa\epsilon}
\ln\sum_{j=1}^{m}\frac{\lambda^{\frac{N-2s}{2}}}{(1+\lambda|y-x_{j}|)^{N-2s}} \nonumber
\\
&\leq C\epsilon \ln\lambda^{\frac{N-2s}{2}}\sum_{j=1}^{m}\frac{\lambda^{\frac{N+2s}{2}+\frac{N-2s}{2}\kappa\epsilon}}{(1+\lambda|y-x_{j}|)^{\frac{N+2s}{2}+\tau}}
\Big(\sum_{j=1}^{m}\frac{1}{(1+\lambda|y-x_{j}|)^{(N-2s-\frac{(N+2s)\tau}{2(2^{*}_{s}-1+\kappa\epsilon)})\frac{2^{*}_{s}-1+\kappa\epsilon}{2^{*}_{s}-2+\kappa\epsilon}}}\Big)^{2^{*}_{s}-2+\kappa\epsilon} \nonumber
\\
&\leq
 C\epsilon \ln\frac{1}{\epsilon}\sum_{j=1}^{m}\frac{\lambda^{\frac{N+2s}{2}+\frac{N-2s}{2}\kappa\epsilon}}{(1+\lambda|y-x_{j}|)^{\frac{N+2s}{2}+\tau}}
\Big(\sum_{j=1}^{m}\frac{1}{(1+\lambda|y-x_{j}|)^{(N-2s-\frac{(N+2s)\tau}{2(2^{*}_{s}-1+\kappa\epsilon)})\frac{2^{*}_{s}-1+\kappa\epsilon}{2^{*}_{s}-2+\kappa\epsilon}}}\Big)^{2^{*}_{s}-2+\kappa\epsilon}
\nonumber
\\
&\leq C\epsilon \ln\frac{1}{\epsilon}\sum_{j=1}^{m}\frac{\lambda^{\frac{N+2s}{2}}}{(1+\lambda|y-x_{j}|)^{\frac{N+2s}{2}+\tau}} \nonumber
\\
&\leq C\epsilon^{\frac{1+\iota}{N-2s}} \sum_{j=1}^{m}\frac{\lambda^{\frac{N+2s}{2}}}{(1+\lambda|y-x_{j}|)^{\frac{N+2s}{2}+\tau}},
\end{align}
where $0<\kappa<1$ and we have used the fact that $N>2s+2$ to implies $\epsilon \ln\frac{1}{\epsilon}\leq C\epsilon^{\frac{1+\iota}{N-2s}}$.

In order to estimate $J_{12},$
first we define
$$
\Omega_{j}=\Big\{y:y=(y',y'')\in \R^{2}\times \R^{N-2},\Big\langle \frac{y'}{|y'|},\frac{x'_{j}}{|x'_{j}|}\Big\rangle\geq \cos\frac{\pi}{m}\Big\},\, j=1,\cdots,m.
$$
By symmetry, we can assume that $y\in\Omega_{1}$, then
$
|y-x_{j}|\geq|y-x_{1}|.$
Note that
\begin{align}\label{2.5.2}
& |J_{2}|\leq C\Big[\Big(\ds\sum_{j=2}^{m}
U_{x_{j},\lambda}\Big)^{2^{*}_{s}-1}+
U_{x_{1},\lambda}^{2^{*}_{s}-2}\ds\sum_{j=2}^{m}U_{x_{j},\lambda}+\ds\sum_{j=2}^{m}
U_{x_{j},\lambda}^{2^{*}_{s}-1} \Big]\nonumber\\
\leq &C\Big(\ds\sum_{j=2}^{m}\ds\frac{\lambda^{\frac{N-2s}{2}}}{(1+\lambda|y-x_{j}|)^{N-2s}}\Big)^{2^{*}_{s}-1}
+\ds\frac{C\lambda^{2s}}{(1+\lambda|y-x_{1}|)^{4s}}\ds\sum_{j=2}^{m}\frac{\lambda^{\frac{N-2s}{2}}}{(1+\lambda|y-x_{j}|)^{N-2s}}
.
 \end{align}

Note that  $\tau<2s$ implies that $\frac{2}{N-2s}\frac{N+2s}{4s}\big(\frac{N-2s}{2}-\tau\frac{N-2s}{N+2s}\big)\frac{4s}{N-2s}>1.$
As in \cite{WY1}, using H\"{o}lder inequality, we can derive
\begin{equation}\label{c-1}
\begin{split}
	&\Big(\ds\sum_{j=2}^{m}\frac{1}{(1+\lambda|y-x_{j}|)^{N-2s}}\Big)^{2^{*}_{s}-1}\\
	&\leq \sum_{j=2}^{m}\frac{1}{(1+\lambda|y-x_{j}|)^{\frac{N+2s}{2}+\tau}}
	\Big(\ds\sum_{j=2}^{m}\frac{1}{(1+\lambda|y-x_{j}|)^{\frac{N+2s}{4s}(\frac{N-2s}{2}-\tau\frac{N-2s}{N+2s})}}\Big)^{\frac{4s}{N-2s}}
	\\
	&\leq C\big(\frac{m}{\lambda}\big)^{\frac{N+2s}{4s}\big(\frac{N-2s}{2}-\tau\frac{N-2s}{N+2s}\big)\frac{4s}{N-2s}}
	\sum_{j=2}^{m}\frac{1}{(1+\lambda|y-x_{j}|)^{\frac{N+2s}{2}+\tau}}
	\\
	&\leq C\epsilon^{\frac{1+\iota}{N-2s}}\sum_{j=2}^{m}\frac{1}{(1+\lambda|y-x_{j}|)^{\frac{N+2s}{2}+\tau}}.
\end{split}
\end{equation}

By Lemma \ref{lemb1}, taking $0<\alpha\leq \min\{\frac{N+2s}{2},N-2s\},$ we obtain that for any $y\in\Omega_{1}$ and $j>1$
\begin{equation}\label{2.5.3}
\begin{array}{ll}
\ds\frac{1}{(1+\lambda|y-x_{1}|)^{4s}}\frac{1}{(1+\lambda|y-x_{j}|)^{N-2s}}
\leq
C\ds\frac{1}{(1+\lambda|y-x_{1}|)^{N+2s-\alpha}}\frac{1}{|\lambda
(x_{j}- x_{1})|^{\alpha}}.
 \end{array}
\end{equation}
Since $\tau<2s,$ we can choose $\alpha>\frac{N-2s}{2}$ satisfying $N+2s-\alpha\geq
\frac{N+2s}{2}+\tau.$

Then
\begin{equation}\label{c-2}
\begin{split}
&\ds\frac{1}{(1+\lambda|y-x_{1}|)^{4s}}\ds\sum_{j=2}^{m}\frac{1}{(1+\lambda|y-x_{j}|)^{N-2s}}\\
&\leq
\frac{C}{(1+\lambda|y-x_{1}|)^{N+2s-\alpha}}\Big(\ds\frac{m}{\lambda}\Big)^{\alpha}
\leq\frac{C}{(1+\lambda|y-x_{1}|)^{N+2s-\alpha}}\epsilon^{\frac{2\alpha}{(N-2s)^{2}}}
\\
&\leq\frac{C}{(1+\lambda|y-x_{1}|)^{\frac{N+2s}{2}+\tau}}\epsilon^{\frac{1+\iota}{N-2s}}.
\end{split}
\end{equation}

 Thus we have proved
\begin{equation}\label{2.5.4}
\|J_{2}\|_{**}\leq C\epsilon^{\frac{1+\iota}{N-2s}}.
\end{equation}

Next, we will estimate the term $J_{3}$. Using the Taylor expansion, in a neighborhood of $y_{0}$ we can rewrite $K(y)$ in the following form
\begin{equation*}
K(y)=K(y_0)+ \nabla K(y_0)\cdot(y-y_0)+ \frac{1}{2} \frac{\partial^2 K (y_0)}{\partial y_i \partial y_j}(y_{i}-y_{0i})(y_{j}-y_{0j}) +o(|y-y_0|^2).
\end{equation*}
In the region $\check{B}:=| (r,y'')-(r_0, y_0'')|\leq \sigma\epsilon^{\frac{\frac{1}{2}+\iota}{N-2s}},$
where $\sigma>0$ is a fixed constant.  Then we have
\begin{equation}\label{16-7-2-1}
\begin{split}
|J_{3}|
&=\Big|\Big(\ds\sum_{i,j=1}^{N}\frac{1}{2}\frac{\partial^{2} K(y_{0})}{\partial y_{i}\partial y_{j}} (y_{i}-y_{0i})(y_{j}-y_{0j})+o(|y-y_{0}|^{2})\Big)\sum_{j=1}^{m} U_{x_{j},\lambda}^{2_{s}^{*}-1}\Big|\\
&\leq C\epsilon^{\frac{1+2\iota}{N-2s}}
\ds\sum_{j=1}^{m}\frac{\lambda^{\frac{N+2s}{2}}}{(1+\lambda|y-x_{j}|)^{N+2s}}
\leq
C\epsilon^{\frac{1+\iota}{N-2s}}\ds\sum_{j=1}^{l}\frac{
\lambda^{\frac{N+2s}{2}}} {(1+\lambda|y-x_{j}|)^{\frac{N+2s}2+\tau}}.
\end{split}
\end{equation}

On the other hand, in the region $\sigma\epsilon^{\frac{\frac{1}{2}+\iota}{N-2s}}\leq| (r,y'')-(r_0, y_0'')|\leq 2\delta$, we have
$$
|y-x_j|\geq |(r,y'')-(r_0, y_0'')|-|(r_0, y_0'')-(\bar{r},\bar{y}'')|\geq \frac{\sigma}{2}\epsilon^{\frac{\frac{1}{2}+\iota}{N-2s}},
$$
which implies that
\begin{equation}\label{add1}
\frac 1{1+\lambda|y-x_j|}\leq C\epsilon^{\frac{\frac{1}{2}-\iota}{N-2s}}.
\end{equation}
Then we have
\begin{equation}\label{16-7-2}
\begin{split}
|J_{3}| &\leq C
\ds\sum_{j=1}^{m}\frac{\lambda^{\frac{N+2s}{2}}}{(1+\lambda|y-x_{j}|)^{N+2s}}\\
&=
C\epsilon^{\frac{1+\iota}{N-2s}}
\ds\sum_{j=1}^{m}\frac{
\lambda^{\frac{N+2s}{2}}}
{(1+\lambda|y-x_{j}|)^{\frac{N+2s}{2}+\tau}}
\frac{1}
{\epsilon^{\frac{1+\iota}{N-2s}}(1+\lambda|y-x_{j}|)^{\frac{N+2s}{2}-\tau}}
\\
&\leq
C\epsilon^{\frac{1+\iota}{N-2s}}\ds\sum_{j=1}^{m}\frac{
\lambda^{\frac{N+2s}{2}}} {(1+\lambda|y-x_{j}|)^{\frac{N+2s}2+\tau}}
\epsilon^{\frac{\frac{1}{2}-\iota}{N-2s}(\frac{N+2s}{2}-\tau)-\frac{1+\iota}{N-2s}}\\
&\leq
C\epsilon^{\frac{1+\iota}{N-2s}}\ds\sum_{j=1}^{m}\frac{
\lambda^{\frac{N+2s}{2}}} {(1+\lambda|y-x_{j}|)^{\frac{N+2s}2+\tau}},
\end{split}
\end{equation}
since $\frac{\frac{1}{2}-\iota}{N-2s}(\frac{N+2s}{2}-\tau)-\frac{1+\iota}{N-2s}\geq0.$

Combining \eqref{16-7-2-1} and \eqref{16-7-2}, we obtain that
\begin{equation}\label{2.5.5}
\|J_{3}\|_{**}\leq C\epsilon^{\frac{1+\iota}{N-2s}}.
\end{equation}

As a result, from \eqref{2.5.4} to \eqref{2.5.5}, we have
$$
\|l_{\epsilon}\|_{**}\leq
C\epsilon^{\frac{1+\iota}{N-2s}}.
$$
\end{proof}

\begin{proof}[\textbf{Proof of Proposition~\ref{prop2.3}}]
First we recall that $\lambda\in[L_{0}\epsilon^{-\frac{1}{N-2s}},L_{1}\epsilon^{-\frac{1}{N-2s}}].$
Set
$$
\mathcal {N}=\Big\{w:w\in
C(\R^{N})\cap H_{s},\,\|w\|_{*}\leq \epsilon^\frac{1}{N-2s}, \int_{\R^N} \sum_{j=1}^m U_{x_{j},\lambda}^{2^{*}_{s}-2}Z_{j,l} w=0\Big\},
$$
where $l=1,\cdots,N.$
Then \eqref{nonlinearequation1} is equivalent to
\begin{equation}\label{A}
\varphi=\mathcal
{A}(\varphi):=L_{\epsilon}(N(\varphi))+L_{\epsilon}(l_{\epsilon}),
\end{equation}
where $L_{\epsilon}$ is defined in Lemma \ref{lem2.2}.
We will prove that $\mathcal {A}$ is a contraction map from
$\mathcal {N}$ to $\mathcal {N}.$

In fact, since
\begin{equation}
\begin{split}
\|\mathcal{A}(\varphi)\|_{*}&\leq C\,\|L_{\epsilon}(N(\varphi))\|_*\,+\,C\,\|L_{\epsilon}(l_{\epsilon})\|_*\\
&
\leq
C\,\Big[\|N(\varphi)\|_{**}\,+\,\|l_{\epsilon}\|_{**}\Big]
\\
&\leq
C\,\Big[\|\varphi\|_{*}^{\min\{2^{*}_{s}-1+\epsilon,2\}}\,+\epsilon^{\frac{1+\iota}{N-2s}}\Big]
\\
&\leq
\epsilon^{\frac{1}{N-2s}},
\end{split}
\end{equation}
therefore
 $\mathcal {A}$ maps $\mathcal {N}$ to $\mathcal {N}.$

On the other hand, we have
\begin{align}
|N'(t)|\,\leq\,\left\{ \begin{array}{ll}
C\,|t|^{2^{*}_{s}-2+\epsilon},\,& 2^{*}_{s}< 3,
\\[2mm]
C\,\big(|Z_{\bar{r},\bar{y}'',\lambda}|^{2^{*}_{s}-3+\epsilon}|t|+|t|^{2^{*}_{s}-2+\epsilon}\big),\,&2^{*}_{s}\geq 3.
\end{array} \right.
\end{align}
 If $2^{*}_{s}< 3$, $\forall \varphi_{1},\, \varphi_{2} \in \mathcal {N}$, by the discrete H\"{o}lder inequality and Remark \ref{re1}, we have
\begin{equation*}
\begin{split}
&|N(\varphi_{1})-N(\varphi_{2})|\\
	&\leq C
	(|\varphi_{1}|^{2^{*}_{s}-2+\epsilon}+|\varphi_{2}|^{2^{*}_{s}-2+\epsilon})
	|\varphi_{1}-\varphi_{2}|\\
	&\leq C
	(||\varphi_{1}||_{*}^{2^{*}_{s}-2+\epsilon}+||\varphi_{2}||_{*}^{2^{*}_{s}-2+\epsilon})
	||\varphi_{1}-\varphi_{2}||_{*}\Big(\sum_{j=1}^{m}\frac{\lambda^{\frac{N-2s}{2}}}{(1+\lambda|y-x_{j}|)^{\frac{N-2s}{2}+\tau}}\Big)^{2^{*}_{s}-1+\epsilon}
	\\
	&\leq C
	(||\varphi_{1}||_{*}^{2^{*}_{s}-2+\epsilon}+||\varphi_{2}||_{*}^{2^{*}_{s}-2+\epsilon})
	||\varphi_{1}-\varphi_{2}||_{*}\sum_{j=1}^{m}\frac{\lambda^{\frac{N+2s}{2}}}{(1+\lambda|y-x_{j}|)^{\frac{N+2s}{2}+\tau}}.
	\end{split}
\end{equation*}
Hence
\begin{equation*}
	\begin{split}
\|\mathcal {A}(\varphi_{1})-\mathcal
{A}(\varphi_{2})\|_{*}&=\|L_{\epsilon}(N(\varphi_{1}))-L_{\epsilon}(N(\varphi_{2}))\|_{*}
\\
&\leq C\|N(\varphi_{1})-N(\varphi_{2})\|_{**}
\\
&\leq C\epsilon^{\frac{4s}{(N-2s)^{2}}}||\varphi_{1}-\varphi_{2}||_{*}
\\
&\leq \frac{1}{2}||\varphi_{1}-\varphi_{2}||_{*}.
\end{split}
\end{equation*}
The case $2^{*}_{s}\geq 3$ can be discussed in a similar way.

 Hence $\mathcal {A}$ is a contraction map.
Now by the contraction mapping theorem, there exists a unique
$\varphi=\varphi_{\bar r,\bar y'',\lambda}\in \mathcal {N}$ such that
\eqref{A} holds. Moreover, by Lemma \ref{lem2.2}, Lemma \ref{lem2.4} and Lemma \ref{lem2.5},  we deduce
\begin{equation*}
\|\varphi\|_{*}\leq\|L_{\epsilon}(N(\varphi))\|_{*}+\|L_{\epsilon}(l_{\epsilon})\|_{*}
\leq C\|N(\varphi)\|_{**}+C\|l_{\epsilon}\|_{**}\leq
C\epsilon^{\frac{1+\iota}{N-2s}}.
\end{equation*}
Moreover, we get the estimate of $c_{l}$ from \eqref{2.10}.

\end{proof}

\section{Proof of the main result}\label{s3}

Let
\[
I(u)=\frac12 \int_{\R^N} |(-\Delta)^{\frac{s}{2}} u|^2 dy  -\frac1{2_{s}^{*}+\epsilon}\int_{\R^N}K(|y'|,y'')(u)_{+}^{2_{s}^{*}+\epsilon} dy.
\]
In this section, we will choose suitable $(\bar r, \bar y'', \lambda)$ so
that $Z_{\bar r,\bar{y}'',\lambda}+ \varphi_{\bar r,\bar{y}'',\lambda}$ is a solution of problem \eqref{equation}.  For this purpose, we need the following result.

\begin{prop}\label{l1-18-1}
Suppose that  $(\bar r, \bar y'', \lambda)$ satisfies
\\
\begin{equation}\label{identityofscaling}
\begin{split}
&\int_{\partial''\mathcal{B}_{\rho}^{+}(y_0)} t^{1-2s}\langle Y,\nabla\tilde{u}_\epsilon\rangle \frac{\partial\tilde{u}_\epsilon}{\partial\nu} dS
-\frac{1}{2}\int_{\partial''\mathcal{B}_{\rho}^{+}(y_0)}t^{1-2s}|\nabla\tilde{u}_\epsilon|^2\langle Y,\nu\rangle dS
\\
&+\frac{N-2s}{2}\int_{\mathcal{B}_{\rho}^{+}(y_0)}t^{1-2s}|\nabla\tilde{u}_\epsilon|^{2}dydt
+\frac{1}{2^{*}_{s}+\epsilon} \int_{\partial B_{\rho}(y_0)} K(r,y'')(u_\epsilon)_{+}^{2^{*}_{s}+\epsilon}\langle y,\nu\rangle ds
\\
&
-\frac{N}{2^{*}_{s}+\epsilon}\int_{ B_{\rho}(y_0)}K(r,y'') (u_\epsilon)_{+}^{2^*_{s}+\epsilon}dy
-\frac{1}{2^{*}_{s}+\epsilon}
\int_{B_{\rho}(y_0)}\langle \nabla K,y\rangle(u_\epsilon)_{+}^{2^*_{s}+\epsilon} dy
=0,
\end{split}
\end{equation}
\begin{equation}\label{identityoftranslation}
\begin{split}
&\int_{\partial''\mathcal{B}_{\rho}^{+}(y_0)}t^{1-2s}\frac{\partial\tilde{u}_\epsilon}{\partial \nu}\frac{\partial\tilde{u}_\epsilon}{\partial y_i} dS
-\frac{1}{2}\int_{\partial''\mathcal{B}_{\rho}^{+}(y_0)}t^{1-2s}\nabla|\tilde{u}_\epsilon|^2 \nu_idS
\\
&
+\frac{1}{2^{*}_{s}+\epsilon}\int_{\partial B_{\rho}(y_0)}K(y)(u_\epsilon)_{+}^{2^{*}_{s}+\epsilon}\nu_i ds
-\frac{1}{2^{*}_{s}+\epsilon}\int_{B_{\rho}(y_0)}\frac{\partial K(y)}{\partial y_i}(u_\epsilon)_{+}^{2^{*}_{s}+\epsilon} dy
=0,
\end{split}
\end{equation}
and
\begin{equation}\label{integralofum}
\int_{\mathbb R^N}\Big[(-\Delta)^{s} u_\epsilon
-
K(r,y'')(u_\epsilon)_+^{2^{*}_{s}-1+\epsilon} \Big]\frac{\partial Z_{\bar r,\bar{y}'',\lambda}}{\partial \lambda}=0,
\end{equation}
where $u_\epsilon= Z_{\bar r,\bar{y}'',\lambda}+\varphi_{\bar r,\bar{y}'',\lambda}$  and $B_\rho(y_0)= \bigl\{ (r, y''):\;  |(r,y'')- (r_0, y_0'')|\le \rho\bigr\}   $
, $\rho$ is a small positive constant. Then $c_{l}=0$, $l\,=\,1, \cdots, N$.
\end{prop}

\begin{proof}
If \eqref{identityofscaling} and \eqref{identityoftranslation} hold, then it follows from \eqref{p1} and \eqref{p2} that
\begin{equation}\label{14-2-1}
\int_{B_{\rho}(y_0)}\sum_{l\,=\,1}^{N}c_{l}\ds\sum_{j\,=\,1}^{m}\,U_{x_{j},\lambda}^{2^{*}_{s}-2}\,Z_{j,l} \, \langle y, \nabla u_\epsilon\bigr\rangle dy
\,=\,0,
\end{equation}
and
\begin{equation}\label{14-2-2}
\int_{B_{\rho}(y_0)}\sum_{l\,=\,1}^{N}c_{l}\ds\sum_{j\,=\,1}^{m}\,U_{x_{j},\lambda}^{2^{*}_{s}-2}\,Z_{j,l} \, \, \frac{\partial u_\epsilon}{\partial y_i}dy
\,=\,0,\,(i=3,\cdots, N).
\end{equation}

By direct computations, we can check that
\begin{equation}\label{14-2-1-1}
\int_{B_{\rho}^{c}(y_0)}\sum_{l\,=\,1}^{N}c_{l}\ds\sum_{j\,=\,1}^{m}\,U_{x_{j},\lambda}^{2^{*}_{s}-2}\,Z_{j,l} \, \langle y, \nabla u_\epsilon\bigr\rangle dy
\,=\,o\bigl(m \lambda^{2}\bigr)\sum_{l\,=\,2}^N |c_l|+ o\,\big(m|c_1|\big),
\end{equation}
and
\begin{equation}\label{14-2-2-2}
\int_{B_{\rho}^{c}(y_0)}\sum_{l\,=\,1}^{N}c_{l}\ds\sum_{j\,=\,1}^{m}\,U_{x_{j},\lambda}^{2^{*}_{s}-2}\,Z_{j,l} \, \, \frac{\partial u_\epsilon}{\partial y_i} dy
\,=\,o\bigl(m \lambda^{2}\bigr)\sum_{l\,=\,2}^N |c_l|+ o\,\big(m|c_1|\big),\,(i=3,\cdots, N),
\end{equation}
whose proofs we put in Appendix \ref{sc}.

Also if \eqref{integralofum} holds, then by \eqref{nonlinearequation} we have
\begin{equation}\label{14-2-3}
\int_{\R^{N}}\sum_{l\,=\,1}^{N}c_{l}\ds\sum_{j\,=\,1}^{m}\,U_{x_{j},\lambda}^{2^{*}_{s}-2}\,Z_{j,l} \frac{\partial Z_{\bar r,\bar{y}'',\lambda}}{\partial \lambda}dy
\,=\,0.
\end{equation}

 From \eqref{14-2-1} to \eqref{14-2-3}, we get
\begin{equation}\label{14-2-1-again}
\sum_{l\,=\,1}^{N}c_{l}\ds\sum_{j\,=\,1}^{m}\int_{\R^N}\,U_{x_{j},\lambda}^{2^{*}_{s}-2}\,Z_{j,l} \, v
dy\,=\,o\bigl(m \lambda^{2}\bigr)\sum_{l\,=\,2}^N |c_l|+ o\,\big(m|c_1|\big)
\end{equation}
for $v\,=\, \bigl\langle y, \nabla u_\epsilon\bigr\rangle$, $v\,=\, \frac{\partial u_\epsilon}{\partial y_i},\,(i=3,\cdots, N)$
and $v\,=\, \frac{\partial Z_{\bar r,\bar{y}'',\lambda}}{\partial \lambda}.$

By direct computations, it is easy to obtain that
\begin{equation}\label{1-7-3}
\sum_{j\,=\,1}^{m}\int_{\R^N}\,U_{x_{j},\lambda}^{2^{*}_{s}-2}\,Z_{j,2} \bigl\langle y', \nabla_{y'} \,Z_{\bar r,\bar{y}'',\lambda}\bigr\rangle dy \,=\,  m( a_1+o(1))\lambda^{2},
\end{equation}

\begin{equation}\label{2-7-3}
\sum_{j\,=\,1}^{m}\int_{\R^N}\,U_{x_{j},\lambda}^{2^{*}_{s}-2}\,Z_{j,i}  \frac{\partial Z_{\bar r,\bar{y}'',\lambda}}{\partial y_i}dy\,=\, m (a_2+o(1))\,\lambda^{2}, \quad i\,=\,3,\cdots, N,
\end{equation}
and

\begin{equation}\label{3-7-3}
\sum_{j\,=\,1}^{m}\int_{\R^N}\,U_{x_{j},\lambda}^{2^{*}_{s}-2}\,Z_{j,1}  \frac{\partial Z_{\bar r,\bar{y}'',\lambda}}{\partial \lambda}dy\,=\,  \frac{m}{\lambda^{2}}(a_3+o(1)),
\end{equation}
for some constants $a_1\ne 0$, $a_2\ne 0$ and $a_3>0$.

For any functions  $f,\,\varphi\in H^1(\R^N)$, we have
\begin{equation}\label{20-18-1}
\int_{\R^N}  f \,\frac{ \partial \varphi}{\partial y_i} dy\,=\,-\int_{\R^N}   \varphi \,\frac{ \partial f}{\partial y_i} dy.
\end{equation}
Using \eqref{20-18-1}, we can prove from \eqref{2.22} that

\begin{equation}\label{4-7-3}
\sum_{l\,=\,1}^{N}c_{l}\ds\sum_{j\,=\,1}^{m}\int_{\R^N}\,U_{x_{j},\lambda}^{2^{*}_{s}-2}\,Z_{j,l} \,vdy
\,=\,o\bigl(m \lambda^{2}\bigr)\sum_{l\,=\,2}^N |c_l|+ o\,\big(m|c_1|\big),
\end{equation}
holds for  $v\,=\, \bigl\langle y,\nabla  \varphi_{\bar r,\bar{y}'',\lambda}\bigr\rangle$  and  $v\,=\, \frac{\partial \varphi_{\bar r,\bar{y}'',\lambda}}{\partial y_i}$.
Therefore,  from \eqref{14-2-1}, we obtain

\begin{equation}\label{21-2-1}
\sum_{l\,=\,1}^{N}c_{l}\ds\sum_{j\,=\,1}^{m}\int_{\R^N}\,U_{x_{j},\lambda}^{2^{*}_{s}-2}\,Z_{j,l}   v dy
\,=\,o\bigl(m \lambda^{2}\bigr)\sum_{l\,=\,2}^N |c_l|+ o\,\big(m|c_1|\big),
\end{equation}
holds for  $v\,=\, \bigl\langle y,\nabla  Z_{\bar r,\bar{y}'',\lambda}\bigr\rangle$  and  $v\,=\, \frac{\partial Z_{\bar r,\bar{y}'',\lambda}}{\partial y_i}$.

 From
\[
\bigl\langle y, \nabla Z_{\bar r,\bar{y}'',\lambda}\bigr\rangle\,=\,\bigl\langle y', \nabla_{y'} Z_{\bar r,\bar{y}'',\lambda}\bigr\rangle + \bigl\langle y'', \nabla_{y''} Z_{\bar r,\bar{y}'',\lambda}\bigr\rangle,
\]
 we find
\begin{equation}\label{30-18-1}
\begin{split}
&\sum_{l\,=\,1}^{N}c_{l}\ds\sum_{j\,=\,1}^{m}\int_{\R^N}U_{x_{j},\lambda}^{2^{*}_{s}-2}\,Z_{j,l}  \bigl\langle y, \nabla Z_{\bar r,\bar{y}'',\lambda}\bigr\rangle dy\\
\,=\,&
c_2\sum_{j\,=\,1}^{m}\int_{\R^N}\,U_{x_{j},\lambda}^{2^{*}_{s}-2}\,Z_{j,2}  \bigl\langle y', \nabla_{y'}\, Z_{\bar r,\bar{y}'',\lambda}\bigr\rangle dy  +o\,\big(m\lambda^{2} \big) \sum_{t\,=\,3}^N |c_t|+
o\big(m |c_1|\big),
\end{split}
\end{equation}
and

\begin{equation}\label{31-18-1}
\begin{split}
&\sum_{l\,=\,1}^{N}c_{l}\ds\sum_{j\,=\,1}^{m}\int_{\R^N}U_{x_{j},\lambda}^{2^{*}_{s}-2}Z_{j,l} \frac{\partial Z_{\bar r,\bar{y}'',\lambda}}{\partial y_i} dy\\
\,=\,&
c_i\sum_{j\,=\,1}^{m}\int_{\R^N}U_{x_{j},\lambda}^{2^{*}_{s}-2}Z_{j,i}  \frac{\partial Z_{\bar r,\bar{y}'',\lambda}}{\partial y_i}dy +o(m \lambda^{2}) \sum_{t\ne 1, \;i} |c_t|+
o(m |c_1|),\quad i\,=\,3,\cdots, N.
\end{split}
\end{equation}

Combining \eqref{21-2-1}, \eqref{30-18-1} and \eqref{31-18-1}, we are led to

\begin{equation}\label{30-7-3}
c_2\,\sum_{j\,=\,1}^{m}\int_{\R^N}U_{x_{j},\lambda}^{2^{*}_{s}-2}\,Z_{j,2}  \bigl\langle y', \nabla_{y'} Z_{\bar r,\bar{y}'',\lambda}\bigr\rangle dy \,=\,o\big(m\lambda^{2}\big ) \sum_{t\,=\,3}^N |c_t|+
o\big(m |c_1|\big),
\end{equation}
and

\begin{equation}\label{31-7-3}
c_i\,\sum_{j\,=\,1}^{m}\int_{\R^N}U_{x_{j},\lambda}^{2^{*}_{s}-2}\,Z_{j,i}  \frac{\partial Z_{\bar r,\bar{y}'',\lambda}}{\partial y_i} dy\,=\,o\big(m \lambda^{2}\big) \sum_{t\ne 1, \;i} |c_t|+
o\big(m |c_1|\big),\quad i\,=\,3,\cdots, N,
\end{equation}
which, together with \eqref{1-7-3} and \eqref{2-7-3}, imply

\begin{equation}\label{equationofci}
c_i\,=\, o\Bigl(\frac1{\lambda^{2}}\Bigr)\,c_1 , \quad i\,=\,2, \cdots, N.
\end{equation}

Now we have
\begin{equation}\label{32-18-1}
\begin{split}
0\,=\,&\sum_{l\,=\,1}^{N}c_{l}\ds\sum_{j\,=\,1}^{m}\int_{\R^N}U_{x_{j},\lambda}^{2^{*}_{s}-2}Z_{j,l} \frac{\partial Z_{\bar r,\bar{y}'',\lambda}}{\partial \lambda} dy\\
\,=\,&
c_1\sum_{j\,=\,1}^{m}\int_{\R^N}U_{x_{j},\lambda}^{2^{*}_{s}-2}Z_{j,1}  \frac{\partial Z_{\bar r,\bar{y}'',\lambda}}{\partial \lambda }  dy\\
\,=\,&m \frac{a_3 }{\lambda^{2}}c_1+o\bigl(\frac m{\lambda^2}\bigr) c_1.
\end{split}
\end{equation}
 So $c_1\,=\,0$. We also have $c_l=0$.
\end{proof}

Next, we will estimate \eqref{identityofscaling}, \eqref{identityoftranslation} and \eqref{integralofum}. Moreover, we will choose suitable $(\bar r, \bar y'', \lambda)$ so
that \eqref{identityofscaling}, \eqref{identityoftranslation} and \eqref{integralofum} holds. First of all, the following lemma gives the estimate of \eqref{integralofum}.

\begin{lem}\label{lem3.2}
We have
\begin{equation}\label{4-18-1}
\begin{split}
&\int_{\mathbb R^N}\bigl((-\Delta)^{s} u_\epsilon
-
K(r,y'')(u_\epsilon)_+^{2^{*}_{s}-1+\epsilon} \bigr)\frac{\partial Z_{\bar r,\bar{y}'',\lambda}}{\partial \lambda} dy
\\
\,=\,&m\Big(-\frac{B_{1}}{\lambda^{3}}
+\sum_{j=2}^{m}\frac{B_{2}}{\lambda^{N-2s+1}|x_{1}-x_{j}|^{N-2s}}
+O\big(\epsilon^{\frac{3+\iota}{N-2s}}\big)\Big) \\
\,=\,&m\Big(-\frac{B_{1}}{\lambda^{3}}+\frac{B_3 m^{N-2s}}{\lambda^{N-2s+1}}
+O\big(\epsilon^{\frac{3+\iota}{N-2s}}\big)\Big),
\end{split}
\end{equation}
where  $B_{i}>0$, $i\,=\,1, 2, 3$.
\end{lem}

\begin{proof}

We have
\begin{align}\label{10-15-5}
&\ds\int_{\mathbb R^N}\bigl((-\Delta)^{s} u_\epsilon
-K(r,y'')
(u_\epsilon)_+^{2^{*}_{s}-1+\epsilon} \bigr)\frac{\partial Z_{\bar r,\bar{y}'',\lambda}}{\partial \lambda}dy \nonumber\\[1mm]
\,=\,&
\Big\langle I'(Z_{\bar r,\bar{y}'',\lambda}),\ds\frac{\partial
Z_{\bar r,\bar{y}'',\lambda}}{\partial \lambda}\Big\rangle
+m \Big \langle  \ds(-\Delta)^{s} \varphi
-(2^{*}_{s}-1+\epsilon)K(r,y'')Z_{\bar r,\bar{y}'',\lambda}^{2^{*}_{s}-2+\epsilon}
\varphi,\ds\frac{\partial
Z_{x_1,\lambda}}{\partial \lambda}\Bigr\rangle \nonumber
\\[1mm]
&-\ds\int_{\R^N} K(r,y'')\bigl( (Z_{\bar r,\bar{y}'',\lambda}+\varphi)_+^{2^{*}_{s}-1+\epsilon} -Z_{\bar r,\bar{y}'',\lambda}^{2^{*}_{s}-1+\epsilon}-(2^{*}_{s}-1+\epsilon) Z_{\bar r,\bar{y}'',\lambda}^{2^{*}_{s}-2+\epsilon}\varphi\bigr) \frac{\partial
Z_{\bar r,\bar{y}'',\lambda}}{\partial\lambda} dy \nonumber
\\
:\,=\,&\Big\langle I'(Z_{\bar r,\bar{y}'',\lambda}),\ds\frac{\partial
Z_{\bar r,\bar{y}'',\lambda}}{\partial \lambda}\Big\rangle+mK_1 - K_2.
\end{align}

Using \eqref{2.9.5}, we have
\begin{equation}\label{3.2.2.2}
K_1\,=\, O\bigl(\frac{\|\varphi\|_*}{\lambda^{2+\iota}}\bigr)
\,=\, O\bigl(\epsilon^{\frac{3+\iota}{N-2s}}\bigr).
\end{equation}

Note that
\begin{align}
|(1+t)_{+}^{\zeta}-1-\zeta t|\,\leq\,\left\{ \begin{array}{ll}
C\,t^{2},\,& 1<\zeta\leq2,
\\[2mm]
C\,t^{2}+|t|^{\zeta},\,&\zeta>2.
\end{array} \right.
\end{align}
Suppose that $2^{*}_{s}<3.$ Then it follows from Remark \ref{re1} that
\begin{align}\label{1-16-5}
|K_2|\le\, & C  \int_{\R^N}Z_{\bar r,\bar{y}'',\lambda}^{2^{*}_{s}-3+\epsilon}\varphi^{2} \Bigl|\frac{\partial
Z_{\bar r,\bar{y}'',\lambda}}{\partial\lambda}\Bigl| dy
\le\, \frac{C}\lambda \int_{\R^N}\Big(
  \sum_{j\,=\,1}^m U_{x_j,\lambda}\Big)^{2^{*}_{s}-2+\epsilon}|\varphi|^{2} dy\nonumber\\
\le\,& \frac{C\|\varphi\|_{*}^{2}}\lambda \int_{\R^N}
\Big(\sum_{j\,=\,1}^m \frac{\lambda^{ \frac{N-2s}2} }{(1+\lambda|y-x_j|)^{N-2s}}\Big)^{2^{*}_{s}-2+\epsilon}\Bigl(\sum_{i\,=\,1}^m
\frac{\lambda^{\frac{N-2s}2}}{ (1+\lambda|y-x_i|)^{\frac{N-2s}2+\tau}}\Bigr)^{2} dy \nonumber
\\
=\,&\frac{C\, m\|\varphi\|_{*}^{2}}\lambda
=
O\Bigl(m \epsilon^{\frac{3+\iota}{N-2s}}\Bigr).
\end{align}
Similarly, for  $2^{*}_{s}\geq 3$,  we have
\begin{equation}\label{30-17-5}
\begin{split}
& |K_2|
\le \int_{\R^N} \Bigl(Z_{\bar r,\bar{y}'',\lambda}^{2^{*}_{s}-3+\epsilon}\varphi^2 \bigl|
\frac{\partial
Z_{\bar r,\bar{y}'',\lambda}}{\partial\lambda}\bigr|+
|\varphi|^{2^{*}_{s}-1+\epsilon} \bigl|
\frac{\partial
 Z_{\bar r,\bar{y}'',\lambda}}{\partial\lambda}\bigr|
\Bigr)dy\,=\,
O\bigl(m \epsilon^{\frac{3+\iota}{N-2s}}\bigr).
\end{split}
\end{equation}

So, we have proved
\begin{equation}\label{3.2.11}
\begin{array}{ll}
\Big\langle I'(Z_{\bar r,\bar{y}'',\lambda}+\varphi),\ds\frac{\partial
Z_{\bar r,\bar{y}'',\lambda}}{\partial \lambda}\Big\rangle \,=\,\Big\langle
I'(Z_{\bar r,\bar{y}'',\lambda}),\ds\frac{\partial Z_{\bar r,\bar{y}'',\lambda}}{\partial
\lambda}\Big\rangle
+O\big(m \epsilon^{\frac{3+\iota}{N-2s}}\big).
\end{array}
\end{equation}
Using Lemma~\ref{lema.2} in Appendix \ref{s4}, we obtain the result.
\end{proof}

%

Next, we will estimate \eqref{identityofscaling} and \eqref{identityoftranslation}.   Let us point out that
\eqref{identityofscaling} is  the local  Pohozaev identity
generating from scaling, while \eqref{identityoftranslation} is  the local Pohozaev identities
generating from translations.


Noting that
$\tilde{u}_\epsilon$ satisfies the following equation
\begin{equation}
\left\{
\begin{array}{ll}
div(t^{1-2s}\nabla\tilde{u}_\epsilon)=0\quad  \text{in}\, \R_{+}^{N+1},
\\
-\lim_{t\rightarrow 0}t^{1-2s}\partial_t
\tilde{u}_\epsilon(x,t)=K(r,y'')(u_\epsilon)_{+}^{2^{*}_{s}-1+\epsilon}+\ds\sum_{l=1}^{N}c_{l}\ds\sum_{j=1}^{m}U_{x_{j},\lambda}^{2^{*}_{s}-2}Z_{j,l},\quad  \text{in}\, \R^{N}.
\end{array}
\right.
\end{equation}
Then, we can get that
\begin{align}
0&=\int_{\mathcal{B}_{\rho}^{+}(y_0)}div(t^{1-2s}\nabla\tilde{u}_\epsilon)\tilde{u}_\epsilon dydt \nonumber
\\
&=\int_{\partial\mathcal{B}_{\rho}^{+}(y_0)}t^{1-2s}\nabla\tilde{u}_\epsilon\tilde{u}_\epsilon\nu dS
-\int_{\mathcal{B}_{\rho}^{+}(y_0)}t^{1-2s}\nabla\tilde{u}_\epsilon\nabla\tilde{u}_\epsilon dydt \nonumber
\\
&=\int_{\partial'\mathcal{B}_{\rho}^{+}(y_0)}t^{1-2s}\nabla\tilde{u}_\epsilon\tilde{u}_\epsilon\nu dS
+\int_{\partial''\mathcal{B}_{\rho}^{+}(y_0)}t^{1-2s}\nabla\tilde{u}_\epsilon\tilde{u}_\epsilon\nu dS
-\int_{\mathcal{B}_{\rho}^{+}(y_0)}t^{1-2s}|\nabla\tilde{u}_\epsilon|^2d ydt \nonumber
\\
&=-\int_{\partial'\mathcal{B}_{\rho}^{+}(y_0)}t^{1-2s}\partial_{t}\tilde{u}_\epsilon\tilde{u}_\epsilon dS
+\int_{\partial''\mathcal{B}_{\rho}^{+}(y_0)}t^{1-2s}\nabla\tilde{u}_\epsilon\tilde{u}_\epsilon\nu dS
-\int_{\mathcal{B}_{\rho}^{+}(y_0)}t^{1-2s}|\nabla\tilde{u}_\epsilon|^2d ydt \nonumber
\\
&=\int_{B_{\rho}(y_0)}\big[K(r,y'')(u_\epsilon)_{+}^{2^{*}_{s}-1+\epsilon}
+\ds\sum_{l=1}^{N}c_{l}\ds\sum_{j=1}^{m}U_{x_{j},\lambda}^{2^{*}_{s}-2}Z_{j,l}\big]u_\epsilon dy \nonumber
\\
&\quad+\int_{\partial''\mathcal{B}_{\rho}^{+}(y_0)}t^{1-2s}\tilde{u}_\epsilon \frac{\partial \tilde{u}_\epsilon}{\partial\nu} dS
-\int_{\mathcal{B}_{\rho}^{+}(y_0)}t^{1-2s}|\nabla\tilde{u}_\epsilon|^2d ydt.
\end{align}
Therefore, we get
\begin{equation}\label{p23}
\begin{split}
\int_{\mathcal{B}_{\rho}^{+}(y_0)}t^{1-2s}|\nabla\tilde{u}_\epsilon|^2d ydt
=&\int_{B_{\rho}(y_0)} \Big[K(r,y'')(u_\epsilon)_{+}^{2^{*}_{s}+\epsilon}
+\sum_{l=1}^{N}c_{l}\ds\sum_{j=1}^{m}U_{x_{j},\lambda}^{2^{*}_{s}-2}Z_{j,l}u_\epsilon\Big] dy
\\
&
+\int_{\partial''\mathcal{B}_{\rho}^{+}(y_0)}t^{1-2s}\frac{\partial \tilde{u}_\epsilon}{\partial\nu}\tilde{u}_\epsilon dS.
\end{split}
\end{equation}

Since
$$
\int_{\R^{N}}\ds\sum_{l=1}^{N}c_{l}\ds\sum_{j=1}^{m}U_{x_{j},\lambda}^{2^{*}_{s}-2}Z_{j,l}\varphi dy\,=\,0,
$$
then
\begin{align*}
 \int_{\R^{N}}\ds\sum_{l=1}^{N}c_{l}\ds\sum_{j=1}^{m}U_{x_{j},\lambda}^{2^{*}_{s}-2}Z_{j,l} u_{\epsilon}dy
 \,=\,&\int_{\R^{N}}\ds\sum_{l=1}^{N}c_{l}\ds\sum_{j=1}^{m}U_{x_{j},\lambda}^{2^{*}_{s}-2}Z_{j,l}(Z_{\bar{r},\bar{y}'',\lambda}+\varphi)dy
 \\
 \,=\,&\int_{\R^{N}}\ds\sum_{l=1}^{N}c_{l}\ds\sum_{j=1}^{m}U_{x_{j},\lambda}^{2^{*}_{s}-2}Z_{j,l}Z_{\bar{r},\bar{y}'',\lambda}dy.
\end{align*}

Therefore \eqref{identityofscaling} is equivalent to
\begin{align}\label{identity2}
&\frac{1}{2^{*}_{s}+\epsilon}
\int_{B_{\rho}(y_0)}\langle \nabla K,y\rangle(u_\epsilon)_{+}^{2^*_{s}+\epsilon} dy \nonumber
\\
&
=\Big(\frac{N-2s}{2}-\frac{N}{2^{*}_{s}+\epsilon}\Big)\int_{ B_{\rho}(y_0)}K(r,y'') (u_\epsilon)_{+}^{2^*_{s}+\epsilon}dy
+\frac{1}{2^{*}_{s}+\epsilon} \int_{\partial B_{\rho}(y_0)} K(r,y'')(u_\epsilon)_{+}^{2^{*}_{s}+\epsilon}\langle y,\nu\rangle ds \nonumber
\\
&
\quad+\frac{N-2s}{2}\int_{\partial''\mathcal{B}_{\rho}^{+}(y_0)}t^{1-2s}\tilde{u}_\epsilon\frac{\partial\tilde{u}_\epsilon}{\partial\nu} d S
+\int_{\partial''\mathcal{B}_{\rho}^{+}(y_0)} t^{1-2s}\langle Y,\nabla\tilde{u}_\epsilon\rangle\frac{\partial\tilde{u}_\epsilon}{\partial\nu} dS
\\
&\quad-\frac{1}{2}\int_{\partial''\mathcal{B}_{\rho}^{+}(y_0)}t^{1-2s}|\nabla\tilde{u}_\epsilon|^2\langle Y,\nu\rangle dS
+\frac{N-2s}{2}\int_{\R^{N}}\ds\sum_{l=1}^{N}c_{l}\ds\sum_{j=1}^{m}U_{x_{j},\lambda}^{2^{*}_{s}-2}Z_{j,l}Z_{\bar{r},\bar{y}'',\lambda}\,dy \nonumber
\\
& \quad- \frac{N-2s}{2}\int_{B^{c}_{\rho}(y_0)}\ds\sum_{l=1}^{N}c_{l}\ds\sum_{j=1}^{m}U_{x_{j},\lambda}^{2^{*}_{s}-2}Z_{j,l}\varphi dy.\nonumber
\end{align}

\begin{lem}\label{lem3.3}
\eqref{identityofscaling} and \eqref{identityoftranslation} are equivalent to
\begin{equation}\label{p1-3}
\begin{split}
\int_{B_{\rho}(y_0)}\langle \nabla K,y\rangle(u_\epsilon)_{+}^{2^*_{s}+\epsilon} dy
=O\big(m\epsilon^{\frac{1+\iota}{N-2s}}\big)
\end{split}
\end{equation}
and
\begin{equation}\label{p2-3}
\int_{B_{\rho}(y_0)}\frac{\partial K(y)}{\partial y_i}(u_\epsilon)_{+}^{2^{*}_{s}+\epsilon}  dy
=O\big(m\epsilon^{\frac{1+\iota}{N-2s}}\big)\,\,\,i=3,\cdots,N.
\end{equation}

\end{lem}

\begin{proof}
Here we only prove \eqref{p1-3} since the proof of \eqref{p2-3} is similar.

First, we have
\begin{equation}\label{l3-0}
\begin{split}
\Big(\frac{N-2s}{2}-\frac{N}{2^{*}_{s}+\epsilon}\Big)\int_{ B_{\rho}(y_0)}K(r,y'') (u_\epsilon)_{+}^{2^*_{s}+\epsilon} dy
=o(\epsilon)=o(m\epsilon^{1+\frac{N-2s-2}{(N-2s)^{2}}})=o(m\epsilon^{\frac{1+\iota}{N-2s}}).
\end{split}
\end{equation}

Noting that $K(r,y'')$ is bounded, we have
\begin{align}\label{l3-00}
&\frac{1}{2^{*}_{s}+\epsilon} \int_{\partial B_{\rho}(y_0)} K(r,y'')(u_\epsilon)_{+}^{2^{*}_{s}+\epsilon}\langle y,\nu\rangle ds\\
& \leq C\int_{\partial B_{\rho}(y_0)}|\varphi_\epsilon|^{2^{*}_{s}}ds
+C\int_{\partial B_{\rho}(y_0)}Z_{\bar{r},\bar{y}'',\lambda}^{2^{*}_{s}} ds\nonumber\\
&\leq C\|\varphi\|^{2^{*}_{s}}\int_{ \partial B_{\rho}(y_0)}
\Big(\sum_{j=1}^{m}\frac{\lambda^{\frac{N-2s}{2}}}{(1+\lambda|y-x_{j}|)^{\frac{N-2s}{2}+\tau}}\Big)^{2^{*}_{s}} ds \nonumber
\\
&\quad
+C\int_{ \partial B_{\rho}(y_0)}
\Big(\sum_{j=1}^{m}\frac{\lambda^{\frac{N-2s}{2}}}{(1+\lambda|y-x_{j}|)^{N-2s}}\Big)^{2^{*}_{s}}ds \nonumber
\\
&\leq C\frac{m^{2^{*}_{s}}\|\varphi\|^{2^{*}_{s}}}{\lambda^{2^{*}_{s}\tau}}
+C\frac{m^{2^{*}_{s}}}{\lambda^{N}}
 \leq Cm\epsilon^{\frac{1+\iota}{N-2s}}.
\end{align}
Note that
\begin{equation}\label{l3-1}
\begin{split}
&\int_{\partial''\mathcal{B}_{\rho}^{+}(y_0)} t^{1-2s}\langle \nabla\tilde{u}_\epsilon, Y\rangle\frac{\partial\tilde{u}_\epsilon}{\partial\nu} dS
\\
&=\int_{\partial''\mathcal{B}_{\rho}^{+}(y_0)}t^{1-2s}\langle \nabla \tilde{Z}_{\bar{r},\bar{y}'',\lambda},Y\rangle\frac{\partial\tilde{Z}_{\bar{r},\bar{y}'',\lambda}}{\partial\nu} dS
+\int_{\partial''\mathcal{B}_{\rho}^{+}(y_0)}t^{1-2s}\langle \nabla\tilde{\varphi},Y\rangle\frac{\partial \tilde{\varphi}}{\partial\nu}dS
\\
&\quad+\int_{\partial''\mathcal{B}_{\rho}^{+}(y_0)}t^{1-2s}\langle \nabla\tilde{Z}_{\bar{r},\bar{y}'',\lambda},Y\rangle\frac{\partial \tilde{\varphi}}{\partial\nu}dS
+\int_{\partial''\mathcal{B}_{\rho}^{+}(y_0)}t^{1-2s}\langle \nabla \tilde{\varphi},Y\rangle\frac{\partial \tilde{Z}_{\bar{r},\bar{y}'',\lambda}}{\partial\nu} dS.
\end{split}
\end{equation}
Next, we will estimates the terms in \eqref{l3-1} one by one.

By Lemma \ref{lembb5}, we have
\begin{align}\label{l3-2}
&\Big|\int_{\partial''\mathcal{B}_{\rho}^{+}(y_0)}t^{1-2s}\langle \nabla\tilde{Z}_{\bar{r},\bar{y}'',\lambda},Y\rangle\frac{\partial \tilde{Z}_{\bar{r},\bar{y}'',\lambda}}{\partial\nu} dS\Big| \nonumber\\
&\leq C\int_{\partial''\mathcal{B}_{\rho}^{+}(y_0)}t^{1-2s}|\nabla\tilde{Z}_{\bar{r},\bar{y}'',\lambda}|^{2}dS \nonumber\\
& \leq\frac{C}{\lambda^{N-2s}}\int_{\partial''\mathcal{B}_{\rho}^{+}(y_0)}t^{1-2s}
\Big(\sum_{j=1}^{m}\frac{1}{(1+|y-x_{j}|)^{N-2s+1}}\Big)^{2}dS\\
&\leq \frac{Cm^{2}}{\lambda^{N-2s}}\int_{\partial''\mathcal{B}_{\rho}^{+}(y_0)}t^{1-2s}
\frac{1}{(1+|y-x_{j}|)^{2N-4s+2}} dS \nonumber\\
&\leq \frac{Cm^{2}}{\lambda^{N-2s}}\leq Cm\epsilon^{\frac{1+\iota}{N-2s}}.\nonumber
\end{align}

By Lemma \ref{lembb6}, we have
\begin{equation}\label{l3-3}
\begin{split}
\Big|\int_{\partial''\mathcal{B}_{\rho}^{+}(y_0)}t^{1-2s}\langle \nabla\tilde{\varphi},Y\rangle\frac{\partial \tilde{\varphi}}{\partial\nu}dS\Big|
\leq\, &C\int_{\partial''\mathcal{B}_{\rho}^{+}(y_0)}t^{1-2s}|\nabla \tilde{\varphi}|^{2} dS
\\
\leq\, &Cm\frac{\|\varphi\|^{2}_{*}}{\lambda^{\tau}}\leq Cm\epsilon^{\frac{1+\iota}{N-2s}}.
\end{split}
\end{equation}

By \eqref{l3-2} and \eqref{l3-3}, we have
\begin{equation}\label{l3-4}
\begin{split}
&\Big|\int_{\partial''\mathcal{B}_{\rho}^{+}(y_0)}t^{1-2s}\langle \nabla\tilde{Z}_{\bar{r},\bar{y}'',\lambda},Y\rangle\frac{\partial \tilde{\varphi}}{\partial\nu}dS\Big|
\\
&\leq C\int_{\partial''\mathcal{B}_{\rho}^{+}(y_0)}t^{1-2s}|\nabla\tilde{Z}_{\bar{r},\bar{y}'',\lambda}||\nabla \tilde{\varphi}|dS\\
&\leq C\int_{\partial''\mathcal{B}_{\rho}^{+}(y_0)}t^{1-2s}|\nabla\tilde{Z}_{\bar{r},\bar{y}'',\lambda}|^{2}dS
+C\int_{\partial''\mathcal{B}_{\rho}^{+}(y_0)}t^{1-2s}|\nabla \tilde{\varphi}|^{2}dS\\
&\leq Cm\epsilon^{\frac{1+\iota}{N-2s}}.
\end{split}
\end{equation}
Similar to \eqref{l3-4}, we have
\begin{equation}\label{l3-5}
\begin{split}
&\Big|\int_{\partial''\mathcal{B}_{\rho}^{+}(y_0)}t^{1-2s}\langle \nabla \tilde{\varphi},Y\rangle\frac{\partial \tilde{Z}_{\bar{r},\bar{y}'',\lambda}}{\partial\nu} dS\Big|
\\
&\leq C\int_{\partial''\mathcal{B}_{\rho}^{+}(y_0)}t^{1-2s}|\nabla \tilde{\varphi}||\nabla\tilde{Z}_{\bar{r},\bar{y}'',\lambda}| dS
\leq Cm\epsilon^{\frac{1+\iota}{N-2s}}.
\end{split}
\end{equation}
From \eqref{l3-1} to \eqref{l3-5}, we have
\begin{equation}\label{l3-6}
\begin{split}
\Big|\int_{\partial''\mathcal{B}_{\rho}^{+}(y_0)} t^{1-2s}\langle \nabla\tilde{u}_\epsilon,Y
\rangle\frac{\partial\tilde{u}_\epsilon}{\partial\nu}dS\Big|
\leq Cm\epsilon^{\frac{1+\iota}{N-2s}} .
\end{split}
\end{equation}
Just by the same argument as that of \eqref{l3-6}, we can prove
\begin{equation}\label{l3-7}
\begin{split}
\Big|\int_{\partial''\mathcal{B}_{\rho}^{+}(y_0)} t^{1-2s}|\nabla\tilde{u}_\epsilon|^{2}\langle Y,\nu\rangle dS\Big|
\leq Cm\epsilon^{\frac{1+\iota}{N-2s}}.
\end{split}
\end{equation}

Similar to \eqref{l3-2}, by Lemma \ref{lembb5} we have
\begin{equation}\label{l3-2-1}
\begin{split}
&\Big|\int_{\partial''\mathcal{B}_{\rho}^{+}(y_0)}t^{1-2s} \frac{\partial\tilde{Z}_{\bar{r},\bar{y}'',\lambda}}{\partial\nu} \tilde{Z}_{\bar{r},\bar{y}'',\lambda}dS\Big|\\
&\leq C\int_{\partial''\mathcal{B}_{\rho}^{+}(y_0)}t^{1-2s}|\nabla\tilde{Z}_{\bar{r},\bar{y}'',\lambda}|\,|\tilde{Z}_{\bar{r},\bar{y}'',\lambda}|\,dS\\
& \leq\frac{C}{\lambda^{N-2s}}\int_{\partial''\mathcal{B}_{\rho}^{+}(y_0)}t^{1-2s}
\sum_{j=1}^{m}\frac{1}{(1+|y-x_{j}|)^{N-2s+1}}\sum_{j=1}^{m}\frac{1}{(1+|y-x_{j}|)^{N-2s}}dS\\
&\leq \frac{Cm^{2}}{\lambda^{N-2s}}\int_{\partial''\mathcal{B}_{\rho}^{+}(y_0)}t^{1-2s}
\frac{1}{(1+|y-x_{j}|)^{2N-4s+1}}dS\\
&\leq \frac{Cm^{2}}{\lambda^{N-2s}}\leq Cm\epsilon^{\frac{1+\iota}{N-2s}}.
\end{split}
\end{equation}
By Lemma \ref{lembb7}, we have
\begin{equation}\label{l3-8}
\begin{split}
\Big|\int_{\partial''\mathcal{B}_{\rho}^{+}(y_0)} t^{1-2s}|\tilde{\varphi}|^{2}dS\Big|
&\leq C\frac{\|\varphi\|^{2}_{*}}{\lambda^{2\tau}}\int_{\partial''\mathcal{B}_{\rho}^{+}(y_0)} t^{1-2s}
\Big(\sum_{j=1}^{m}\frac{1}{(1+|y-x_{j}|)^{\frac{N-2s}{2}+\tau}}\Big)^{2}dS\\
&\leq Cm\frac{\|\varphi\|^{2}_{*}}{\lambda^{2\tau}}\int_{\partial''\mathcal{B}_{\rho}^{+}(y_0)} t^{1-2s}
\frac{1}{(1+|y-x_{j}|)^{N-2s+2\tau}}dS\\
&\leq Cm\epsilon^{\frac{1+\iota}{N-2s}}.
\end{split}
\end{equation}
It follows from \eqref{l3-3} and \eqref{l3-8} that
\begin{equation}\label{l3-9}
\begin{split}
\Big|\int_{\partial''\mathcal{B}_{\rho}^{+}(y_0)} t^{1-2s}\frac{\partial \tilde{\varphi}}{\partial\nu}\tilde{\varphi} dS\Big|
&\leq C\int_{\partial''\mathcal{B}_{\rho}^{+}(y_0)} t^{1-2s}|\nabla \tilde{\varphi}||\tilde{\varphi}|dS\\
&\leq C\int_{\partial''\mathcal{B}_{\rho}^{+}(y_0)} t^{1-2s}|\nabla \tilde{\varphi}|^{2} dS
+C\int_{\partial''\mathcal{B}_{\rho}^{+}(y_0)} t^{1-2s}| \tilde{\varphi}|^{2}dS\\
&\leq Cm\epsilon^{\frac{1+\iota}{N-2s}}.
\end{split}
\end{equation}
Similarly, by \eqref{l3-2} and \eqref{l3-8}, we have
\begin{equation}\label{l3-10}
\begin{split}
\Big|\int_{\partial''\mathcal{B}_{\rho}^{+}(y_0)} t^{1-2s}\frac{\partial \tilde{Z}_{\bar{r},\bar{y}'',\lambda}}{\partial\nu}\tilde{\varphi}dS\Big|
\leq Cm\epsilon^{\frac{1+\iota}{N-2s}}.
\end{split}
\end{equation}
Similar to \eqref{l3-4}, by \eqref{l3-3} and Lemma \ref{lembb5}  we can also prove
\begin{equation}\label{l3-11}
\begin{split}
\Big|\int_{\partial''\mathcal{B}_{\rho}^{+}(y_0)} t^{1-2s}\frac{\partial \tilde{\varphi}}{\partial\nu}\tilde{Z}_{\bar{r},\bar{y}'',\lambda}dS\Big|
\leq Cm\epsilon^{\frac{1+\iota}{N-2s}}.
\end{split}
\end{equation}
Hence, from \eqref{l3-2-1} to \eqref{l3-11} we have
\begin{equation}\label{l3-12}
\begin{split}
&\Big|\int_{\partial''\mathcal{B}_{\rho}^{+}(y_0)}t^{1-2s}\tilde{u}_\epsilon\frac{\partial\tilde{u}_\epsilon}{\partial\nu}dS\Big|\\
&=\Big|\int_{\partial''\mathcal{B}_{\rho}^{+}(y_0)}t^{1-2s}
\frac{\partial\tilde{Z}_{\bar{r},\bar{y}'',\lambda}}{\partial\nu}\tilde{Z}_{\bar{r},\bar{y}'',\lambda} dS
+\int_{\partial''\mathcal{B}_{\rho}^{+}(y_0)}t^{1-2s}
\frac{\partial\tilde{\varphi}}{\partial\nu}\tilde{\varphi} dS\\
&
\quad\quad+\int_{\partial''\mathcal{B}_{\rho}^{+}(y_0)}t^{1-2s}
\frac{\partial\tilde{Z}_{\bar{r},\bar{y}'',\lambda}}{\partial\nu}\tilde{\varphi} dS
+\int_{\partial''\mathcal{B}_{\rho}^{+}(y_0)}t^{1-2s}
\frac{\partial\tilde{\varphi}}{\partial\nu}\tilde{Z}_{\bar{r},\bar{y}'',\lambda} dS\Big|\\
&\leq Cm\epsilon^{\frac{1+\iota}{N-2s}}.
\end{split}
\end{equation}
By \eqref{2.22}, we know
$$
 |c_{l}|\leq
C\epsilon^{\frac{1+n_l+\iota}{N-2s}}.
$$
Note that
\begin{equation}\label{l3-13}
\begin{split}
\int_{ B_{\rho}(x_{0})}\ds\sum_{j=1}^{m}U_{x_{j},\lambda}^{2^{*}_{s}-2}Z_{j,l}Z_{\bar{r},\bar{y}'',\lambda}
&=\ds\sum_{j=1}^{m}\int_{ B_{\rho}(x_{0})}U_{x_{j},\lambda}^{2^{*}_{s}-1}Z_{j,l}
+\ds\sum_{j=1}^{m}\int_{ B_{\rho}(x_{0})}\sum_{i\neq j}U_{x_{j},\lambda}^{2^{*}_{s}-2}Z_{j,l}U_{x_{i},\lambda}dy\\
&=O\Big(m \lambda^{n_l}\Big).
\end{split}
\end{equation}
Therefore, we have
\begin{equation}\label{l3-14}
\begin{split}
\sum_{l=1}^{N}c_{l}\int_{ B_{\rho}(x_{0})}\ds\sum_{j=1}^{m}U_{x_{j},\lambda}^{2^{*}_{s}-2}Z_{j,l}Z_{\bar{r},\bar{y}'',\lambda}
\leq Cm\epsilon^{\frac{1+\iota}{N-2s}}.
\end{split}
\end{equation}
Similarly, we can prove that
\begin{equation} \label{13-15}
\sum_{l=1}^{N}c_{l}\int_{ B_{\rho}^{c}(x_{0})}\ds\sum_{j=1}^{m}U_{x_{j},\lambda}^{2^{*}_{s}-2}Z_{j,l} \varphi_{\epsilon}
\leq Cm\epsilon^{\frac{1+\iota}{N-2s}}.
\end{equation}

Combining \eqref{identity2}, \eqref{l3-0}, \eqref{l3-00}, \eqref{l3-6}, \eqref{l3-7}, \eqref{l3-12},
\eqref{l3-14} and \eqref{13-15}, we can prove that \eqref{p1-3} holds.
\end{proof}

Next, we prove
\begin{lem}\label{lem3.4}
For  any $C^1$ bounded function $g(r, y''), $ it holds
 \begin{equation}\label{20-7-2}
\int_{B_\rho(y_0)}  g(r, y'') |u_l|^{2^{*}_{s}+\epsilon}\,dy
 =m\Bigl(  g (\bar r,\bar{y}'') \int_{\R^{N}} U_{0, 1}^{2^{*}_{s}+\epsilon} \,dy+ o\bigl(\epsilon^{\frac{1-\iota}{N-2s}}\bigr)\Bigr),
\end{equation}
where $o(1)$ denotes a quantity that goes to zero when $\epsilon$ goes to zero.
\end{lem}

\begin{proof}
Since $u_\epsilon  = Z_{\bar r,\bar{y}'',\mu } +\varphi,$ we have
\begin{equation}\label{21-7-2}
\begin{split}
 &\int_{B_\rho(y_0)}g(r, y'') |u_\epsilon|^{2^{*}_{s}+\epsilon}\,dy
 =\int_{B_\rho(y_0)} g(r, y'') |Z_{\bar r,\bar{y}'',\lambda }+\varphi|^{2^{*}_{s}+\epsilon}\,dy
 \\
=&\int_{B_\rho(y_0)} g(r, y'') |Z_{\bar r,\bar{y}'',\lambda }|^{2^{*}_{s}+\epsilon}\,dy+ \int_{B_\rho(y_0)}  g(r, y'') |\varphi|^{2^{*}_{s}+\epsilon}\,dy
 \\
 &+O\Big(\int_{B_\rho(y_0)}   |Z_{\bar r,\bar{y}'',\lambda }||\varphi|^{2^{*}_{s}+\epsilon-1}\,dy
 +\int_{B_\rho(y_0)}  |Z_{\bar r,\bar{y}'',\epsilon }|^{2^{*}_{s}+\epsilon-1}|\varphi|\,dy\Big).
\end{split}
\end{equation}

By the discrete H\"{o}lder inequality and Remark \ref{re1}, we can check that
\begin{align}\label{25-7-3}
&\int_{B_\rho(y_0)}
 |Z_{\bar r,\bar{y}'',\lambda }||\varphi|^{2^{*}_{s}+\epsilon-1} \,dy\nonumber
\\
&\leq C\|\varphi\|^{2^{*}_{s}+\epsilon-1}_{*}\int_{\R^N} \sum_{j=1}^{m}\frac{\lambda^{\frac{N-2s}{2}}}{(1+\lambda|y-x_{j}|)^{N-2s}}
\Big(\sum_{i=1}^{m}\frac{\lambda^{\frac{N-2s}{2}}}{(1+\lambda|y-x_{i}|)^{\frac{N-2s}{2}+\tau}}\Big)^{\frac{N+2s}{N-2s}+\epsilon} \nonumber
\\
&\leq C\|\varphi\|^{2^{*}_{s}+\epsilon-1}_{*}\int_{\R^N} \sum_{j=1}^{m}\frac{\lambda^{\frac{N-2s}{2}}}{(1+\lambda|y-x_{j}|)^{N-2s}}
\sum_{i=1}^{m}\frac{\lambda^{\frac{N+2s}{2}+\frac{N-2s}{2}\epsilon}}{(1+\lambda|y-x_{i}|)^{\frac{N+2s}{2}+\tau}} \nonumber
\\
&\leq C\|\varphi\|^{2^{*}_{s}+\epsilon-1}_{*}
\int_{\R^N} \Big[\sum_{j=1}^{m}\frac{ \lambda^{N}}{(1+\lambda|y-x_{j}|)^{\frac{3N-2s}{2}+\tau}}
+\sum_{j\neq i}\frac{ \lambda^{N}}{(1+\lambda|y-x_{j}|)^{N-2s}}\frac{1}{(1+\lambda|y-x_{i}|)^{\frac{N+2s}{2}+\tau}}\Big] \nonumber
\\
&\leq  Cm\epsilon^{\frac{1+\iota}{N-2s}}+C\epsilon^{\frac{1+\iota}{N-2s}}\int_{\R^N}\sum_{j\neq i}\frac{1}{(\lambda|x_{i}-x_{j}|)^{\tau}}
\Big(\frac{\lambda^{N}}{(1+\lambda|y-x_{j}|)^{\frac{3N-2s}{2}}}
+\frac{ \lambda^{N}}{(1+\lambda|y-x_{i}|)^{\frac{3N-2s}{2}}}\Big) \nonumber
\\
&\leq C m\epsilon^{\frac{1+\iota}{N-2s}}
\end{align}
and
\begin{align}\label{25-7-2}
&\int_{B_\rho(y_0)}
 |Z_{\bar r,\bar{y}'',\lambda }|^{2^{*}_{s}+\epsilon-1}|\varphi|\,dy\nonumber
\\
&\leq C\|\varphi\|_{*}\int_{\mathbb R^N} \sum_{j=1}^{m}\frac{\lambda^{\frac{N+2s}{2}+\frac{N-2s}{2}\epsilon}}{(1+\lambda|y-x_{j}|)^{N+2s}}
\sum_{i=1}^{m}\frac{\lambda^{\frac{N-2s}{2}}}{(1+\lambda|y-x_{i}|)^{\frac{N-2s}{2}+\tau}}\nonumber\\
&\leq C\|\varphi\|_{*}
\int_{\mathbb R^N} \Big[\sum_{j=1}^{m}\frac{\lambda^{N}}{(1+\lambda|y-x_{j}|)^{\frac{3N+2s}{2}+\tau}}
+\sum_{j\neq i}\frac{\lambda^{N}}{(1+\lambda|y-x_{j}|)^{N+2s}}\frac{1}{(1+\lambda|y-x_{i}|)^{\frac{N-2s}{2}+\tau}}\Big]\nonumber\\
&\leq Cm\epsilon^{\frac{1+\iota}{N-2s}}
+C\epsilon^{\frac{1+\iota}{N-2s}}\int_{\mathbb R^N}\sum_{j\neq i}\frac{1}{(\lambda|x_{j}-x_{i}|)^{\tau}}
\Big(\frac{ \lambda^{N}}{(1+\lambda|y-x_{i}|)^{\frac{3N+2s}{2}}}
+\frac{ \lambda^{N}}{(1+\lambda|y-x_{i}|)^{\frac{3N+2s}{2}}}\Big)\nonumber
\\
&\leq Cm\epsilon^{\frac{1+\iota}{N-2s}}.
\end{align}

So from \eqref{21-7-2} and \eqref{25-7-2}, we obtain the following estimate
\begin{equation}\label{28-7-2}
 \int_{B_\rho(y_0)}  g(r, y'')|u_\epsilon|^{2^{*}_{s}+\epsilon}\,dy\,=\, \int_{B_\rho(y_0)}  g(r, y'')  Z_{\bar r,\bar{y}'',\lambda }^{2^{*}_{s}+\epsilon}\,dy  + mo\bigl(\epsilon^{\frac{1-\iota}{N-2s}}\bigr).
\end{equation}
Since
\begin{align}\label{29-7-2}
  &\int_{B_\rho(y_0)}  g(r, y'')  U_{x_{j},\lambda}^{2^{*}_{s}+\epsilon}\,dy
  \,=\,\int_{B_\rho(y_0)}  g(\bar{r}, \bar{y}'')  U_{x_{j},\lambda}^{2^{*}_{s}+\epsilon}\,dy
  +\int_{B_\rho(y_0)} [g(r, y'')- g(\bar{r}, \bar{y}'')]  U_{x_{j},\lambda}^{2^{*}_{s}+\epsilon}\,dy \nonumber
  \\
  &=\int_{\mathbb R^N}  g(\bar{r}, \bar{y}'')  U_{x_{j},\lambda}^{2^{*}_{s}+\epsilon}
  +\int_{B^{c}_\rho(y_0)}  g(\bar{r}, \bar{y}'')  U_{x_{j},\lambda}^{2^{*}_{s}+\epsilon}+\int_{B_\rho(y_0)} [g(r, y'')- g(\bar{r}, \bar{y}'')]  U_{x_{j},\lambda}^{2^{*}_{s}+\epsilon}\,dy\nonumber
  \\
  &=g (\bar r,\bar{y}'') \int_{\mathbb R^N} U_{0, 1}^{2^{*}_{s}+\epsilon}\,dy  + o\bigl(\epsilon^{\frac{1-\iota}{N-2s}}\bigr),
\end{align}
and
\begin{equation}\label{29-7-3}
\begin{split}
&\sum_{i\ne j}  \int_{B_{\rho}(y_0)}  g(r, y'')  U_{x_{i},\lambda} U^{2^{*}_{s}+\epsilon-1}_{x_{j},\lambda}\,dy
 \leq C\sum_{i\ne j}  \int_{\mathbb R^N} U_{x_{i},\lambda} U^{2^{*}_{s}+\epsilon-1}_{x_{j},\lambda}\,dy
\\
&\leq C\sum_{i\ne j}  \int_{\mathbb R^N}\frac{\lambda^{N+\frac{N-2s}{2}\epsilon}}{(1+\lambda|y-x_{i}|)^{N-2s}}
\frac{1}{(1+\lambda|y-x_{j}|)^{N+2s}}\,dy
\\
&\leq \sum_{i\neq j}\frac{1}{(\lambda|x_{i}-x_{j}|)^{N-2s}}
\int_{\mathbb R^N} \Big[\frac{\lambda^{N}}{(1+\lambda|y-x_{i}|)^{N+2s}}
+\frac{\lambda^{N}}{(1+\lambda|y-x_{j}|)^{N+2s}}\Big]\,dy
\\
&\leq C\sum_{i\neq j}\Big(\frac{m}{\lambda}\Big)^{N-2s}
=O \Big( \frac{m}{\lambda^{N-2s}} \Big)  =o\Big(m\epsilon^{\frac{1-\iota}{N-2s}}\Big).
\end{split}
\end{equation}

As a result,
\begin{equation}\label{30-7-2}
 \int_{B_\rho(y_0)}g(r, y'') |u_\epsilon|^{2^{*}_{s}+\epsilon}\,dy
 \,=\,m\Bigl(  g (\bar r,\bar{y}'') \int_{\mathbb R^N} U_{0, 1}^{2^{*}_{s}+\epsilon} \,dy + o\bigl(\epsilon^{\frac{1-\iota}{N-2s}}\bigr)\Bigr).
\end{equation}

\end{proof}

Now it follows form Lemma~\ref{lem3.4}, \eqref{p1-3} and \eqref{p2-3} that
\begin{equation}\label{31-7-2}
  m\Bigl( \frac{\partial K(\bar r,\bar y'')}{\partial \bar r} \int_{\mathbb R^N} U_{0, 1}^{2^{*}_{s}+\epsilon} \,dy + o\bigl(\epsilon^{\frac{1-\iota}{N-2s}}\bigr)\Bigr)\,=\,o\bigl( m\epsilon^{\frac{1}{N-2s}}\bigr),
\end{equation}
and
\begin{equation}\label{32-7-2}
  m\Bigl( \frac{\partial (K(\bar r, \bar y''))}{\partial \bar y_{i}} \int_{\mathbb R^N} U_{0, 1}^{2^{*}_{s}+\epsilon} \,dy + o\bigl(\epsilon^{\frac{1-\iota}{N-2s}}\bigr)\Bigr)
  \,=\,o\bigl(m\epsilon^{\frac{1}{N-2s}}\bigr),\,\,\,\,i=3,\cdots,N.
\end{equation}
Therefore, the equations to determine $( \bar r, \bar y'')$ are
\begin{equation}\label{67-18-1}
\frac{\partial K(\bar r,\bar y'')}{\partial \bar r}\,=\,o(\epsilon^{\frac{1-\iota}{N-2s}}),
\end{equation}
and
\begin{equation}\label{70-18-1}
  \frac{\partial (K(\bar r, \bar y''))}{\partial \bar y_{i}}
=\,o(\epsilon^{\frac{1-\iota}{N-2s}}), \quad i=3, \cdots, N.
\end{equation}

\begin{proof}[Proof of Theorem~\ref{thm1.2}]

We have proved that \eqref{identityofscaling}, \eqref{identityoftranslation} and \eqref{integralofum} are equivalent to
\begin{equation}\label{71-18-1}
   \frac{\partial (K(\bar r, \bar y'')}{\partial \bar r}
=o(\epsilon^{\frac{1-\iota}{N-2s}}),
\end{equation}
\begin{equation}\label{72-18-1}
\frac{\partial K(\bar r,\bar y'')}{\partial \bar y_i}=o(\epsilon^{\frac{1-\iota}{N-2s}}), \quad i=3, \cdots, N
\end{equation}
and
\begin{equation}\label{73-18-1}
-\frac{B_{1}}{\lambda^{3}} +\frac{B_3 m^{N-2s}}{\lambda^{N-2s+1}}
=O\big(\epsilon^{\frac{3+\iota}{N-2s}}\big).
\end{equation}

Let  $\lambda = t m^{\frac{N-2s}{N-2s-2}}$, then $t\in [ L_0, L_1]$ since   $\lambda\in[L_{0}m^{\frac{N-2s}{N-2s-2}},L_{1}m^{\frac{N-2s}{N-2s-2}}].$
Then, from \eqref{73-18-1}, we get
\begin{equation}\label{74-18-1}
-\frac{B_{1}}{t^{3}} +\frac{B_3 }{t^{N-2s+1}}
=o(1),\quad t\in [L_0, L_1].
\end{equation}

Let
$$
F(t, \bar r, \bar y'') =
\bigl(  \nabla_{\bar r,\bar y''} (K(\bar r, \bar {y}'')),
-\frac{B_{1}}{t^{3}}+\frac{B_3 }{t^{N-2s+1}}
\bigr).
$$
 Then
\[
\deg\bigl( F(t, \bar r, \bar y''), [ L_0, L_1]\times B_\theta ((r_0, y_0'')) )
=-\deg\bigl( \nabla_{\bar r, \bar y''} (K(\bar r,\bar y'')),  B_\theta ((r_0, y_0'')) )\ne 0.
\]
So, \eqref{71-18-1}, \eqref{72-18-1} and \eqref{74-18-1} have a solution  $t_m\in [ L_0, L_1]$, $(\bar r_m, \bar y''_m)\in B_\theta ((r_0, y_0''))$.
\end{proof}

\appendix

\section{Pohozaev Identities }\label{sa}
For the readers' convenient, we give the detailed proof of some local Pohozaev identities.
Note that if $u_{\epsilon}$ satisfies \eqref{nonlinearequation}, which is equivalent to
$\tilde{u}_{\epsilon}$ satisfies
\begin{equation}\label{l-e}
\left\{
\begin{array}{ll}
div(t^{1-2s}\nabla\tilde{u}_\epsilon)=0\quad  \text{in}\, \R_{+}^{N+1},
\\
-\lim_{t\rightarrow 0} t^{1-2s}\partial_t
\tilde{u}_\epsilon(x,t)=K(r,y'')(u_\epsilon)_{+}^{2^{*}_{s}-1+\epsilon}+\ds\sum_{l=1}^{N}c_{l}\ds\sum_{j=1}^{m}U_{x_{j},\lambda}^{2^{*}_{s}-2}Z_{j,l},\quad  \text{in}\, \R^{N}.
\end{array}
\right.
\end{equation}

First we have the following local Pohozaev identities by translations.
\begin{lemma}\label{le-p1}
If $\tilde{u}_\epsilon$ satisfies \eqref{l-e}, then there holds
\begin{equation}\label{p1}
\begin{split}
&\int_{\partial''\mathcal{B}_{\rho}^{+}(y_0)}t^{1-2s}\frac{\partial\tilde{u}_\epsilon}{\partial \nu}\frac{\partial\tilde{u}_\epsilon}{\partial y_i} dS
+\frac{1}{2^{*}_{s}+\epsilon}\int_{\partial B_{\rho}(y_0)}K(y)(u_\epsilon)_{+}^{2^{*}_{s}+\epsilon}\nu_i ds
\\
&-\frac{1}{2^{*}_{s}+\epsilon}\int_{B_{\rho}(y_0)}\frac{\partial K(y)}{\partial y_i}(u_\epsilon)_{+}^{2^{*}_{s}+\epsilon} dy
-\frac{1}{2}\int_{\partial''\mathcal{B}_{\rho}^{+}(y_0)}t^{1-2s}\nabla|\tilde{u}_\epsilon|^2 \nu_idS
\\
&
+\int_{B_{\rho}(y_0)}\sum_{l=1}^{N}c_{l}\ds\sum_{j=1}^{m} U_{x_{j},\lambda}^{2^{*}_{s}-2}Z_{j,l}\frac{\partial u_\epsilon}{\partial y_i}dy=0,\,\,i=3,\cdots,N.
\end{split}
\end{equation}
\end{lemma}

\begin{proof}
Noting that $\tilde{u}_\epsilon$ satisfies \eqref{l-e}, then we have
\begin{equation}
\begin{split}
0&=\int_{\mathcal{B}_{\rho}^{+}(y_0)}div(t^{1-2s}\nabla\tilde{u}_\epsilon)\frac{\partial\tilde{u}_\epsilon}{\partial y_i} dydt
\\
&=\int_{\partial\mathcal{B}_{\rho}^{+}(y_0)}t^{1-2s}\nabla\tilde{u}_\epsilon\frac{\partial\tilde{u}_\epsilon}{\partial y_i}\nu dS-\int_{\mathcal{B}_{\rho}^{+}(y_0)}t^{1-2s}\nabla\tilde{u}_\epsilon\nabla(\frac{\partial\tilde{u}_\epsilon}{\partial y_i})d ydt
\\
&=\int_{\partial''\mathcal{B}_{\rho}^{+}(y_0)}t^{1-2s}\frac{\partial\tilde{u}_\epsilon}{\partial y_i}\frac{\partial\tilde{u}_\epsilon}{\partial\nu} dS
+\int_{\partial'\mathcal{B}_{\rho}^{+}(y_0)}t^{1-2s}\nabla\tilde{u}_\epsilon\frac{\partial\tilde{u}_\epsilon}{\partial y_i}\nu dS
\\
&\quad-\int_{\mathcal{B}_{\rho}^{+}(y_0)}t^{1-2s}\nabla\tilde{u}_\epsilon\nabla(\frac{\partial\tilde{u}_\epsilon}{\partial y_i})dydt.
\end{split}
\end{equation}
and
\begin{equation}
\begin{split}
&\int_{\partial'\mathcal{B}_{\rho}^{+}(y_0)}t^{1-2s}\nabla\tilde{u}_\epsilon\frac{\partial\tilde{u}_\epsilon}{\partial y_i}\nu dS
\\
&=-	\int_{\partial'\mathcal{B}_{\rho}^{+}(y_0)}t^{1-2s}\partial_t\tilde{u}_\epsilon\frac{\partial\tilde{u}_\epsilon}{\partial y_i}dS\\
&=\int_{B_{\rho}(y_0)}\big[K(r,y'')(u_\epsilon)_{+}^{2^{*}_{s}-1+\epsilon}
+\ds\sum_{l=1}^{N}c_{l}\ds\sum_{j=1}^{m}U_{x_{j},\lambda}^{2^{*}_{s}-2}Z_{j,l}\big]\frac{\partial u_\epsilon}{\partial y_i}dy
\\
&=\frac{1}{2^{*}_{s}+\epsilon}\int_{\partial{B}_{\rho}(y_0)}K(y) (u_\epsilon)_{+}^{2^{*}_{s}+\epsilon}\nu_{i}\,ds
-\frac{1}{2^{*}_{s}+\epsilon}\int_{B_{\rho}(y_0)} \frac{\partial K}{\partial y_{i}}(u_\epsilon)_{+}^{2^{*}_{s}+\epsilon}dy
\\
&\quad+\int_{B_{\rho}(y_0)}\ds\sum_{l=1}^{N}c_{l}\ds\sum_{j=1}^{m}U_{x_{j},\lambda}^{2^{*}_{s}-2}Z_{j,l}\frac{\partial u_\epsilon}{\partial y_i}dy.
\end{split}
\end{equation}
Moreover we have
\begin{equation}
\begin{split}
&-\int_{\mathcal{B}_{\rho}^{+}(y_0)}t^{1-2s}\nabla\tilde{u}_\epsilon\nabla(\frac{\partial\tilde{u}_\epsilon}{\partial y_i})d ydt
=-\frac{1}{2}\int_{\mathcal{B}_{\rho}^{+}(y_0)}t^{1-2s}
\frac{\partial}{\partial y_i}\nabla|\tilde{u}_\epsilon|^2dydt
\\
&=-\frac{1}{2}\int_{\partial\mathcal{B}_{\rho}^{+}(y_0)}t^{1-2s}\nabla|\tilde{u}_\epsilon|^2 \nu_idS
\\
&=-\frac{1}{2}\int_{\partial'\mathcal{B}_{\rho}^{+}(y_0)}t^{1-2s}\nabla|\tilde{u}_\epsilon|^2 \nu_idS
-\frac{1}{2}\int_{\partial''\mathcal{B}_{\rho}^{+}(y_0)}t^{1-2s}\nabla|\tilde{u}_\epsilon|^2 \nu_idS,\quad i=1,\cdots,n
\\
&=-\frac{1}{2}\int_{\partial''\mathcal{B}_{\rho}^{+}(y_0)}t^{1-2s}\nabla|\tilde{u}_\epsilon|^2 \nu_idS,
\end{split}
\end{equation}
 where we use the fact that on $\partial'\mathcal{B}_{\rho}^{+}(y_0)$, $\nu_i=0$, so the first term equals to 0.

Combining all the equations above, we can get \eqref{p1}.
\end{proof}

\medskip
Next, we will obtain a local Pohozaev identity by scaling.

\begin{lemma}\label{le-p2}
If $\tilde{u}_\epsilon$ satisfies \eqref{l-e}, then there holds
\begin{equation}\label{p2}
\begin{split}
&\int_{\partial''\mathcal{B}_{\rho}^{+}(y_0)} t^{1-2s}\langle Y,\nabla\tilde{u}_\epsilon \rangle\frac{\partial\tilde{u}_\epsilon}{\partial \nu}dS
-\frac{1}{2}\int_{\partial''\mathcal{B}_{\rho}^{+}(y_0)}t^{1-2s}|\nabla\tilde{u}_\epsilon|^2\langle Y,\nu\rangle dS
\\
&+\frac{1}{2^{*}_{s}+\epsilon} \int_{\partial B_{\rho}(y_0)} K(r,y'')(u_\epsilon)_{+}^{2^{*}_{s}+\epsilon}\langle y,\nu\rangle ds
-\frac{N}{2^{*}_{s}+\epsilon}\int_{ B_{\rho}(y_0)}K(r,y'') (u_\epsilon)_{+}^{2^*_{s}+\epsilon} dy
\\
&
-\frac{1}{2^{*}_{s}+\epsilon}
\int_{B_{\rho}(y_0)}\langle \nabla K,y\rangle(u_\epsilon)_{+}^{2^*_{s}+\epsilon} dy
+\frac{N-2s}{2}\int_{\mathcal{B}_{\rho}^{+}(y_0)}t^{1-2s}|\nabla\tilde{u}_\epsilon|^{2} dS\\
&
+\int_{B_{\rho}(y_0)}\ds\sum_{l=1}^{N}c_{l}\sum_{j=1}^{m}U_{x_{j},\lambda}^{2^{*}_{s}-2}Z_{j,l}
\langle y,\nabla u_\epsilon\rangle dy
=0.
\end{split}
\end{equation}
\end{lemma}

\begin{proof}
Since
\begin{equation}
\begin{split}
div (t^{1-2s}\nabla\tilde{u}_\epsilon )\langle Y,\nabla\tilde{u}_\epsilon\rangle
\,=\,&div \Big(t^{1-2s}\nabla\tilde{u}_\epsilon\langle Y,\nabla\tilde{u}_\epsilon\rangle-t^{1-2s}Y\frac{|\nabla\tilde{u}_\epsilon|^2}{2}\Big)
\\
&-t^{1-2s}|\nabla\tilde{u}_\epsilon|^2
+\frac{div(t^{1-2s}Y)}{2}|\nabla\tilde{u}_\epsilon|^2,
\end{split}
\end{equation}
 we have
\begin{align*}
 0=&\int_{\mathcal{B}_{\rho}^{+}(y_0)}
 div (t^{1-2s}\nabla\tilde{u}_\epsilon )\langle Y,\nabla\tilde{u}_\epsilon\rangle dy dt
 \\
 =&\int_{\mathcal{B}_{\rho}^{+}(y_0)}div (t^{1-2s}\nabla\tilde{u}_\epsilon\langle Y,\nabla\tilde{u}_\epsilon\rangle-t^{1-2s}Y\frac{|\nabla\tilde{u}_\epsilon|^2}{2})dydt
 \\
&-\int_{\mathcal{B}_{\rho}^{+}(y_0)} \Big[-t^{1-2s}|\nabla\tilde{u}_\epsilon|^2
+\frac{div(t^{1-2s}Y)}{2}|\nabla\tilde{u}_\epsilon|^2\Big] dydt.
\end{align*}
It is easy to derive that
\begin{equation}
\begin{split}
&\int_{\mathcal{B}_{\rho}^{+}(y_0)}div (t^{1-2s}\nabla\tilde{u}_\epsilon\langle Y,\nabla\tilde{u}_\epsilon\rangle-t^{1-2s}Y\frac{|\nabla\tilde{u}_\epsilon|^2}{2})dydt
\\
&=\int_{\partial\mathcal{B}_{\rho}^{+}(y_0)} \big[t^{1-2s}\nabla\tilde{u}_\epsilon\langle Y,\nabla\tilde{u}_\epsilon\rangle\nu
-t^{1-2s}Y\frac{|\nabla\tilde{u}_\epsilon|^2}{2}\nu  \big]dS
\\
&=\int_{\partial''\mathcal{B}_{\rho}^{+}(y_0)}\big[ t^{1-2s}\nabla\tilde{u}_\epsilon\langle Y,\nabla\tilde{u}_\epsilon\rangle\nu
-t^{1-2s}Y\frac{|\nabla\tilde{u}_\epsilon|^2}{2}\nu \big] dS
\\
&\quad+\int_{\partial'\mathcal{B}_{\rho}^{+}(y_0)} \big[t^{1-2s}\nabla\tilde{u}_\epsilon\langle Y,\nabla\tilde{u}_\epsilon\rangle\nu
-t^{1-2s}Y\frac{|\nabla\tilde{u}_\epsilon|^2}{2}\nu \big] dS
\\
&:=A_1+A_2.
\end{split}
\end{equation}
Noting that
$$
\partial'\mathcal{B}_{\rho}^{+}(y_0)=\{Y=(y,t):|Y-(y_0,0)|\leq  \rho \quad and \quad t=0\}\in \R^{N},
$$
and $-t^{1-2s}Y\frac{|\nabla\tilde{u}_\epsilon|^2}{2}\nu$ is odd about $y$,
we have
\begin{equation}
\begin{split}
A_2&=\int_{\partial'\mathcal{B}_{\rho}^{+}(y_0)} t^{1-2s}\nabla\tilde{u}_\epsilon\langle Y,\nabla\tilde{u}_\epsilon\rangle\nu  dS
=\int_{\partial'\mathcal{B}_{\rho}^{+}(y_0)} -t^{1-2s}\partial_t \tilde{u}_\epsilon\langle Y,\nabla \tilde{u}_\epsilon\rangle dS
 \\
 &=\int_{{B}_{\rho}(y_0)}\Big[K(r,y'')(u_\epsilon)_{+}^{2^{*}_{s}-1+\epsilon}
 +\ds\sum_{l=1}^{N}c_{l}\ds\sum_{j=1}^{m}Z_{x_{j},\lambda}^{2^{*}_{s}-2}Z_{j,l}\Big]\langle y,\nabla u_\epsilon\rangle dy
 \\
 &=\frac{1}{2^{*}_{s}+\epsilon}\int_{\partial B_{\rho}(y_0)}K(r,y'') (u_\epsilon)_{+}^{2^{*}_{s}+\epsilon}\langle y,\nu\rangle ds
 -\frac{1}{2^{*}_{s}+\epsilon}\int_{B_{\rho}(y_0)}\langle y,\nabla K\rangle(u_\epsilon)_{+}^{2^{*}_{s}+\epsilon} dy
 \\
 &
\quad -\frac{N}{2^*_{s}+\epsilon}\int_{B_{\rho}(y_0)}K(r,y'')(u_\epsilon)_{+}^{2^{*}_{s}+\epsilon} dy
 +\int_{B_{\rho}(y_0)}\sum_{l=1}^{N}c_{l}\ds\sum_{j=1}^{m}U_{x_{j},\lambda}^{2^{*}_{s}-2}Z_{j,l}\langle y,\nabla u_\epsilon\rangle dy.
 \end{split}
\end{equation}

Direct calculates will give that
\begin{equation}
\begin{split}
\frac{div(t^{1-2s}Y)}{2}|\nabla\tilde{u}_\epsilon|^2=&\frac{t^{1-2s}+t^{1-2s}+\cdots+(2-2s)t^{1-2s}}{2}|\nabla\tilde{u}_\epsilon|^2
\\
=&\frac{N+2-2s}{2}t^{1-2s}|\nabla\tilde{u}_\epsilon|^2.
\end{split}
\end{equation}

Then it follows from all the equations above that
\begin{equation}\label{p22}
\begin{split}
&\int_{\partial''\mathcal{B}_{\rho}^{+}(y_0)} t^{1-2s}\nabla\tilde{u}_\epsilon\langle Y,\nabla\tilde{u}_\epsilon\rangle\nu dS
-\frac{1}{2}\int_{\partial''\mathcal{B}_{\rho}^{+}(y_0)}t^{1-2s}|\nabla\tilde{u}_\epsilon|^2\langle Y,\nu\rangle dS
\\
&+\frac{1}{2^{*}_{s}+\epsilon} \int_{\partial B_{\rho}(y_0)} K(r,y'')(u_\epsilon)_{+}^{2^{*}_{s}+\epsilon}\langle y,\nu\rangle ds
-\frac{N}{2^{*}_{s}+\epsilon}\int_{ B_{\rho}(y_0)}K(r,y'') (u_\epsilon)_{+}^{2^*_{s}+\epsilon} dy
\\
&
-\frac{1}{2^{*}_{s}+\epsilon}
\int_{B_{\rho}(y_0)}\langle \nabla K,y\rangle(u_\epsilon)_{+}^{2^*_{s}+\epsilon} dy
+\frac{N-2s}{2}\int_{\mathcal{B}_{\rho}^{+}(y_0)}t^{1-2s}|\nabla\tilde{u}_\epsilon|^2 dydt
\\
&+\int_{B_{\rho}(y_0)}\ds\sum_{l=1}^{N}c_{l}\ds\sum_{j=1}^{m}U_{x_{j},\lambda}^{2^{*}_{s}-2}Z_{j,l}\langle y,\nabla u_\epsilon\rangle dy\,=\,0.
\end{split}
\end{equation}
\end{proof}

%
%


%
\section{{Basic Estimates}}\label{s4}

For each fixed $k$ and $j$, $k\neq j$, we consider the following
function
\begin{equation}\label{b.1}
g_{k,j}(y)=\frac{1}{(1+|y-x_{j}|)^{\alpha}}\frac{1}{(1+|y-x_{k}|)^{\beta}},
\end{equation}
where $\alpha\geq 1$ and $\beta\geq 1$ are two constants.
 \begin{lem}\label{lemb1}(Lemma B.1, \cite{WY1})
 For any constants $0<\delta\leq \min\{\alpha,\beta\}$, there is a constant $C>0$, such that
 $$
 g_{k,j}(y)\leq \frac{C}{|x_{k}-x_{j}|^{\delta}}\Big(\frac{1}{(1+|y-x_{k}|)^{\alpha+\beta-\delta}}+\frac{1}{(1+|y-x_{j}|)^{\alpha+\beta-\delta}}\Big).
 $$
 \end{lem}

\begin{lem}\label{lemb2}(Lemma 2.1, \cite{GN-16-DCDS})
For any constant $0<\delta<N-2s$, there is a constant $C>0$, such
that
$$
\int_{\R^N} \frac{1}{|y-z|^{N-2s}}\frac{1}{(1+|z|)^{2s+\delta}}dz\leq
\frac{C}{(1+|y|)^{\delta}}.
$$
\end{lem}

\begin{lem}\label{lembb5}(Lemma A.5, \cite{gln-19})
Suppose that $(y-x)^{2}+t^{2}=\rho^{2},t>0.$ Then there exists a constant $C>0$ such that
$$
|\tilde{U}_{x_{j},\lambda}|\leq \frac{C}{\lambda^{\frac{N-2s}{2}}(1+|y-x_{j}|)^{N-2s}}\,\,\,
\text{and}\,\,|\nabla \tilde{U}_{x_{j},\lambda}|\leq \frac{C}{\lambda^{\frac{N-2s}{2}}(1+|y-x_{j}|)^{N-2s+1}}.
$$
\end{lem}

\begin{lem}\label{lembb6}(Lemma A.6, \cite{gln-19})
For any $\delta>0,$ there is a $\rho=\rho(\delta)\in (2\delta,5\delta)$ such that
$$
\int_{\partial'' \mathcal{B}_{\rho}^{+}}t^{1-2s}|\nabla \tilde{\varphi}|^{2} dydt\leq  C\frac{m}{\lambda^{\tau}}\|\varphi\|^{2}_{*},
$$
where $C>0$ is a constant, depending on $\delta.$
\end{lem}

From the proof of Lemma A.6 in \cite{gln-19}, we have
\begin{lem}\label{lembb7}
There exists a positive constant C such that
\begin{equation}\label{lem-b7}
|\tilde{\varphi}(y,t)|\leq \frac{C\|\varphi\|_{*}}{\lambda^{\tau}}\sum_{j=1}^{m}\frac{1}{(1+|y-x_{j}|)^{\frac{N-2s}{2}+\tau}}.
\end{equation}
\end{lem}

Let us recall that
$$
Z_{\bar{r},\bar{y}'',\lambda}(y)=\sum_{j=1}^{m} U_{x_{j},\lambda}=(4^s\gamma)^{\frac{N-2s}{4s}}\sum_{j=1}^{m}
\Big(\frac{\lambda}{1+\lambda^{2}|y-x_{j}|^{2}}\Big)^{\frac{N-2s}{2}}.
$$

\begin{lem}\label{lemb3}
 There is a small constant $\sigma>0$,
such that
$$
\int_{\R^{N}}
\frac{1}{|y-z|^{N-2s}}Z^{\frac{4s}{N-2s}+\epsilon}_{\bar r,\bar{y}'',\lambda}(z)\sum_{j=1}^{m}\frac{1}{(1+\lambda|z-x_{j}|)^{\frac{N-2s}{2}+\tau}}dz
\leq
\sum_{j=1}^{m}\frac{C}{(1+\lambda|y-x_{j}|)^{\frac{N-2s}{2}+\tau+\sigma}}.
$$
\end{lem}

\begin{proof}
Here we prove it by some different arguments from Lemma B.3 of \cite{WY1} and Lemma 2.2 of \cite{GN-16-DCDS}.
By direct computations and Remark \ref{re1}, we have
\begin{align*}
&\int_{\R^{N}} \frac{1}{|y-z|^{N-2s}}Z^{\frac{4s}{N-2s}+\epsilon}_{\bar r,\bar{y}'',\lambda}(z)\sum_{j=1}^{m}\frac{1}{(1+\lambda|z-x_{j}|)^{\frac{N-2s}{2}+\tau}}dz\\
\leq \,&
C\int_{\R^{N}} \frac{\lambda^{2s+\frac{N-2s}{2}\epsilon}}{|y-z|^{N-2s}}\Big(\sum_{i=1}^{m}\frac{1}{(1+\lambda|z-x_{i}|)^{N-2s}}
\Big)^{\frac{4s}{N-2s}+\epsilon}\sum_{j=1}^{m}\frac{1}{(1+\lambda|z-x_{j}|)^{\frac{N-2s}{2}+\tau}}\\
\leq\, &
C\sum_{i=1}^{m}\int_{\R^{N}} \frac{\lambda^{2s}}{|y-z|^{N-2s}}\frac{1}{(1+\lambda|z-x_{i}|)^{\frac{N+6s}{2}+\tau+(N-2s)\epsilon}}
\\
&+\sum_{i=1}^{m}\int_{\R^{N}} \frac{C\lambda^{2s}}{|y-z|^{N-2s}}
\Big(\sum_{j\neq i,j=1}^{m}\frac{1}{(1+\lambda|z-x_{i}|)^{N-2s}}
\frac{1}{(1+\lambda|z-x_{j}|)^{(\frac{N-2s}{2}+\tau)\frac{N-2s}{4s+(N-2s)\epsilon}}}
\Big)^{\frac{4s}{N-2s}+\epsilon}\\
\leq\, &\sum_{j=1}^{m}\frac{1}{(1+\lambda|y-x_{j}|)^{\frac{N-2s}{2}+\tau+\sigma}},
\end{align*}
since it follow from Lemma \ref{lemb2} that
\begin{align*}
&\sum_{i=1}^{m}\int_{\R^{N}} \frac{\lambda^{2s}}{|y-z|^{N-2s}}\frac{1}{(1+\lambda|z-x_{i}|)^{\frac{N+6s}{2}+\tau+(N-2s)\epsilon}}\, dz\\
\leq\,&\sum_{i=1}^{m}\int_{\R^{N}} \frac{\lambda^{2s}}{|y-z|^{N-2s}}\frac{1}{(1+\lambda|z-x_{i}|)^{2s+(\frac{N-2s}{2}+\tau+2s-(N-2s)\epsilon)}}\, dz\\
\leq\, &\sum_{j=1}^{m}\frac{1}{(1+\lambda|y-x_{j}|)^{\frac{N-2s}{2}+\tau+\sigma}},
\end{align*}
and
\begin{align*}
&\sum_{i=1}^{m}\int_{\R^{N}} \frac{\lambda^{2s}}{|y-z|^{N-2s}}
\Big(\sum_{j\neq i,j=1}^{m}\frac{1}{(1+\lambda|z-x_{i}|)^{N-2s}}
\frac{1}{(1+\lambda|z-x_{j}|)^{(\frac{N-2s}{2}+\tau)\frac{N-2s}{4s+(N-2s)\epsilon}}}
\Big)^{\frac{4s}{N-2s}+\epsilon}\, dz\\
\leq& \sum_{i=1}^{m}\int_{\R^{N}} \frac{\lambda^{2s}}{|y-z|^{N-2s}}
\sum_{j\neq i,j=1}^{m}\frac{1}{|\lambda(x_{j}-x_{i})|^{2\tau}}
\Big(\frac{1}{(1+\lambda|z-x_{i}|)^{N-2s+(\frac{N-2s}{2}+\tau)\frac{N-2s}{4s+(N-2s)\epsilon}-2\tau}}
\\
&\quad\quad+\frac{1}{(1+\lambda|z-x_{j}|)^{N-2s+(\frac{N-2s}{2}+\tau)\frac{N-2s}{4s+(N-2s)\epsilon}-2\tau}}
\Big)^{\frac{4s}{N-2s}+\epsilon}\, dz
\\
\leq& C\sum_{i=1}^{m}\int \frac{\lambda^{2s}}{|y-z|^{N-2s}}
\frac{1}{(1+\lambda|z-x_{i}|)^{2s+(\frac{N-2s}{2}+\tau)+2s-\tau\frac{4s}{N-2s}+(N-2s-\tau)\epsilon}} \, dz
\\
\leq &C\sum_{j=1}^{m}\frac{1}{(1+\lambda|y-x_{j}|)^{\frac{N-2s}{2}+\tau+\sigma}},
\end{align*}
where we used $2s-\tau\frac{4s}{N-2s}>0.$
\end{proof}

\begin{lem} \label{lema.2}
If $N\geq 4$ and $0<s<1$, then
$$
\frac{\partial I(Z_{\bar r,\bar{y}'',\lambda})}{\partial
\lambda}=m\Big(-\frac{B_{1}}{\lambda^{3}}
+\sum_{j=2}^{m}\frac{B_{2}}{\lambda^{N-2s+1}|x_{1}-x_{j}|^{N-2s}}
+O\Big(\epsilon^{\frac{3+\iota}{N-2s}}\Big)\Big),
$$
where   $B_{j},j=1,2$ are some positive constants.
\end{lem}

\begin{proof}

Direct calculations show that
\begin{align*}
\ds\frac{\partial I(Z_{\bar r,\bar{y}'',\lambda})}{\partial\lambda}
=\,&\int_{\R^N}
\sum_{j=1}^mU_{x_j,\lambda}^{2_{s}^{*}-1}\frac{\partial Z_{\bar r,\bar{y}'',\lambda}}{\partial\lambda}dy
-\ds\int_{\R^N} K(y) (Z_{\bar r,\bar{y}'',\lambda})^{2_{s}^{*}-1+\epsilon}\frac{\partial Z_{\bar r,\bar{y}'',\lambda}}{\partial\lambda}dy
\\
=\,&-\int_{\R^N}\Big[(Z_{\bar r,\bar{y}'',\lambda})^{2_{s}^{*}-1}-
\sum_{j=1}^mU_{x_j,\lambda}^{2_{s}^{*}-1}\Big]\frac{\partial Z_{\bar r,\bar{y}'',\lambda}}{\partial\lambda}dy
\\
&+\int_{\R^N}\Big[1-K(y)\Big] (Z_{\bar r,\bar{y}'',\lambda})^{2_{s}^{*}-1+\epsilon}\frac{\partial Z_{\bar r,\bar{y}'',\lambda}}{\partial\lambda}dy
\\
&-\int_{\R^N}\Big[(Z_{\bar r,\bar{y}'',\lambda})^{2_{s}^{*}-1+\epsilon}-(Z_{\bar r,\bar{y}'',\lambda})^{2_{s}^{*}-1}
\Big]\frac{\partial Z_{\bar r,\bar{y}'',\lambda}}{\partial\lambda}dy
\\
:=\,&-F_{1}+F_{2}-F_{3}.
\end{align*}

First, we have
\begin{equation}\label{f1}
\begin{split}
F_{1}=\,&\int_{\R^N} \Bigl( (Z_{\bar r,\bar{y}'',\lambda})^{2_{s}^{*}-1}-\sum_{j=1}^mU_{x_j,\lambda}^{2_{s}^{*}-1}
\Bigr)\frac{\partial Z_{\bar r,\bar{y}'',\lambda}}{\partial\lambda}\,dy\\
=\,& m \int_{\Omega_1} \Bigl( (Z_{\bar r,\bar{y}'',\lambda})^{2_{s}^{*}-1}-\sum_{j=1}^mU_{x_j,\lambda}^{2_{s}^{*}-1}
\Bigr)\frac{\partial Z_{\bar r,\bar{y}'',\lambda}}{\partial\lambda} \,dy\\
=\,& m \Bigl( \int_{\Omega_1} (2_{s}^{*}-1) U_{x_1,\lambda}^{2_{s}^{*}-2}\sum_{j=2}^m U_{x_j,\lambda}\frac{\partial U_{x_1,\lambda}}{\partial\lambda}\,dy
+O\big(\epsilon^{\frac{3+\iota}{N-2s}}\bigr)\Bigr)
\\
=\,& m \Bigl( - \sum_{j=2}^{m}\frac{B_{2}}{\lambda^{N-2s+1}|x_{1}-x_{j}|^{N-2s}}+O\big(\epsilon^{\frac{3+\iota}{N-2s}}\bigr)\Bigr)
\end{split}
\end{equation}
for some constant $B_2>0$.

Noting that $K(y_{0})=1,$ it is easy to check that
\begin{align}\label{qq}
F_{2}=&\int_{\R^N}\big(1-K(y)\big)(Z_{\bar r,\bar{y}'',\lambda})^{2_{s}^{*}-1+\epsilon}\frac{\partial Z_{\bar r,\bar{y}'',\lambda}}{\partial\lambda} \,dy\nonumber\\
=&
m\Big(\int_{\R^N}\big(1-K(y)\big) U^{2_{s}^{*}-1+\epsilon}_{x_1,\lambda}\frac{\partial  U_{x_1,\lambda}}{\partial \lambda} dy +O\Big(\frac1\lambda \int_{\R^N}
U^{2_{s}^{*}-1+\epsilon}_{x_1,\lambda}\sum\limits_{j=2}^m
U_{x_j,\lambda}dy\Big)\Big)\nonumber
\\
=&
m\Big(-\int_{B_{\epsilon^{\frac{\frac{1}{2}+\iota}{N-2s}}}(y_{0})}
\Big(\sum_{i,j=1}^{N}
\frac{1}{2}\frac{\partial ^{2}K(y_{0})}{\partial y_{i}\partial y_{j}}(y_{i} -y_{0i})(y_{j} -y_{0j})\nonumber
\\
&\quad\quad\quad\quad+
\ds\sum_{i,j,k=1}^{N}\frac{1}{6}
\frac{\partial ^{3}K(y_{0}+\vartheta(y-y_{0}))}{\partial y_{i}\partial y_{j}\partial y_{k}}(y_{i} -y_{0i})
(y_{j} -y_{0j})(y_{k} -y_{0k})\Big)\frac{1}{2_{s}^{*}+\epsilon}\frac{\partial  U^{2_{s}^{*}+\epsilon}_{x_1,\lambda}}{\partial \lambda}dy\nonumber
\\
&\quad-\int_{B^{C}_{\epsilon^{\frac{\frac{1}{2}+\iota}{N-2s}}}(y_{0})}\big(K(y)-1\big) \frac{1}{2_{s}^{*}+\epsilon}\frac{\partial  U^{2_{s}^{*}+\epsilon}_{x_1,\lambda}}{\partial \lambda}dy
+O\big(\epsilon^{\frac{3+\iota}{N-2s}}\big)\Big)\nonumber
\\
=&m\Big(-\int_{B_{\frac{3}{2}\epsilon^{\frac{\frac{1}{2}+\iota}{N-2s}}}(x_{1})}
\Big(\sum_{i,j=1}^{N}
\frac{1}{2}\frac{\partial ^{2}K(y_{0})}{\partial y_{i}\partial y_{j}}(y_{i}-y_{0i})(y_{j}-y_{0j})+O(\|y-y_{0}\|^{3})\Big)
\frac{1}{2_{s}^{*}+\epsilon}\frac{\partial  U^{2_{s}^{*}+\epsilon}_{x_1,\lambda}}{\partial \lambda}dy \nonumber
\\
&\quad+O\Big(\int_{B^{C}_{\frac{1}{2}\epsilon^{\frac{\frac{1}{2}+\iota}{N-2s}}}(x_{1})}\frac{ U^{2_{s}^{*}+\epsilon}_{x_1,\lambda}}{\lambda}dy\Big)
+O\big(\epsilon^{\frac{3+\iota}{N-2s}}\big)\Big)\nonumber
\\
=&m\Big(-\int_{B_{\frac{3}{2}\epsilon^{\frac{\frac{1}{2}+\iota}{N-2s}}(0)}}
\Big(\sum_{i,j=1}^{N}
\frac{1}{2}\frac{\partial ^{2}K(y_{0})}{\partial y_{i}\partial y_{j}}(z_{i}+x_{1i} -y_{0i})(z_{j}+x_{1j} -y_{0j}) \nonumber
\\
&\quad\quad\quad\quad\quad\quad\quad\quad\quad+O(\|z+x_{1}-y_{0}\|^{3})\Big)
\frac{1}{2_{s}^{*}+\epsilon}\frac{\partial  U^{2_{s}^{*}+\epsilon}_{0,\lambda}}{\partial \lambda}dz
+O\big(\epsilon^{\frac{3+\iota}{N-2s}}\big)\Big)\nonumber
\\
=&\ds m\Big(-\frac{1}{2_{s}^{*}+\epsilon}\frac{\partial }{\partial\lambda}
\int
\sum_{i,j=1}^{N}\frac{1}{2}\frac{\partial ^{2}K(y_{0})}{\partial y_{i}\partial y_{j}}\Big(\frac{z_{i}}{\lambda}+x_{1i}-y_{0i}\Big)
\Big(\frac{z_{j}}{\lambda}+x_{1j}-y_{0j}\Big)U^{2_{s}^{*}+\epsilon}_{0,1}dz \nonumber
\\
&-\int_{B^{C}_{\frac{3}{2}\epsilon^{\frac{\frac{1}{2}+\iota}{N-2s}}}(0)}
\sum_{i,j=1}^{N}\frac{1}{2}\frac{\partial ^{2}K(y_{0})}{\partial y_{i}\partial y_{j}}\Big(z_{i}+x_{1i}-y_{0i}\Big)
\Big(z_{j}+x_{1j}-y_{0j}\Big)\frac{1}{2^{*}}\frac{\partial  U^{2_{s}^{*}+\epsilon}_{0,\lambda}}{\partial \lambda}dz \nonumber
\\
&+\ds\int O(\|\frac{z}{\lambda}+x_{1}-y_{0}\|^{3}) \frac{U^{2_{s}^{*}+\epsilon}_{0,1}}{\lambda}\Big)+O\big(\epsilon^{\frac{3+\iota}{N-2s}}\big)\Big) \nonumber
\\
=&\ds l\Big(-\frac{1}{2_{s}^{*}+\epsilon}\frac{\partial}{\partial \lambda}
\int \frac{1}{2\lambda^{2}}\sum_{i=1}^{N}\frac{\partial^{2}K(y_{0})}{\partial y^{2}_{i}}z^{2}_{i}U^{2_{s}^{*}+\epsilon}_{0,1}dz
+O\Big(\epsilon^{\frac{3+\iota}{N-2s}}\Big)\Big) \nonumber
\\
=&\ds l\Big(\frac{1}{2_{s}^{*}+\epsilon}\frac{1 }{\lambda^{3}}
\int\sum_{i=1}^{N}\frac{\partial ^{2}K}{\partial y^{2}_{i}}(y_{0})
z^{2}_{i}U^{2_{s}^{*}+\epsilon}_{0,1}dz+O\big(\epsilon^{\frac{3+\iota}{N-2s}}\big)\Big) \nonumber
\\
=&\ds l\Big(\frac{1}{2_{s}^{*}+\epsilon}\frac{1 }{\lambda^{3}}\frac{\Delta K(y_{0})}{N}
\int z^{2}U^{2_{s}^{*}+\epsilon}_{0,1}dz+O\Big(\epsilon^{\frac{3+\iota}{N-2s}}\Big)\Big)
=
l\Big(-\frac{B_{1}}{\lambda^{3}}+O\big(\epsilon^{\frac{3+\iota}{N-2s}}\big)\Big),
\end{align}
where we have used \eqref{r} and \eqref{q1}.

Finally similar to \eqref{j11}, since $N>2s+2,$ we have
\begin{align}\label{f2}
F_{3}&=\int_{\R^N}\Big[(Z_{\bar r,\bar{y}'',\lambda})^{2_{s}^{*}-1+\epsilon}-(Z_{\bar r,\bar{y}'',\lambda})^{2_{s}^{*}-1}
\Big]\frac{\partial Z_{\bar r,\bar{y}'',\lambda}}{\partial\lambda} \,dy \nonumber
\\
&\leq \frac{C}{\lambda}\epsilon \ln\frac{1}{\epsilon}  \int_{\R^N}\sum_{j=1}^{m}\frac{\lambda^{\frac{N+2s}{2}}}{(1+\lambda|y-x_{j}|)^{\frac{N+2s}{2}+\tau}}\frac{\lambda^{\frac{N-2s}{2}}}{\lambda(1+\lambda|y-x_{j}|)^{N-2s}} \,dy\nonumber
\\
& \leq  C\epsilon^{\frac{2+\iota}{N-2s}} \int_{\R^N}\sum_{j=1}^{m}\frac{\lambda^{\frac{N+2s}{2}}}{(1+\lambda|y-x_{j}|)^{\frac{N+2s}{2}+\tau}}
\frac{\lambda^{\frac{N-2s}{2}}}{(1+\lambda|y-x_{j}|)^{N-2s}} \,dy\nonumber
\\
& \leq  Cm\epsilon^{\frac{3+\iota}{N-2s}}.
\end{align}

So, we obtain
\begin{eqnarray*}
\frac{\partial I(Z_{\bar r,\bar{y}'',\lambda})}{\partial
\lambda}=m\Big(-\frac{B_{1}}{\lambda^{3}}
+\sum_{j=2}^{m}\frac{B_{2}}{\lambda^{N-2s+1}|x_{1}-x_{j}|^{N-2s}}
+O\big(\epsilon^{\frac{3+\iota}{N-2s}}\big)\Big).
\end{eqnarray*}
\end{proof}

\section{{Proof of \eqref{2.9.5}}}\label{sd}
In this section, we mainly prove \eqref{2.9.5}.
\begin{proof}
Note that
\begin{align}\label{2.9.5-1}
&\bigl\langle \ds(-\Delta)^{s} \varphi_\epsilon
-(2_{s}^{*}-1+\epsilon)K(r,y'')Z_{\bar r,\bar{y}'',\lambda}^{2_{s}^{*}-2+\epsilon}
\varphi_\epsilon,Z_{1,k}\bigr\rangle \nonumber
\\[1mm]
=\,&\big\langle(-\Delta)^{s} Z_{1,k}, \varphi_\epsilon\big\rangle
-(2_{s}^{*}-1+\epsilon)\big\langle K(r,y'')Z_{\bar r,\bar{y}'',\lambda}^{2_{s}^{*}-2+\epsilon}
\varphi_\epsilon,Z_{1,k}\bigr\rangle \nonumber
\\[1mm]
=\,&(2_{s}^{*}-1)\big\langle U_{x_{1},\lambda}^{2_{s}^{*}-2}Z_{1,k},\varphi_\epsilon\big\rangle
-(2_{s}^{*}-1+\epsilon)\big\langle K(r,y'')Z_{\bar r,\bar{y}'',\lambda}^{2_{s}^{*}-2+\epsilon}
\varphi_\epsilon,Z_{1,k}\bigr\rangle \nonumber
\\[1mm]
=\,&(2_{s}^{*}-1)\big\langle (U_{x_{1},\lambda}^{2_{s}^{*}-2}-Z_{\bar r,\bar{y}'',\lambda}^{2_{s}^{*}-2}
)Z_{1,k},\varphi_\epsilon\big\rangle
-(2_{s}^{*}-1)\big\langle (K(r,y'')-1)Z_{\bar r,\bar{y}'',\lambda}^{2_{s}^{*}-2}Z_{1,k},\varphi_\epsilon\big\rangle\nonumber
\\
&-(2_{s}^{*}-1)\big\langle K(r,y'')(Z_{\bar r,\bar{y}'',\lambda}^{2_{s}^{*}-2+\epsilon}-Z_{\bar r,\bar{y}'',\lambda}^{2_{s}^{*}-2})Z_{1,k},\varphi_\epsilon\big\rangle
-\epsilon \big\langle K(r,y'')Z_{\bar r,\bar{y}'',\lambda}^{2_{s}^{*}-2+\epsilon}Z_{1,k},\varphi_\epsilon\big\rangle \nonumber
\\[1mm]
:=\,&M_{1}-M_{2}-M_{3}-M_{4}.
\end{align}

First, we can rewrite $M_{1}$ as following
\begin{align*}
M_{1}\,=\,&(2_{s}^{*}-1) \int_{\R^N} \big( U_{x_{1},\lambda}^{2_{s}^{*}-2}-Z_{\bar r,\bar{y}'',\lambda}^{2_{s}^{*}-2} \big)\,Z_{1,k}\varphi_\epsilon \,dy
\\
\,=\,& (2_{s}^{*}-1) \int_{\Omega_1} \big( U_{x_{1},\lambda}^{2_{s}^{*}-2}-Z_{\bar r,\bar{y}'',\lambda}^{2_{s}^{*}-2} \big)\,Z_{1,k}\varphi_\epsilon \,dy
+(2_{s}^{*}-1) \int_{\Omega_1^{c}} \big( U_{x_{1},\lambda}^{2_{s}^{*}-2}-Z_{\bar r,\bar{y}'',\lambda}^{2_{s}^{*}-2} \big)\,Z_{1,k}\varphi_\epsilon \,dy
\\
:=\,&M_{11}\,+\,M_{12}.
\end{align*}
Next, we will estimates $M_1$ in the following two cases.

If $2_{s}^{*}>3,$ then $\frac{4s}{N-2s}>1$ and
\begin{align}\label{2.9.5-3}
M_{11}
&\leq C\|\varphi_\epsilon\|_{*}\int_{\Omega_{1}} \Big[\big(\sum_{i=2}^{m}U_{x_{i},\lambda}\big)^{2_{s}^{*}-2}
+U_{x_{1},\lambda}^{2_{s}^{*}-3}\sum_{i=2}^{m}U_{x_{i},\lambda}\Big]|Z_{1,k}|
\sum_{j=1}^{m}\frac{\lambda^{\frac{N-2s}{2}}}{(1+\lambda|y-x_{j}|)^{\frac{N-2s}{2}+\tau}} \nonumber
\\
&\leq C\|\varphi_\epsilon\|_{*}\lambda^{n_{k}}\int_{\Omega_{1}}\big(\sum_{i=2}^{m}\frac{1}{(1+\lambda|y-x_{i}|)^{N-2s}}\big)^{2_{s}^{*}-2}
\frac{\lambda^{N}}{(1+\lambda|y-x_{1}|)^{\frac{3(N-2s)}{2}+\tau}} \nonumber
\\
&\quad+C\|\varphi_\epsilon\|_{*}\lambda^{n_{k}}\int_{\Omega_{1}}\big(\sum_{i=2}^{m}\frac{1}{(1+\lambda|y-x_{i}|)^{N-2s}}\big)^{2_{s}^{*}-2}
\frac{1}{(1+\lambda|y-x_{1}|)^{N-2s}}
\sum_{j=2}^{m}\frac{\lambda^{N}}{(1+\lambda|y-x_{j}|)^{\frac{N-2s}{2}+\tau}} \nonumber
\\
&\quad+C\|\varphi_\epsilon\|_{*}\lambda^{n_{k}}\int_{\Omega_{1}}\frac{1}{(1+\lambda|y-x_{1}|)^{\frac{N+6s}{2}+\tau}}\sum_{j=2}^{m}
\frac{\lambda^N}{(1+\lambda|y-x_{j}|)^{N-2s}} \nonumber
\\
&\quad+C\|\varphi_\epsilon\|_{*}\lambda^{n_{k}}\int_{\Omega_{1}}\frac{1}{(1+\lambda|y-x_{1}|)^{4s}}\sum_{j=2}^{m}
\frac{\lambda^N}{(1+\lambda|y-x_{j}|)^{N-2s}}
\sum_{j=2}^{m}\frac{1}{(1+\lambda|y-x_{j}|)^{\frac{N-2s}{2}+\tau}} \nonumber
\\
&\leq C\|\varphi_\epsilon\|_{*}\lambda^{n_{k}}\int_{\Omega_{1}}\big(\sum_{i=2}^{m}\frac{1}{|\lambda(x_{1}-x_{i})|^{\frac{N-2s}{2}}}\big)^{2_{s}^{*}-2}
\frac{\lambda^{N}}{(1+\lambda|y-x_{1}|)^{\frac{3N-2s}{2}+\tau}} \nonumber
\\
&\quad+C\|\varphi_\epsilon\|_{*}\lambda^{n_{k}}\int_{\Omega_{1}}\big(\sum_{i=2}^{m}\frac{1}{|\lambda(x_{1}-x_{i})|^{\frac{N-2s}{2}}}\big)^{2_{s}^{*}-2}
\frac{1}{(1+\lambda|y-x_{1}|)^{N+\frac{N-2s}{2}}}\sum_{j=2}^{m}\frac{1}{|\lambda(x_{1}-x_{j})|^{\tau}} \nonumber
\\
&\quad+C\|\varphi_\epsilon\|_{*}\lambda^{n_{k}}\int_{\Omega_{1}}\sum_{j=2}^{m}\frac{1}{|\lambda(x_1-x_j)|^{\frac{N-2s}{2}+\tau}}\Big( \frac{\lambda^N}{(1+\lambda|y-x_{1}|)^{N+2s}} +\frac{\lambda^N}{(1+\lambda|y-x_{j}|)^{N+2s}}\Big)  \nonumber
\\
&\quad+C\|\varphi_\epsilon\|_{*}\lambda^{n_{k}}\int_{\Omega_{1}}\frac{1}{(1+\lambda|y-x_{1}|)^{\frac{N-2s}{2}+4s}}\sum_{j=2}^{m}
\frac{\lambda^N}{(1+\lambda|y-x_{j}|)^{N-2s}}
\sum_{j=2}^{m}\frac{1}{|\lambda(x_1-x_{j})|^{\tau}} \nonumber
\\
&\leq C\|\varphi_\epsilon\|_{*}\lambda^{n_{k}}\Big(\frac{m}{\lambda}\Big)^{2s}
\leq C\|\varphi_\epsilon\|_{*}\frac{\lambda^{n_{k}}}{\lambda^{\frac{4s}{N-2s}}}\leq  C\|\varphi_\epsilon\|_{*}\frac{\lambda^{n_{k}}}{\lambda^{1+\iota}}.
\end{align}
In the above, we have used the fact that $N>2+2s$.

Noting that $N>2+2s,\frac{4s}{N-2s}>1$ and $\tau<2s,$ by Lemma \ref{lemb1} and the discrete H\"{o}lder inequality we have
\begin{align}\label{2.9.5-4}
&\int_{\Omega_1^{c} } \big(\sum_{i=2}^{m}U_{x_{i},\lambda}\big)^{2_{s}^{*}-2}
|Z_{1,k}|
\sum_{j=1}^{m}\frac{\lambda^{\frac{N-2s}{2}}}{(1+\lambda|y-x_{j}|)^{\frac{N-2s}{2}+\tau}} \nonumber\\
&\leq C\lambda^{n_{k}}\int_{\Omega_1^{c} }\big(\sum_{i=2}^{m}\frac{1}{(1+\lambda|y-x_{i}|)^{N-2s}}\big)^{2_{s}^{*}-2}
\frac{1}{(1+\lambda|y-x_{1}|)^{N-2s}}
\sum_{j=1}^{m}\frac{\lambda^{N}}{(1+\lambda|y-x_{j}|)^{\frac{N-2s}{2}+\tau}} \nonumber\\
&\leq C\lambda^{n_{k}}\int_{\Omega_1^{c} }\frac{1}{(1+\lambda|y-x_{2}|)^{\frac{N+6s}{2}+\tau}}
\frac{\lambda^{N}}{(1+\lambda|y-x_{1}|)^{N-2s}} \nonumber\\
&\quad+C\lambda^{n_{k}}\int_{\Omega_1^{c} }\frac{1}{(1+\lambda|y-x_{2}|)^{4s}}\frac{1}{(1+\lambda|y-x_{1}|)^{N-2s}}
\sum_{j=1,j\neq 2}^{m}\frac{\lambda^{N}}{(1+\lambda|y-x_{j}|)^{\frac{N-2s}{2}+\tau}} \nonumber\\
&\quad+C\lambda^{n_{k}}\int_{\Omega_1^{c} }\big(\sum_{i=3}^{m}\frac{1}{(1+\lambda|y-x_{i}|)^{N-2s}}\big)^{2_{s}^{*}-2}
\frac{1}{(1+\lambda|y-x_{1}|)^{N-2s}}
\frac{\lambda^{N}}{(1+\lambda|y-x_{2}|)^{\frac{N-2s}{2}+\tau}}\nonumber \\
&\quad+C\lambda^{n_{k}}\int_{\Omega_1^{c} }\big(\sum_{i=3}^{m}\frac{1}{(1+\lambda|y-x_{i}|)^{N-2s}}\big)^{2_{s}^{*}-2}
\frac{1}{(1+\lambda|y-x_{1}|)^{N-2s}}
\sum_{j=1,j\neq 2}^{m}\frac{\lambda^{N}}{(1+\lambda|y-x_{j}|)^{\frac{N-2s}{2}+\tau}} \nonumber
\\
&\leq C\lambda^{n_{k}}\int_{\Omega_1^{c} }\frac{1}{|\lambda(x_{1}-x_{2})|^{\frac{N-2s}{2}+\tau}}
\Big(\frac{\lambda^{N}}{(1+\lambda|y-x_{1}|)^{N+2s}}+\frac{\lambda^{N}}{(1+\lambda|y-x_{2}|)^{N+2s}}\Big) \nonumber
\\
&\quad+C\lambda^{n_{k}}\int_{\Omega_1^{c} }\sum_{j=1,j\neq 2}^{m}\frac{1}{|\lambda(x_{j}-x_{2})|^{\frac{N-2s}{2}+\tau}}
\Big(\frac{\lambda^{N}}{(1+\lambda|y-x_{j}|)^{N+2s}}+\frac{\lambda^{N}}{(1+\lambda|y-x_{2}|)^{N+2s}}\Big) \nonumber
\\
&\quad+C\lambda^{n_{k}}\int_{\Omega_1^{c} }\sum_{i=3}^{m}\frac{1}{|\lambda(x_{i}-x_{2})|^{\frac{N-2s}{2}+\tau}}
\Big(\frac{\lambda^{N}}{(1+\lambda|y-x_{i}|)^{N+2s-\tilde{\sigma}}}
+\frac{\lambda^{N}}{(1+\lambda|y-x_{2}|)^{N+2s-\tilde{\sigma}}}\Big) \nonumber
\\
&\quad+C\lambda^{n_{k}}\int_{\Omega_1^{c} }\sum_{i=3}^{m}\frac{1}{|\lambda(x_{i}-x_{2})|^{\frac{N-2s}{2}+\tau}}\sum_{j=1,j\neq 2}^{m}\frac{1}{|\lambda(x_{j}-x_{2})|^{\tau}}
\Big(\frac{\lambda^{N}}{(1+\lambda|y-x_{i}|)^{N+2s-\tau-\tilde{\sigma}}} \nonumber
\\
&\quad\quad\quad+\frac{\lambda^{N}}{(1+\lambda|y-x_{j}|)^{N+2s-\tau-\tilde{\sigma}}}
+\frac{\lambda^{N}}{(1+\lambda|y-x_{2}|)^{N+2s-\tau-\tilde{\sigma}}}\Big) \nonumber
\\
&\leq C\lambda^{n_{k}}\Big(\frac{m}{\lambda}\Big)^{\frac{N-2s}{2}+\tau}
\leq C\frac{\lambda^{n_{k}}}{\lambda^{1+\frac{2}{N-2s}\tau}}\leq  C\frac{\lambda^{n_{k}}}{\lambda^{1+\iota}},
\end{align}
where $\tilde{\sigma}$ is a small positive constant. Similar to \eqref{2.9.5-4}, we can prove
\begin{equation}\label{2.9.5-5}
\begin{split}
\int_{\Omega_1^{c} } U_{x_{1},\lambda}^{2_{s}^{*}-3}\sum_{i=2}^{m}U_{x_{i},\lambda}|Z_{1,k}|
\sum_{j=1}^{m}\frac{\lambda^{\frac{N-2s}{2}}}{(1+\lambda|y-x_{j}|)^{\frac{N-2s}{2}+\tau}}
\leq  C\frac{\lambda^{n_{k}}}{\lambda^{1+\iota}}.
\end{split}
\end{equation}
From \eqref{2.9.5-3} to \eqref{2.9.5-5}, when $2_{s}^{*}>3,$ we have
\begin{equation}\label{2.9.5-6}
\begin{split}
|M_{1}| \leq  C\|\varphi_{\epsilon}\|_{*}\frac{\lambda^{n_{k}}}{\lambda^{1+\iota}}.
\end{split}
\end{equation}

When $2_{s}^{*}\leq 3,$  we can estimate $M_{1}$ similarly. First of all, we have
\begin{align}\label{2.9.5-5-1}
|M_{11}|
&\leq C\|\varphi_{\epsilon}\|_{*}\lambda^{n_k}\int_{\Omega_{1}} \Big[Z_{\bar r,\bar{y}'',\lambda}^{2_{s}^{*}-2}-U_{x_{1},\lambda}^{2_{s}^{*}-2}\Big]|U_{x_1,\lambda}|
\sum_{j=1}^{m}\frac{\lambda^{\frac{N-2s}{2}}}{(1+\lambda|y-x_{j}|)^{\frac{N-2s}{2}+\tau}} \nonumber
\\
&\leq C\|\varphi_{\epsilon}\|_{*}\lambda^{n_k}\int_{\Omega_{1}} \Big[\Big( U_{x_{1},\lambda}+\sum_{j=2}^{m}U_{x_{j},\lambda} \Big)^{2_{s}^{*}-2}  -U_{x_{1},\lambda}^{2_{s}^{*}-2}\Big]|U_{x_1,\lambda}|
\sum_{j=1}^{m}\frac{\lambda^{\frac{N-2s}{2}}}{(1+\lambda|y-x_{j}|)^{\frac{N-2s}{2}+\tau}} \nonumber
\\
&= C\|\varphi_{\epsilon}\|_{*}\lambda^{n_k}\int_{\Omega_{1}} \Big[U_{x_{1},\lambda}^{2_{s}^{*}-2}\Big( 1+\frac{\sum_{j=2}^{m}U_{x_{j},\lambda}}{U_{x_{1},\lambda}} \Big)^{2_{s}^{*}-2}  -U_{x_{1},\lambda}^{2_{s}^{*}-2}\Big]|U_{x_1,\lambda}|
\sum_{j=1}^{m}\frac{\lambda^{\frac{N-2s}{2}}}{(1+\lambda|y-x_{j}|)^{\frac{N-2s}{2}+\tau}} \nonumber
\\
&\leq C\|\varphi_{\epsilon}\|_{*}\lambda^{n_k}\int_{\Omega_{1}} U_{x_{1},\lambda}^{2_{s}^{*}-3}\sum_{j=2}^{m}U_{x_{j},\lambda}|U_{x_1,\lambda}|
\sum_{j=1}^{m}\frac{\lambda^{\frac{N-2s}{2}}}{(1+\lambda|y-x_{j}|)^{\frac{N-2s}{2}+\tau}} \nonumber
\\
&\leq C\|\varphi_{\epsilon}\|_{*}\lambda^{n_k}\int_{\Omega_{1}} U_{x_{1},\lambda}^{2_{s}^{*}-2}\sum_{j=2}^{m}U_{x_{j},\lambda}
\sum_{j=1}^{m}\frac{\lambda^{\frac{N-2s}{2}}}{(1+\lambda|y-x_{j}|)^{\frac{N-2s}{2}+\tau}} \nonumber
\\
&\leq C\|\varphi_{\epsilon}\|_{*}\lambda^{n_k}\int_{\Omega_{1}} \frac{\lambda^N}{(1+\lambda|y-x_{1}|)^{\frac{N+6s}{2}}}\sum_{j=2}^{m}\frac{1}{(1+\lambda|y-x_j|)^{N-2s}} \nonumber
\\
& \quad + C\|\varphi_{\epsilon}\|_{*}\lambda^{n_k}\int_{\Omega_{1}} \frac{\lambda^N}{(1+\lambda|y-x_{1}|)^{4s}}\sum_{j=2}^{m}\frac{1}{(1+\lambda|y-x_j|)^{N-2s}}
\sum_{j=2}^{m}\frac{1}{(1+\lambda|y-x_{j}|))^{\frac{N-2s}{2}+\tau}} \nonumber
\\
&\leq C\|\varphi_{\epsilon}\|_{*}\lambda^{n_k}\int_{\Omega_{1}} \sum_{j=2}^{m} \frac{1}{(\lambda|x_{1}-x_{j}|)^{\frac{N-2s}{2}}}\Big[\frac{\lambda^N}{(1+\lambda|y-x_{1}|)^{N+2s}}+\frac{\lambda^N}{(1+\lambda|y-x_{i}|)^{N+2s}}\Big] \nonumber
\\
& \quad + C\|\varphi_{\epsilon}\|_{*}\lambda^{n_k}\int_{\Omega_{1}} \frac{\lambda^N}{(1+\lambda|y-x_{1}|)^{\frac{N-2s}{2}+4s}}\sum_{j=2}^{m}\frac{1}{(1+\lambda|y-x_j|)^{N-2s}}
\sum_{j=2}^{m}\frac{1}{(\lambda|x_1-x_{j}|)^{\tau}} \nonumber
\\
&\leq C\|\varphi_{\epsilon}\|_{*}\lambda^{n_{k}}\Big(  \frac{m}{\lambda} \Big)^{\frac{N-2s}{2}}
\leq  C\frac{\lambda^{n_{k}}}{\lambda^{1+\iota}} \|\varphi_{\epsilon}\|_{*}.
\end{align}
Noting that $\tau>2s$ and $|y-x_{1}|\geq |y-x_{i}|,y\in \Omega^{c} _{1},\forall i\neq 1,$  we also have
\begin{align}\label{2.9.5-5-2}
&|M_{12}| \nonumber
\\
&\leq C\|\varphi_{\epsilon}\|_{*}\lambda^{n_k}\int_{\Omega^{c} _{1}} \sum_{i=2}^{m}U^{2^{*}_{s}-2}_{x_i,\lambda} U_{x_1,\lambda}
\sum_{j=1}^{m}\frac{\lambda^{\frac{N-2s}{2}}}{(1+\lambda|y-x_{j}|)^{\frac{N-2s}{2}+\tau}} \nonumber
\\
&\leq C\|\varphi_{\epsilon}\|_{*}\lambda^{n_k}\Big[\int_{\Omega^{c} _{1}}
\sum_{i=2}^{m}\frac{\lambda^{N}}{(1+\lambda|y-x_{i}|)^{4s}}\frac{1}{(1+\lambda|y-x_{1}|)^{\frac{3(N-2s)}{2}+\tau}}\nonumber\\
&\quad+\int_{\Omega^{c} _{1}}
\sum_{i=2}^{m}\frac{\lambda^{N}}{(1+\lambda|y-x_{i}|)^{4s}}\frac{1}{(1+\lambda|y-x_{1}|)^{N-2s}}
\sum_{j=2}^{m}\frac{1}{(1+\lambda|y-x_{j}|)^{\frac{N-2s}{2}+\tau}}\nonumber
\\
&\leq C\|\varphi_{\epsilon}\|_{*}\lambda^{n_k}
\int_{\Omega^{c} _{1}}
\Big[\sum_{i=2}^{m}\frac{1}{|\lambda(x_{i}-x_{1})|^{\frac{N-2s}{2}+\tau}}\frac{\lambda^{N}}{(1+\lambda|y-x_{j}|)^{N+2s}} \nonumber\\
&\quad+
\sum_{i=2}^{m}\frac{1}{|\lambda(x_{i}-x_{j})|^{\frac{N-2s}{2}+\tau}}
\sum_{j=2}^{m}\frac{1}{|\lambda(x_{j}-x_{1})|^{\tau}}
\big(\frac{\lambda^{N}}{(1+\lambda|y-x_{i}|)^{N+2s-\tau}}+\frac{\lambda^{N}}{(1+\lambda|y-x_{j}|)^{N+2s-\tau}}\big)\Big]
\nonumber\\
&\leq C\|\varphi_{\epsilon}\|_{*}\frac{\lambda^{n_k}}{\lambda^{1+\iota}}.
\end{align}

From \eqref{2.9.5-5-1} and \eqref{2.9.5-5-2}, when $2_{s}^{*}\leq 3,$ we have
\begin{equation}
\begin{split}
|M_{1}| \leq  C\|\varphi_{\epsilon}\|_{*}\frac{\lambda^{n_{k}}}{\lambda^{1+\iota}}.
\end{split}
\end{equation}

Similar to \eqref{16-7-2-1} and \eqref{16-7-2}, by the discrete H\"{o}lder inequality we have
\begin{align}\label{2.9.5-7}
|M_{2}| &\leq
C\|\varphi_{\epsilon}\|_{*}\lambda^{n_k}\int_{\R^N} |K(y)-1|\frac{\lambda^{\frac{N-2s}{2}}}{(1+\lambda|y-x_{1}|)^{N-2s}} \nonumber
\\
& \qquad \qquad \qquad \quad \times \Big( \sum_{j=1}^{m}\frac{\lambda^{\frac{N-2s}{2}}}{(1+\lambda|y-x_{j}|)^{N-2s}} \Big)^{2^*-2}
\sum_{j=1}^{m}\frac{\lambda^{\frac{N-2s}{2}}}{(1+\lambda|y-x_{j}|)^{\frac{N-2s}{2}+\tau}}  \nonumber
\\
 &\leq C\|\varphi_{\epsilon}\|_{*}\lambda^{n_k}\int_{\R^N} |K(y)-1|\frac{\lambda^{\frac{N-2s}{2}}}{(1+\lambda|y-x_{1}|)^{N-2s}}\Big( \sum_{j=1}^{m}\frac{\lambda^{\frac{N-2s}{2}}}{(1+\lambda|y-x_{j}|)^{\frac{N-2s}{2}+\tau}} \Big)^{2^*-1}
 \nonumber
 \\
  &\leq C\|\varphi_{\epsilon}\|_{*}\lambda^{n_k}\int_{\R^N} |K(y)-1|\frac{\lambda^{\frac{N-2s}{2}}}{(1+\lambda|y-x_{1}|)^{N-2s}}\sum_{j=1}^{m}\frac{\lambda^{\frac{N+2s}{2}}}{(1+\lambda|y-x_{j}|)^{\frac{N+2s}{2}+\tau}}
 \nonumber
\\
&\leq
C\|\varphi_{\epsilon}\|_{*}\lambda^{n_k}\int_{\check{B}} |K(y)-1|\frac{\lambda^{\frac{N-2s}{2}}}{(1+\lambda|y-x_{1}|)^{N-2s}}\sum_{j=1}^{m}\frac{\lambda^{\frac{N+2s}{2}}}{(1+\lambda|y-x_{j}|)^{\frac{N+2s}{2}+\tau}} \,dy \nonumber
\\
&\quad+C\|\varphi_{\epsilon}\|_{*}\lambda^{n_k}\int_{(\check{B})^{c} } \frac{\lambda^{\frac{N-2s}{2}}}{(1+\lambda|y-x_{1}|)^{N-2s}}\sum_{j=1}^{m}\frac{\lambda^{\frac{N+2s}{2}}}{(1+\lambda|y-x_{j}|)^{\frac{N+2s}{2}+\tau}}\nonumber
\\
&\leq
C\|\varphi_{\epsilon}\|_{*}\frac{\lambda^{n_{k}}}{\lambda^{1+\iota}}\int_{\check{B}}
\frac{\lambda^{\frac{N-2s}{2}}}{(1+\lambda|y-x_{1}|)^{N-2s}}\sum_{j=1}^{m}\frac{\lambda^{\frac{N+2s}{2}}}{(1+\lambda|y-x_{j}|)^{\frac{N+2s}{2}+\tau}}  \nonumber
\\
&\quad+C\|\varphi_{\epsilon}\|_{*}\lambda^{n_k}\int_{(\check{B})^{c} } \frac{\lambda^{N}}{(1+\lambda|y-x_{1}|)^{\frac{N+2s}{2}+\tau+N-2s}}  \nonumber
\\
&\quad+C\|\varphi_{\epsilon}\|_{*}\lambda^{n_k}\int_{(\check{B})^{c} } \frac{\lambda^{N}}{(1+\lambda|y-x_{1}|)^{\frac{N+2s}{2}+\tau+N-2s}} \nonumber
\\
& \quad+ C\|\varphi_{\epsilon}\|_{*}\lambda^{n_k}\int_{(\check{B})^{c} } \frac{\lambda^{\frac{N-2s}{2}}}{(1+\lambda|y-x_{1}|)^{N-2s}}\sum_{j=2}^{m}\frac{\lambda^{\frac{N+2s}{2}}}{(1+\lambda|y-x_{j}|)^{\frac{N+2s}{2}+\tau}}\nonumber
\\
&\leq C\frac{\lambda^{n_{k}}}{\lambda^{1+\iota}}\|\varphi_{\epsilon}\|_{*},
\end{align}
where $\check{B}$ is the same as that of Lemma \ref{lem2.5}.

Similar to \eqref{j11}, by direct computations we have
\begin{align}\label{2.9.5-8}
|M_{3}| &\leq
C\epsilon\|\varphi_{\epsilon}\|_{*}\int_{\R^N} Z_{\bar r,\bar{y}'',\lambda }^{2^{*}_{s}-2+\kappa\epsilon}|\ln Z_{\bar r,\bar{y}'',\lambda } | |Z_{1,k}|
\sum_{j=1}^{m}\frac{\lambda^{\frac{N-2s}{2}}}{(1+\lambda|y-x_{j}|)^{\frac{N-2s}{2}+\tau}} \,dy \nonumber\\
&\leq
 C\epsilon \ln\frac{1}{\epsilon}\lambda^{n_{k}}\|\varphi_{\epsilon}\|_{*}
 \int_{\R^N}\Big( \sum_{i=1}^{m}\frac{\lambda^{\frac{N-2s}{2}}}{(1+\lambda|y-x_{i}|)^{N-2s}} \Big)^{2^{*}_{s}-2+\kappa\epsilon} \nonumber
 \\
 & \qquad \qquad \qquad \qquad \qquad  \times \frac{\lambda^{\frac{N-2s}{2}}}{(1+\lambda|y-x_{1}|)^{N-2s}}
\sum_{j=1}^{m}\frac{\lambda^{\frac{N-2s}{2}}}{(1+\lambda|y-x_{j}|)^{\frac{N-2s}{2}+\tau}} \nonumber
\\
&\leq C\frac{\lambda^{n_{k}}}{\lambda^{1+\iota}}\|\varphi_{\epsilon}\|_{*}
\int_{\R^N}\frac{\lambda^{\frac{N-2s}{2}}}{(1+\lambda|y-x_{1}|)^{N-2s}} \Big( \sum_{j=1}^{m}\frac{\lambda^{\frac{N-2s}{2}}}{(1+\lambda|y-x_{j}|)^{\frac{N-2s}{2}+\tau}} \Big)^{2^{*}_{s}-1+\kappa\epsilon}  \nonumber
\\
&\leq C\frac{\lambda^{n_{k}}}{\lambda^{1+\iota}}\|\varphi_{\epsilon}\|_{*}
\int_{\R^N}\frac{\lambda^{\frac{N-2s}{2}}}{(1+\lambda|y-x_{1}|)^{N-2s}}  \sum_{j=1}^{m}\frac{\lambda^{\frac{N+2s}{2}}}{(1+\lambda|y-x_{j}|)^{\frac{N+2s}{2}+\tau}}  \nonumber
\\
&\leq C\frac{\lambda^{n_{k}}}{\lambda^{1+\iota}}\|\varphi_{\epsilon}\|_{*}
,
\end{align}
where $0<\kappa<1$.

Finally, similar to \eqref{2.9.5-4} we estimate $M_{4}$ as follows
\begin{align}\label{2.9.5-9}
|M_{4}| &\leq
C\epsilon\|\varphi_{\epsilon}\|_{*}\int_{\R^N} Z_{\bar r,\bar{y}'',\lambda }^{2^{*}_{s}-2+\epsilon} |Z_{1,k}|
\sum_{j=1}^{m}\frac{\lambda^{\frac{N-2s}{2}}}{(1+\lambda|y-x_{j}|)^{\frac{N-2s}{2}+\tau}} \,dy\nonumber\\
&\leq C\epsilon\|\varphi_{\epsilon}\|_{*}\lambda^{n_k}\int_{\R^N}\Big( \sum_{i=1}^{m}\frac{\lambda^{\frac{N-2s}{2}}}{(1+\lambda|y-x_{i}|)^{N-2s}} \Big)^{2^{*}_{s}-2+\epsilon} \nonumber
 \\
 & \qquad \qquad \qquad  \times \frac{\lambda^{\frac{N-2s}{2}}}{(1+\lambda|y-x_{1}|)^{N-2s}}
\sum_{j=1}^{m}\frac{\lambda^{\frac{N-2s}{2}}}{(1+\lambda|y-x_{j}|)^{\frac{N-2s}{2}+\tau}}\nonumber
\\
&\leq C\epsilon \lambda^{n_{k}}\|\varphi_{\epsilon}\|_{*} \int_{\R^N}\frac{\lambda^{\frac{N-2s}{2}}}{(1+\lambda|y-x_{1}|)^{N-2s}} \Big( \sum_{j=1}^{m}\frac{\lambda^{\frac{N-2s}{2}}}{(1+\lambda|y-x_{j}|)^{\frac{N-2s}{2}+\tau}} \Big)^{2^{*}_{s}-1+\epsilon} \nonumber
\\
&\leq C\epsilon \lambda^{n_{k}}\|\varphi_{\epsilon}\|_{*} \int_{\R^N}\frac{\lambda^{\frac{N-2s}{2}}}{(1+\lambda|y-x_{1}|)^{N-2s}}  \sum_{j=1}^{m}\frac{\lambda^{\frac{N+2s}{2}}}{(1+\lambda|y-x_{j}|)^{\frac{N+2s}{2}+\tau}}   \nonumber
\\
&\leq C\frac{\lambda^{n_{k}}}{\lambda^{1+\iota}}\|\varphi_{\epsilon}\|_{*}
\int_{\R^N}\frac{\lambda^{\frac{N-2s}{2}}}{(1+\lambda|y-x_{1}|)^{N-2s}}  \sum_{j=1}^{m}\frac{\lambda^{\frac{N+2s}{2}}}{(1+\lambda|y-x_{j}|)^{\frac{N+2s}{2}+\tau}} \nonumber
\\
&\leq C\frac{\lambda^{n_{k}}}{\lambda^{1+\iota}}\|\varphi_{\epsilon}\|_{*}.
\end{align}
It follows from \eqref{2.9.5-1} to \eqref{2.9.5-9} that \eqref{2.9.5} holds.
\end{proof}

\section{{Proofs of \eqref{14-2-1-1} and \eqref{14-2-2-2}}}\label{sc}

Now first we prove \eqref{14-2-1-1}.

\begin{proof}
Observe that
\begin{align}\label{14-2-1-1-1}
&\int_{B^{c}_{\rho}(y_0)}\sum_{l\,=\,1}^{N}c_{l}\ds\sum_{j\,=\,1}^{m}\,U_{x_{j},\lambda}^{2^{*}_{s}-2}\,Z_{j,l} \, \langle y, \nabla u_\epsilon\bigr\rangle dy \nonumber
\\
=&\int_{B^{c}_{\rho}(y_0)}\sum_{l\,=\,1}^{N}c_{l}\ds\sum_{j\,=\,1}^{m}\,U_{x_{j},\lambda}^{2^{*}_{s}-2}\,Z_{j,l} \, \sum_{i=1}^{N}y_{i}\frac{\partial u_\epsilon}{\partial y_{i}}dy \nonumber
\\
=&\int_{B^{c}_{\rho}(y_0)}\sum_{l\,=\,1}^{N}c_{l}\ds\sum_{j\,=\,1}^{m}\,U_{x_{j},\lambda}^{2^{*}_{s}-2}\,Z_{j,l} \, \sum_{i=1}^{N}y_{i}\big(\frac{\partial Z_{\bar{r},\bar{y}'',\lambda}}{\partial y_{i}}+\frac{\partial \varphi}{\partial y_{i}}\big)dy\nonumber
\\
:=&H_{1}+H_{2}.
\end{align}
Next we estimate $H_{1}$ and $H_{2}$ respectively.

First, we will give the estimates of $H_{1}$.
\begin{equation}\label{14-2-1-1-2}
\begin{split}
H_{1}=&\int_{B^{c}_{\rho}(y_0)}\sum_{l\,=\,1}^{N}c_{l}\ds\sum_{j\,=\,1}^{m}\,U_{x_{j},\lambda}^{2^{*}_{s}-2}\,Z_{j,l} \, \sum_{i=1}^{N}\sum_{k=1}^{m}\big[(y_{i}-x_{k,i})\frac{\partial U_{x_{k},\lambda}}{\partial y_{i}}+x_{k,i}\frac{\partial U_{x_{k},\lambda}}{\partial y_{i}}\big]dy\\
:=&H_{11}+H_{12}.
\end{split}
\end{equation}
Noting that $B^{c}_{\rho}(y_0)\subset B^{c}_{\frac{\rho}{2}}(x_{j}),$
by direct computations, we have
\begin{equation}\label{14-2-1-1-3}
\begin{split}
|H_{11}|\leq\,& C\int_{B^{c}_{\frac{\rho}{2}}(x_{j})}\sum_{l=2}^{N}|c_{l}|\sum_{j=1}^{m}\lambda U^{2^{*}_{s}-1}_{x_{j},\lambda}
\sum_{k=1}^{m}\frac{\lambda^{\frac{N-2s}{2}}}{(1+\lambda|y-x_{k}|)^{N-2s}}dy
\\
&\quad+C\int_{B^{c}_{\frac{\rho}{2}}(x_{j})}|c_{1}|\sum_{j=1}^{m}\frac{U^{2^{*}_{s}-1}_{x_{j},\lambda}}{\lambda}
\sum_{k=1}^{m}\frac{\lambda^{\frac{N-2s}{2}}}{(1+\lambda|y-x_{k}|)^{N-2s}}dy
\\
:=\,&C(H_{111}+H_{112}).
\end{split}
\end{equation}
Using Lemma \ref{lemb1}, it is easy to obtain that
\begin{align}\label{14-2-1-1-4}
|H_{111}|&\leq C\int_{B^{c}_{\frac{\rho}{2}}(x_{j})}\sum_{l=2}^{N}|c_{l}|\sum_{j=1}^{m}\lambda U^{2^{*}_{s}-1}_{x_{j},\lambda}
\sum_{k=1}^{m}\frac{\lambda^{\frac{N-2s}{2}}}{(1+\lambda|y-x_{k}|)^{N-2s}}dy \nonumber
\\
&\leq C\sum_{l=2}^{N}|c_{l}|m\int_{B^{c}_{\frac{\rho}{2}}(x_{j})}\Big[\frac{\lambda^{N+1}}{(1+\lambda|y-x_{j}|)^{2N}}\nonumber
\\
&\quad\quad+\sum_{k\neq j}\frac{\lambda}{(\lambda|x_{k}-x_{j}|)^{N-2s}}\big(\frac{\lambda^{N}}{(1+\lambda|y-x_{j}|)^{N+2s}}+\frac{\lambda^{N}}{(1+\lambda|y-x_{k}|)^{N+2s}}\big)\Big]dy \nonumber
\\&\leq C\sum_{l=2}^{N}|c_{l}|\frac{m}{\lambda^{N-1}}
+C\sum_{l=2}^{N}|c_{l}|\frac{m}{\lambda}=o(m\lambda^{2})\sum_{l=2}^{N}|c_{l}|,
     \end{align}
and
\begin{equation}\label{14-2-1-1-5}
\begin{split}
|H_{112}|&\leq C|c_{1}|m\int_{B^{c}_{\frac{\rho}{2}}(x_{j})}\Big[\frac{\lambda^{N-1}}{(1+\lambda|y-x_{j}|)^{2N}}
\\
&\quad\quad+\sum_{k\neq j}\frac{\lambda^{-1}}{(\lambda|x_{k}-x_{j}|)^{N-2s}}\big(\frac{\lambda^{N}}{(1+\lambda|y-x_{j}|)^{N+2s}}+\frac{\lambda^{N}}{(1+\lambda|y-x_{k}|)^{N+2s}}\big)\Big]dy
\\&\leq C|c_{1}|\frac{m}{\lambda^{N+1}}
+C|c_{1}|\frac{m}{\lambda^{3}}=o(m)|c_{1}|.
\end{split}
\end{equation}

Similarly, we have
\begin{equation}\label{14-2-1-1-6}
\begin{split}
|H_{12}|\leq& C\int_{B^{c}_{\frac{\rho}{2}}(x_{j})}\sum_{l=2}^{N}|c_{l}|\sum_{j=1}^{m}\lambda U^{2^{*}_{s}-1}_{x_{j},\lambda}
\sum_{k=1}^{m}\frac{\lambda^{\frac{N-2s}{2}+1}}{(1+\lambda|y-x_{k}|)^{N-2s+1}}dy
\\
&\quad+C\int_{B^{c}_{\frac{\rho}{2}}(x_{j})}|c_{1}|\sum_{j=1}^{m} \frac{U^{2^{*}_{s}-1}_{x_{j},\lambda}}{\lambda}
\sum_{k=1}^{m}\frac{\lambda^{\frac{N-2s}{2}+1}}{(1+\lambda|y-x_{k}|)^{N-2s+1}}dy
\\
:=&C(H_{121}+H_{122}),
\end{split}
\end{equation}

\begin{align}\label{estimateofH121}
|H_{121}|
&\leq C\sum_{l=2}^{N}|c_{l}|m\int_{B^{c}_{\frac{\rho}{2}}(x_{j})}\Big[\frac{\lambda^{N+2}}{(1+\lambda|y-x_{j}|)^{2N+1}} \nonumber
\\
&\quad\quad+\sum_{k\neq j}\frac{\lambda^{2}}{(\lambda|x_{k}-x_{j}|)^{N-2s}}\big(\frac{\lambda^{N}}{(1+\lambda|y-x_{j}|)^{N+1+2s}}
+\frac{\lambda^{N}}{(1+\lambda|y-x_{k}|)^{N+1+2s}}\big)\Big] dy\nonumber
\\&\leq C\sum_{l=2}^{N}|c_{l}|\frac{m}{\lambda^{N-1}}
+C\sum_{l=2}^{N}|c_{l}|\frac{m\lambda^{2}}{\lambda^{2}}=o(m\lambda^{2})\sum_{l=2}^{N}|c_{l}|,
\end{align}
and
\begin{align}\label{14-2-1-1-7}
|H_{122}|&\leq C|c_{1}|m\int_{B^{c}_{\frac{\rho}{2}}(x_{j})}\Big[\frac{\lambda^{N}}{(1+\lambda|y-x_{j}|)^{2N+1}} \nonumber
\\
&\quad\quad+\sum_{k\neq j}\frac{1}{(\lambda|x_{k}-x_{j}|)^{N-2s}}\big(\frac{\lambda^{N}}{(1+\lambda|y-x_{j}|)^{N+1+2s}}
+\frac{\lambda^{N}}{(1+\lambda|y-x_{k}|)^{N+1+2s}}\big)\Big] dy \nonumber
\\&\leq C|c_{1}|\frac{m}{\lambda^{N+1}}
+C|c_{1}|\frac{m}{\lambda^{2}}=o(m)|c_{1}|.
\end{align}

From \eqref{14-2-1-1-2} to \eqref{14-2-1-1-7}, we have
\begin{equation}\label{14-2-1-1-8}
\begin{split}
|H_{1}|=o(m\lambda^{2})\sum_{l=2}^{N}|c_{l}|+o(m)|c_{1}|.
\end{split}
\end{equation}

Noting that $\partial B^{c}_{\rho}(y_0)\subset B^{c}_{\frac{\rho}{2}}(x_{j})$ and $B^{c}_{\rho}(y_0)\subset B^{c}_{\frac{\rho}{2}}(x_{j}),$
 applying integrating by parts, by Proposition \ref{prop2.3} and Lemma \ref{lemb1} we have

\begin{align}\label{estimateof H2}
H_{2}
=\,&\int_{\partial B^{c}_{\rho}(y_0)}\sum_{l\,=\,1}^{N}c_{l}\ds\sum_{j\,=\,1}^{m}\,U_{x_{j},\lambda}^{2^{*}_{s}-2}\,Z_{j,l} \, \varphi y\cdot \nu ds \nonumber
\\
&\quad-\int_{B^{c}_{\rho}(y_0)}\sum_{l\,=\,1}^{N}c_{l}\ds\sum_{j\,=\,1}^{m}\,\sum_{i=1}^{N}\frac{\partial}{\partial y_{i}}\big(U_{x_{j},\lambda}^{2^{*}_{s}-2}\,Z_{j,l} y_{i}\big)\, \varphi dy \nonumber
\\
=\,&\int_{\partial B^{c}_{\rho}(y_0)}\sum_{l\,=\,1}^{N}c_{l}\ds\sum_{j\,=\,1}^{m}\,U_{x_{j},\lambda}^{2^{*}_{s}-2}\,Z_{j,l} \, \varphi y\cdot \nu ds \nonumber
\\
&\quad
-\int_{B^{c}_{\rho}(y_0)}\sum_{l\,=\,1}^{N}c_{l}\ds\sum_{j\,=\,1}^{m}\sum_{i=1}^{N}\,\big[\frac{\partial U_{x_{j},\lambda}^{2^{*}_{s}-2}}{\partial y_{i}}\,Z_{j,l}y_{i}\, +\frac{\partial Z_{j,l}}{\partial y_{i}}\,U_{x_{j},\lambda}^{2^{*}_{s}-2} y_{i}
+\frac{d y_{i}}{d y_{i}}U_{x_{j},\lambda}^{2^{*}_{s}-2}\,Z_{j,l}\big]\varphi dy \nonumber
\\
:=\,&H_{21}-H_{22}-H_{23}-H_{24}.
\end{align}
In the following, we will estimate the terms one by one.

For $H_{21}$, we have
\begin{align}\label{14-2-1-1-9}
|H_{21}|
&\leq C|c_{1}|\|\varphi\|_{*}\int_{B^{c}_{\frac{\rho}{2}}(x_{j})}\sum_{j=1}^{m}\frac{\lambda^{N-1}}{(1+\lambda|y-x_{j}|)^{N+2s}}
\sum_{k=1}^{m}\frac{1}{(1+\lambda|y-x_{k}|)^{\frac{N-2s}{2}+\tau}} \nonumber
\\
&\quad+C\sum_{l=2}^{N}|c_{l}|\|\varphi\|_{*}\int_{B^{c}_{\frac{\rho}{2}}(x_{j})}\sum_{j=1}^{m}\frac{\lambda^{N+1}}{(1+\lambda|y-x_{j}|)^{N+2s}}
\sum_{k=1}^{m}\frac{1}{(1+\lambda|y-x_{k}|)^{\frac{N-2s}{2}+\tau}}  \nonumber
\\
&\leq
C|c_{1}|\|\varphi\|_{*}m\int_{B^{c}_{\frac{\rho}{2}}(x_{j})}\Big[\frac{\lambda^{N-1}}{(1+\lambda|y-x_{j}|)^{\frac{3N+2s}{2}+\tau}}  \nonumber
\\
&\quad\quad+\sum_{k\neq j}\frac{\lambda^{-1}}{(\lambda|x_{k}-x_{j}|)^{\frac{N-2s}{2}+\tau}}
\big(\frac{\lambda^{N}}{(1+\lambda|y-x_{j}|)^{N+2s}}+\frac{\lambda^{N}}{(1+\lambda|y-x_{k}|)^{N+2s}}\big)\Big]  \nonumber
\\
&\quad+
 C\sum_{l=2}^{N}|c_{l}|\|\varphi\|_{*}m\int_{B^{c}_{\frac{\rho}{2}}(x_{j})}\Big[\frac{\lambda^{N+1}}{(1+\lambda|y-x_{j}|)^{\frac{3N+2s}{2}+\tau}}  \nonumber
\\
&\quad\quad+\sum_{k\neq j}\frac{\lambda}{(\lambda|x_{k}-x_{j}|)^{\frac{N-2s}{2}+\tau}}
\big(\frac{\lambda^{N}}{(1+\lambda|y-x_{j}|)^{N+2s}}+\frac{\lambda^{N}}{(1+\lambda|y-x_{k}|)^{N+2s}}\big)\Big]\nonumber
\\&=o(m\lambda^{2})\sum_{l=2}^{N}|c_{l}|+o(m)|c_{1}|.
\end{align}

From directly computations, we obtain
\begin{align}\label{14-2-1-1-10}
H_{22}&=\int_{B^{c}_{\rho}(y_0)}\sum_{l\,=\,1}^{N}c_{l}\ds\sum_{j\,=\,1}^{m}\sum_{i=1}^{N}\frac{\partial U_{x_{j},\lambda}^{2^{*}_{s}-2}}{\partial y_{i}}\,Z_{j,l} \,\big[(y_{i}-x_{j,i})+x_{j,i}\,\big]\varphi dy  \nonumber
\\
&\leq C\|\varphi\|_{*}\sum_{l\,=\,1}^{N}|c_{l}|m\int_{B_{\frac{\rho}{2}}(x_{j})}\frac{\lambda^{n_{l}+N}}{(1+\lambda|y-x_{j}|)^{N+2s}}
\sum_{k=1}^{m}\frac{1}{(1+\lambda|y-x_{k}|)^{\frac{N-2s}{2}+\tau}} \nonumber
\\
&\quad+C\|\varphi\|_{*}\sum_{l\,=\,1}^{N}|c_{l}|m\int_{B_{\frac{\rho}{2}}(x_{j})}\frac{\lambda^{n_{l}+1+N}}{(1+\lambda|y-x_{j}|)^{N+2s+1}}
\sum_{k=1}^{m}\frac{1}{(1+\lambda|y-x_{k}|)^{\frac{N-2s}{2}+\tau}} \nonumber
\\
&\leq
C|c_{1}|\|\varphi\|_{*}m\int_{B^{c}_{\frac{\rho}{2}}(x_{j})}\Big[\frac{\lambda^{N-1}}{(1+\lambda|y-x_{j}|)^{\frac{3N+2s}{2}+\tau}} \nonumber
\\
&\quad\quad+\sum_{k\neq j}\frac{\lambda^{-1}}{(\lambda|x_{k}-x_{j}|)^{\frac{N-2s}{2}+\tau}}
\big(\frac{\lambda^{N}}{(1+\lambda|y-x_{j}|)^{N+2s}}+\frac{\lambda^{N}}{(1+\lambda|y-x_{k}|)^{N+2s}}\big)\Big] \nonumber
\\
&\quad+
 C\sum_{l=2}^{N}|c_{l}|\|\varphi\|_{*}m\int_{B^{c}_{\frac{\rho}{2}}(x_{j})}\Big[\frac{\lambda^{N+1}}{(1+\lambda|y-x_{j}|)^{\frac{3N+2s}{2}+\tau}} \nonumber
\\
&\quad\quad+\sum_{k\neq j}\frac{\lambda}{(\lambda|x_{k}-x_{j}|)^{\frac{N-2s}{2}+\tau}}
\big(\frac{\lambda^{N}}{(1+\lambda|y-x_{j}|)^{N+2s}}+\frac{\lambda^{N}}{(1+\lambda|y-x_{k}|)^{N+2s}}\big)\Big] \nonumber
\\
&\quad+
C|c_{1}|\|\varphi\|_{*}m\int_{B^{c}_{\frac{\rho}{2}}(x_{j})}\Big[\frac{\lambda^{N}}{(1+\lambda|y-x_{j}|)^{\frac{3N+2s+2}{2}+\tau}} \nonumber
\\
&\quad\quad+\sum_{k\neq j}\frac{1}{(\lambda|x_{k}-x_{j}|)^{\frac{N-2s}{2}+\tau}}
\big(\frac{\lambda^{N}}{(1+\lambda|y-x_{j}|)^{N+2s}}+\frac{\lambda^{N}}{(1+\lambda|y-x_{k}|)^{N+2s+1}}\big)\Big] \nonumber
\\
&\quad+
 C\sum_{l=2}^{N}|c_{l}|\|\varphi\|_{*}m\int_{B^{c}_{\frac{\rho}{2}}(x_{j})}\Big[\frac{\lambda^{N+2}}{(1+\lambda|y-x_{j}|)^{\frac{3N+2s+2}{2}+\tau}} \nonumber
\\
&\quad\quad+\sum_{k\neq j}\frac{\lambda^{2}}{(\lambda|x_{k}-x_{j}|)^{\frac{N-2s}{2}+\tau}}
\big(\frac{\lambda^{N}}{(1+\lambda|y-x_{j}|)^{N+2s+1}}+\frac{\lambda^{N}}{(1+\lambda|y-x_{k}|)^{N+2s+1}}\big)\Big] \nonumber
\\
&=o(m\lambda^{2})\sum_{l=2}^{N}|c_{l}|+o(m)|c_{1}|.
\end{align}

Similar to \eqref{14-2-1-1-10}, we can also check that
\begin{equation}\label{14-2-1-1-11}
\begin{split}
H_{23}=o(m\lambda^{2})\sum_{l=2}^{N}|c_{l}|+o(m)|c_{1}|.
\end{split}
\end{equation}
Finally, we estimate $H_{24}$ as follows
\begin{equation}\label{14-2-1-1-12}
\begin{split}
H_{24}&\leq  C\|\varphi\|_{*}\sum_{l\,=\,1}^{N}|c_{l}|m\int_{B_{\frac{\rho}{2}}(x_{j})}\frac{\lambda^{n_{l}+N}}{(1+\lambda|y-x_{j}|)^{N+2s}}
\sum_{k=1}^{m}\frac{1}{(1+\lambda|y-x_{k}|)^{\frac{N-2s}{2}+\tau}}\\
&=o(m\lambda^{2})\sum_{l=2}^{N}|c_{l}|+o(m)|c_{1}|.
\end{split}
\end{equation}
It follows from all the estimates above that \eqref{14-2-1-1} holds.
\end{proof}

Then we prove \eqref{14-2-2-2}.

\begin{proof}
Note that
\begin{equation}\label{14-2-2-2-1}
\begin{split}
&\int_{B^{c}_{\rho}(y_0)}\sum_{l\,=\,1}^{N}c_{l}\ds\sum_{j\,=\,1}^{m}\,U_{x_{j},\lambda}^{2^{*}_{s}-2}\,Z_{j,l} \, \, \frac{\partial u_\epsilon}{\partial y_i}\\
=&\int_{B^{c}_{\rho}(y_0)}\sum_{l\,=\,1}^{N}c_{l}\ds\sum_{j\,=\,1}^{m}\,U_{x_{j},\lambda}^{2^{*}_{s}-2}\,Z_{j,l} \, \, \Big(\frac{\partial Z_{\bar{r},\bar{y}'',\lambda}}{\partial y_i}+\frac{\partial \varphi}{\partial y_{i}}\Big)\\
:=&G_{1}+G_{2}.
\end{split}
\end{equation}
By direct computations, from Proposition \ref{prop2.3} and Lemma \ref{lemb1} we have
\begin{align}\label{14-2-2-2-2}
|G_{1}|
& \leq C\int_{B^{c}_{\frac{\rho}{2}}(x_{j})}\sum_{l=1}^{N}|c_{l}|\sum_{j=1}^{m}\lambda^{n_{l}} U^{2^{*}_{s}-1}_{x_{j},\lambda}
\Big|\frac{\partial Z_{\bar{r},\bar{y}'',\lambda}}{\partial y_i}\Big| \nonumber
\\
&\leq
 C\int_{B^{c}_{\frac{\rho}{2}}(x_{j})}\sum_{l=1}^{N}|c_{l}|\sum_{j=1}^{m}\lambda^{n_{l}+1} U^{2^{*}_{s}-1}_{x_{j},\lambda}
\sum_{k=1}^{m}\frac{\lambda^{\frac{N-2s}{2}}}{(1+\lambda|y-x_{k}|)^{N-2s+1}} \nonumber
\\
&\leq C|c_{1}|m\int_{B^{c}_{\frac{\rho}{2}}(x_{j})}\Big[\frac{\lambda^{N}}{(1+\lambda|y-x_{j}|)^{2N+1}} \nonumber
\\
&\quad\quad+\sum_{k\neq j}\frac{1}{(\lambda|x_{k}-x_{j}|)^{N-2s}}\big(\frac{\lambda^{N}}{(1+\lambda|y-x_{j}|)^{N+1+2s}}
+\frac{\lambda^{N}}{(1+\lambda|y-x_{k}|)^{N+1+2s}}\big)\Big] \nonumber
\\
&\quad+ C\sum_{l=2}^{N}|c_{l}|m\int_{B^{c}_{\frac{\rho}{2}}(x_{j})}\Big[\frac{\lambda^{N+2}}{(1+\lambda|y-x_{j}|)^{2N+1}} \nonumber
\\
&\quad\quad+\sum_{k\neq j}\frac{\lambda^{2}}{(\lambda|x_{k}-x_{j}|)^{N-2s}}\big(\frac{\lambda^{N}}{(1+\lambda|y-x_{j}|)^{N+1+2s}}
+\frac{\lambda^{N}}{(1+\lambda|y-x_{k}|)^{N+1+2s}}\big)\Big] \nonumber
\\&\leq C|c_{1}|\frac{m}{\lambda^{N}}+C|c_{1}|\frac{m}{\lambda^{2}}+C\sum_{l=2}^{N}|c_{l}|\frac{m}{\lambda^{N-3}}
+Cm\sum_{l=2}^{N}|c_{l}|\nonumber\\
&=o(m)|c_{1}|+o(m\lambda^{2})\sum_{l=2}^{N}|c_{l}|.
\end{align}

Applying integrating by parts, we have
\begin{align}\label{14-2-2-2-3}
G_{2}
=\,&\int_{\partial B^{c}_{\rho}(y_0)}
\sum_{l=1}^{N}|c_{l}|\sum_{j=1}^{m} U^{2^{*}_{s}-2}_{x_{j},\lambda}Z_{j,l}
\varphi \nu_{i} \, ds
-\int_{B_{\rho}^{c}(y_0)}\sum_{l=1}^{N}|c_{l}|\sum_{j=1}^{m}\frac{\partial }{\partial y_{i}} (U^{2^{*}_{s}-2}_{x_{j},\lambda}Z_{j,l})\varphi  \, dy \nonumber
\\
=\,&\int_{\partial B^{c}_{\rho}(y_0)}
\sum_{l=1}^{N}|c_{l}|\sum_{j=1}^{m}\lambda^{n_{l}} U^{2^{*}_{s}-1}_{x_{j},\lambda}\varphi \nu_{i} \, ds \nonumber
\\
&-\int_{\partial'\mathcal{B}_{\rho}^{+}(y_0))^{c}}\sum_{l=1}^{N}|c_{l}|\sum_{j=1}^{m}\big(\frac{\partial U^{2^{*}_{s}-2}_{x_{j},\lambda}}{\partial y_{i}} Z_{j,l}+\frac{\partial Z_{j,l}}{\partial y_{i}}U^{2^{*}_{s}-2}_{x_{j},\lambda}\big)\varphi \, dy \nonumber
\\
:=\,&G_{21}-G_{22}-G_{23}.
\end{align}
Similar to \eqref{14-2-1-1}, we can check that
\begin{equation}\label{14-2-2-2-4}
\begin{split}
|G_{21}|=o(m\lambda^{2})\sum_{l=2}^{N}|c_{l}|+o(m)|c_{1}|.
\end{split}
\end{equation}

Noting that $\partial B^{c}_{\rho}(y_0)\subset B^{c}_{\frac{\rho}{2}}(x_{j})$ and $B^{c}_{\rho}(y_0)\subset B^{c}_{\frac{\rho}{2}}(x_{j}),$
 applying integrating by parts, by Proposition \ref{prop2.3} and Lemma \ref{lemb1} we have
 \begin{align}\label{14-2-2-2-5}
|G_{22}|&\leq
C|c_{1}|\|\varphi\|_{*}m\int_{B^{c}_{\frac{\rho}{2}}(x_{j})}\Big[\frac{\lambda^{N}}{(1+\lambda|y-x_{j}|)^{\frac{3N+2s}{2}+\tau}} \nonumber
\\
&\quad\quad+\sum_{k\neq j}\frac{1}{(\lambda|x_{k}-x_{j}|)^{\frac{N-2s}{2}+\tau}}
\big(\frac{\lambda^{N}}{(1+\lambda|y-x_{j}|)^{N+2s}}+\frac{\lambda^{N}}{(1+\lambda|y-x_{k}|)^{N+2s}}\big)\Big] \nonumber
\\
&\quad+
 C\sum_{l=2}^{N}|c_{l}|\|\varphi\|_{*}m\int_{B^{c}_{\frac{\rho}{2}}(x_{j})}\Big[\frac{\lambda^{N+2}}{(1+\lambda|y-x_{j}|)^{\frac{3N+2s}{2}+\tau}} \nonumber
\\
&\quad\quad+\sum_{k\neq j}\frac{\lambda^{2}}{(\lambda|x_{k}-x_{j}|)^{\frac{N-2s}{2}+\tau}}
\big(\frac{\lambda^{N}}{(1+\lambda|y-x_{j}|)^{N+2s}}+\frac{\lambda^{N}}{(1+\lambda|y-x_{k}|)^{N+2s}}\big)\Big] \nonumber
\\
&=o(m\lambda^{2})\sum_{l=2}^{N}|c_{l}|+o(m)|c_{1}|.
\end{align}
Just by the same argument as \eqref{14-2-2-2-5}, we can prove
\begin{equation}\label{14-2-2-2-6}
\begin{split}
|G_{23}|=o(m\lambda^{2})\sum_{l=2}^{N}|c_{l}|+o(m)|c_{1}|.
\end{split}
\end{equation}
By all the estimates above, we know that \eqref{14-2-2-2} holds.
\end{proof}

\vspace{0.2cm}
\textbf{Acknowledgements}
This paper was partially supported by NSFC (No.11671162;
No.11571130; No. 11601194) and CCNU18CXTD04.

\end{document}